\documentclass[11pt, english]{article}
\usepackage[T1]{fontenc}
\usepackage[latin9]{inputenc}
\usepackage{geometry}
\usepackage{array}
\usepackage{float}
\usepackage{mathrsfs}
\usepackage{amsmath,amsfonts,amssymb,mathtools}
\usepackage{xcolor,tabularx,nccmath}
\usepackage{amsthm,enumitem}
\usepackage{booktabs,mathtools,multirow}
\usepackage{subcaption}
\usepackage{rotating}
\PassOptionsToPackage{normalem}{ulem}
\usepackage{rotating}
\usepackage{makecell}
\usepackage{amsmath}        
\usepackage{amsfonts}       
\usepackage{algorithm}      
\usepackage{algorithmicx}   
\usepackage{algpseudocode}  
\usepackage{comment}
\usepackage{algorithm}
\usepackage{algpseudocode}

\newtheorem{remark}{Remark}

\usepackage[round]{natbib}
\setcitestyle{aysep={},citesep={, }}
\usepackage{hyperref}
\usepackage{todonotes}
\usepackage{setspace}
\usepackage{adjustbox}
\usepackage{xurl}      
\usepackage{csquotes}
\usepackage{authblk}

\newtheorem{prop}{Proposition}


\title{\vspace{-50pt}\textbf{Designing Hierarchical Hub-and-Spoke Drone-Based Networks for Delivering Time-Sensitive Healthcare Items}}

\author[1]{Sayed Hamid Hosseini Dolatabadi}
\author[1,*]{Tanveer Hossain Bhuiyan}
\author[2]{Bo Zeng}
\author[3]{J. Tyler Kennedy}
\author[4]{Brian O' Neill}

\affil[1]{Department of Mechanical, Aerospace, and Industrial Engineering, The University of Texas at San Antonio, TX, USA.}
\affil[2]{Department of Industrial Engineering, University of Pittsburgh, PA, USA.}
\affil[3]{Interpath Laboratory, Inc., Pendleton, OR, USA.}
\affil[4]{San Antonio Fire Department, TX, USA.}

\affil[*]{Corresponding author: \href{mailto:tanveer.bhuiyan@utsa.edu}{tanveer.bhuiyan@utsa.edu};
Contributing authors: \href{mailto:sayedhamid.hosseinidolatabadi@utsa.edu}{sayedhamid.hosseinidolatabadi@utsa.edu}; \href{mailto:bzeng@pitt.edu}{bzeng@pitt.edu}; \href{mailto:TylerKennedy@adaugeohealthcare.com}{TylerKennedy@adaugeohealthcare.com}; \href{mailto:brian.o'neill@sanantonio.gov}{brian.o'neill@sanantonio.gov}
}

\usepackage{babel}

\begin{document}
\setstretch{1.15}
\date{}
\maketitle

\begin{abstract}

Timely distribution of medical items (e.g., whole blood and vaccines) across regional healthcare networks often requires high-volume delivery operations from central facilities (e.g., blood banks) to intermediate regional hospitals, followed by rapid last-mile deliveries to points of injury. Traditional ground-based delivery systems often suffer from limited responsiveness (e.g., traffic congestion) and operational inefficiencies (e.g., blood waste). Leveraging aerial-drone-based delivery systems offers a promising solution for the fast and efficient delivery of time-sensitive medical items across regional healthcare networks. Therefore, we study a hierarchical hub-and-spoke drone-based network for delivering time-sensitive medical items with distinct release and due times to fixed and mobile delivery destinations. We consider a heterogeneous fleet of drones with distinct characteristics (e.g., cost, battery capacity, and speed) and different multi-trip delivery modes. We propose an efficient mixed-integer programming model for location/allocation of mobile delivery destinations, as well as routing and scheduling drones to minimize the total investment and operational costs of the drone delivery network while maintaining the delivery due times. We develop a customized exact solution method integrating problem-specific reformulations and dynamic cutting planes, as well as a fast heuristic algorithm by leveraging a simplified problem variant. Results based on a real-life case study of whole blood delivery data from Pendleton, Oregon, United States, and actual drone flight test data demonstrate that our exact and heuristic solution methods are 31 and 3,375 times faster, respectively, than the Gurobi solver. Results also show that allowing drones to perform multiple trips is 142.8\% more cost-efficient than single trips.
\end{abstract}

\textbf{Keywords: } Heterogeneous drone fleet, Two-echelon vehicle routing, Release and due times, Mixed-integer linear programming, Dynamic cutting plane.

\vspace{-10pt}
\section{Introduction}

\vspace{-5pt}
\subsection{Motivation}

The rapid expansion of medical logistics, the increasing frequency of disasters, and the persistent limitations in healthcare accessibility across rural regions, have revealed critical limitations in conventional delivery systems~\citep{palmer2024unique, koshta2022evaluating, national2022building}. Time-sensitive medical products (e.g., whole blood, blood products, and emergency medications) typically require delivery within strict time windows---defined by release times and patient-dependent due times~\citep{AmericaBlood2022, FDA2025EmergencyDevices, gentili2022locating}. Traditional ground transportation often fails to guarantee such responsiveness, particularly in regions with limited infrastructure, road barriers, or congestion~\citep{varela2019transportation, eisner2024transportation}.
Unmanned aerial vehicles, or aerial drones, have emerged as promising alternatives for last-mile delivery of time-sensitive medical items due to their ability to bypass ground congestion, reach hard-to-access communities, and maintain predictable travel times~\citep{kellermann2020drones, mohsan2023unmanned}. Real-world deployments have demonstrated the effectiveness of drones in critical healthcare logistics~\citep{karsten2018drones}. Leading companies, such as Wingcopter and Zipline, have implemented large-scale medical drone delivery networks across several countries~\citep{wingcopter2020vanuatu, ackerman2019blood, jairoun2025evolution, flyzipline_website}. For instance, Zipline's logistics network in Rwanda delivers more than 35\% of the country's blood supply~\citep{Simons_Rwanda}. Accounting for the critical role of drones in the healthcare logistics, the global medical drone delivery sector was valued at \$245.4 million in 2023, and is projected to reach nearly \$1.9 billion by 2032 with more than 808 million drone deliveries until 2034 \citep{jairoun2025evolution, pwc2024drone}.

Despite these benefits and real-world implications, the limited flight range and battery capacity of drones limit their applicability in wide geographical regions~\citep{hosseini2026branch}. To mitigate this limitation and expand the coverage of drone-based delivery logistics, three main strategies have been proposed: ($i$) synchronized and unsynchronized drone-truck multimodal logistics (e.g., \citealt{murray2015flying, li2021ground}), ($ii$) intermediate battery recharging or swapping stations (e.g., \citealt{hong2018range, cokyasar2021optimization}), and ($iii$) hub-and-spoke (H\&S) drone delivery network designs (e.g., \citealt{gao2022optimizing, hu2024drone}). Although drone-truck logistics enable trucks to reach locations outside the drones' coverage~\citep{mirzapour2024optimizing, moshref2021comparative}, the integration of trucks reduces the autonomy of drones~\citep{pinto2022point, murray2015flying}. Intermediate battery recharging or swapping stations enable drones to make frequent stops to recharge/swap their batteries~\citep{hong2018range, cokyasar2021designing}. Although this offers a viable solution to extend the operational range of drones, placing intermediate stations in rural areas often poses significant logistical complexities, such as limited grid connectivity and the lack of accessible installation sites~\citep{hassan2022charging}. Moreover, frequent stops and long queues at intermediate stations further delay the delivery operations. This time delay can cause perishable medical items (e.g., whole blood, emergency medications, and transplant organs) to be damaged, and may lead to patient death in case of an immediate delivery requirement to a point of injury in mass shootings~\citep{glick2022case}.

In contrast, an H\&S drone delivery network, comprising a central depot, regional hospitals, and clinics, provides an effective solution for time-sensitive healthcare logistics, especially when perishable medical items with limited shelf life need to be transported across wide regional healthcare networks. Medical items often vary in package weight, release time at the distribution center (i.e., central depot), final delivery destination, and delivery due time. A hierarchical (i.e., multi-layer) H\&S structure allows heterogeneous medical items to be first transported in bulk shipments from a central depot to intermediate logistics hubs (e.g., large regional hospitals), using drones with large payloads and battery capacities. The hubs serve as regional intermediate locations, where medical items are consolidated, temporarily stored in clinically acceptable conditions, and prepared for secondary transshipment to the final delivery locations. In healthcare logistics, final delivery destinations are often points of injury, rural clinics, and temporary triage zones, where patients require urgent medical attention. Therefore, the secondary distribution (i.e., distribution from the logistics hubs to the final delivery destinations) in the H\&S structure utilizes small high-speed drones to perform the last-mile delivery of medical items within the specified time windows.

To fully leverage the autonomy of drones in a hierarchical H\&S network, an efficient last-mile delivery necessitates utilizing a heterogeneous drone fleet. In contrast to a homogeneous fleet, a heterogeneous fleet of drones with distinct characteristics---cost, battery capacity, power consumption, package weight carrying capacity (PWCC), and speed---provides a greater operational flexibility to meet the varying requirements (e.g., package weight and time window) of the medical items. Smaller drones can be deployed for lightweight medical items with tighter time windows, whereas larger drones with extended range and larger PWCC are better suited for heavier packages (i.e., medical items) that need to be transported over longer distances.

Furthermore, safe and efficient operation of drones requires ($i$) continuously tracking their battery energy consumption, accounting for practical factors (e.g., package weight, flight distance, and drone type); and ($ii$) efficiently managing their battery swapping. For drones with larger battery capacities, battery swapping can be performed exogenously at each return to the origin (i.e., central depot or hubs). In contrast, for smaller drones that perform more frequent, shorter-distance delivery operations, the decision to swap the drone battery is made endogenously (i.e., only when needed).
Moreover, due to the dynamic nature of emergency medical response, final delivery locations are often mobile (e.g., ambulances) or reallocated frequently according to real-time incident locations. Delivering medical items to these mobile destinations requires inter-hub and intra-hub coordination, where hubs serve as fixed regional points for consolidation, and last-mile delivery operations dynamically adapt to serve mobile destinations.

Despite the benefits of the hierarchical H\&S networks, to the best of our knowledge, there is a lack of studies on designing hierarchical H\&S drone delivery networks for healthcare logistics, accounting for all the aforementioned practical aspects. Moreover, there is a critical need for fast and exact solution methods to provide accurate and practically usable healthcare logistics solutions. However, developing efficient exact solution methods with tailored enhancement strategies (e.g., problem-specific reformulations and dynamic cutting planes) to tighten the underlying LP-relaxation of this inherently complex problem has not received considerable attention. Therefore, this study aims to develop efficient models and exact solution methods for routing and scheduling a heterogeneous drone fleet within a hierarchical H\&S network to deliver time-sensitive medical items to fixed and mobile delivery destinations.

\vspace{-5pt}
\subsection{Related Literature}

The existing literature relevant to our study falls into four key categories, including (1) the hub-and-spoke (H\&S) networks in transportation logistics~\citep{li2025hierarchical, sun2024multi, lin2008integral}, (2) the multi-echelon vehicle routing problems (VRPs)~\citep{santos2015branch, dellaert2021multi, marques2022branch}, (3) the drone delivery network designs~\citep{pinto2020network, jiu2024benders, hu2024drone}, and (4) the drone-based healthcare logistics ~\citep{enayati2023multimodal, kar2024optimal, gentili2022locating}. 

The H\&S network structures form the foundation of the hierarchical logistics and transportation systems~\citep{unisHS}. The early study by~\citet{chou1990hierarchical} introduced the pioneering hierarchical hub model that used a simplified shortest path algorithm to determine the optimal number and arrangement of hubs in multi-tiered airline networks. \citet{lin2004hierarchical} leveraged the hierarchical network design for time-definite express common carriers. They developed a binary programming model, solved using an implicit enumeration algorithm (IEA), to determine the optimal fleet size and scheduling decisions across primary (inter-hub) and secondary (intra-hub) routes. \citet{thomadsen2005hierarchical} further advanced the hierarchical network design through a two-layer ring network model in telecommunications, solved using a branch-and-price (B\&P) algorithm. \citet{lin2008integral} extended this line of research by formulating an integer programming model, solved using IEA, for time-definite freight carriers that minimized total operating costs over a generalized H\&S structure. \citet{correia2011hub} integrated hub capacity decisions in a hierarchical hub location problem, formulated as a mixed-integer linear programming (MILP) model and solved using CPLEX. Furthermore, \citet{carlsson2013euclidean} investigated the H\&S networks in a continuous Euclidean space and proved that the Euclidean H\&S networks in a continuous space can be reduced to a generalized discrete $p$-median problem. \citet{dukkanci2017routing} studied a hierarchical hub location problem with a ring-star-star (RSS) network structure, solved using a subgradient-based heuristic algorithm. \citet{shang2020stochastic} also adopted this RSS network structure in a stochastic hierarchical multimodal hub location problem for cargo delivery systems. They formulated the problem as a mixed-integer nonlinear programming (MINLP) model, solved using a memetic algorithm and Monte Carlo simulation, to consider uncertainties in the travel times and construction costs. Incorporating demand uncertainty, \citet{li2025hierarchical} proposed a hierarchical hub-location model for integrated urban-rural logistics, solved using a scenario-based branch-and-benders-cut algorithm. While these studies significantly advanced the design and operation of H\&S networks in transportation logistics, they did not include aerial drones for freight transshipment, nor did they explicitly study the vehicle routing and scheduling problems to transport different commodities across a hierarchical H\&S network.

The integration of vehicle routing within hierarchical network structures for freight transshipment has led to extensive research on the two-echelon vehicle routing problem (2E-VRP). \citet{hemmelmayr2012adaptive} introduced one of the early studies on the 2E-VRP, where freight is shipped from a depot through intermediate satellite points to the final customers. The authors used an adaptive large neighborhood search algorithm to solve the problem. Later, \citet{baldacci2013exact} and \citet{santos2013branch} incorporated capacity constraints into the 2E-VRP. They proposed MILP models, solved using B\&P, for the two-echelon capacitated vehicle routing problem (2E-CVRP). To develop more advanced exact solution methods, \citet{santos2015branch} solved the 2E-CVRP using a branch-price-and-cut (BPC) algorithm. \citet{wang2018two} extended this line of research by incorporating time windows into the 2E-VRP, solved using a non-dominated sorting Genetic algorithm (NSGA). Later, \citet{dellaert2019branch} proposed an exact B\&P algorithm for the 2E-VRP with time windows (2E-VRPTW), where the authors used a dynamic programming algorithm for solving the pricing subproblems. \citet{mhamedi2022branch} further improved the computational efficiency in solving the 2E-VRPTW, using a BPC algorithm. They introduced multiple classes of valid inequalities to strengthen the linear relaxation. \citet{marques2022branch} extended the 2E-VRPTW to the multi-trip operations of the vehicles. The authors formulated the problem as an MILP model, and solved using an advanced BPC algorithm. \citet{dellaert2021multi} also studied the 2E-VRPTW by incorporating multiple commodities with customer-specific demands and distinct origin-destination pairs, solved using a BPC algorithm. Furthermore, \citet{wu2023branch} studied the 2E-VRP with electric vehicles (EVs) in the second echelon, where EVs need to be recharged at recharging stations. The authors formulated the problem as an arc flow MILP model, and solved using a B\&P algorithm. Despite developing advanced exact solution methods for the 2E-VRPs with different practical aspects, such as time windows and capacity constraints, these studies did not fully address the hierarchical H\&S network design for a heterogeneous fleet of drones with capacity and energy constraints. Moreover, the simultaneous integration of fixed and mobile delivery locations with distinct release times and delivery due times remains largely unexplored.

The growing need for automation and on-demand logistics has shifted research toward drone-based delivery systems. However, the limited battery capacity and flight range of drones limit the adaptability of drone logistics systems in wide geographical distribution networks. To address this challenge, numerous studies have focused on drone delivery network designs to extend the coverage of the drone-based delivery system. Many researchers studied the integration of ground vehicles with drones to extend the drones' coverage (e.g., \citealt{mirzapour2024optimizing, moshref2021comparative, ulmer2018same, zou2024delivery}). Another line of research incorporated intermediate recharging stations or automated battery swapping machines (ABSMs) to enable drones operating over extended distances. For instance, \citet{pinto2020network} proposed a network design model for a meal delivery service using drones. They developed an MILP model, solved using Gurobi, to determine the optimal number and locations of charging stations to maximize the potential demand that can be covered. Later, \citet{pinto2022point} studied the point-to-point drone-based delivery network design with intermediate charging stations to extend the drone coverage. This study used a heuristic method based on the $m$-shortest path concept to solve the problem. \citet{cokyasar2021designing} also studied a drone delivery network design, leveraging ABSMs to further extend the drones' coverage, solved using a cutting plane algorithm. A few studies leveraged the H\&S network design in drone-based delivery logistics. \citet{gao2022optimizing} studied a truck-drone H\&S delivery network, where trucks transport goods between the hubs, and drones perform the last-mile delivery to the spoke nodes. They used a three-stage metaheuristic algorithm---variable neighborhood search, nearest neighbor algorithm, and reassignment strategy---to determine the optimal location of hubs, and the allocation of spoke nodes to the hubs. \citet{hu2024drone} further extended this research by including the flight range limitation of the drones and service time limits of the customers. The authors used an NSGA heuristic to determine the minimum consolidation and transshipment cost of the assignment of the spoke nodes to the hubs. However, the hierarchical H\&S network structure has not been studied yet in the drone delivery network design literature, especially considering a heterogeneous fleet of drones with different operational characteristics and delivery modes. The hierarchical H\&S network designs are particularly important for regionally distributed demand points, where consolidation, transshipment, and last-mile operations must be coordinated across multiple layers of the network~\citep{ogazon2025designing}. Furthermore, the coordinated delivery of time-sensitive products with distinct weights, release times, and time windows, has remained largely unexplored in the hierarchical drone delivery network design context.

In the healthcare domain, several studies have explored drone-based delivery logistics to improve the accessibility of critical supplies in time-sensitive contexts. \citet{wen2016multi} studied a drone-based logistics system for delivering critical medical supplies to wounded individuals in an emergency situation. They proposed a bi-objective non-linear model, solved using an NSGA. Their goal was to minimize the total distribution time, as well as the required number of drones. \citet{kim2017drone} designed a drone-aided healthcare service for patients with Chronic diseases in rural areas. The authors developed a pre-processing algorithm, a partition method, and a Lagrangian relaxation (LR) to minimize the drone operating costs, where drones deliver medications and collect the test kits. Later, \citet{rabta2018drone} proposed an MILP model, solved using GAMS, to deliver medical items using a homogeneous drone fleet in disaster relief operations. \citet{ghelichi2021logistics} and \citet{gentili2022locating} extended this line of research by studying drone-based timely deliveries of medical items to remote or hard-to-access locations. Both studies developed time-slot-based MILP models, solved using commercial solvers (e.g., Gurobi and CPLEX), to jointly route and schedule drones, and determine the optimal locations for intermediate facilities. \citet{ghelichi2021logistics} used intermediate recharging stations, while \citet{gentili2022locating} focused on the location of drone platforms, as well as the allocation of demand points to the platforms. The aforementioned studies addressed important aspects of drone-based healthcare and emergency logistics. However, they lack hierarchical drone delivery network designs, which are essential for serving geographically distributed demand points in a multi-layer distribution system. Moreover, continuous tracking of the energy consumption of a heterogeneous drone fleet across multiple trips to deliver medical items with distinct package release times and delivery due times remains largely unexplored. Furthermore, these studies typically used heuristic methods without an optimality guarantee, or commercial solvers (e.g., CPLEX) with limited scalability to realistic-sized problem instances. Given the life-saving implications of drone-based healthcare logistics, strengthening the mathematical formulation through advanced separation procedures is essential for solving large-scale instances to optimality. Yet, such methodological advancements remain largely unexplored in this literature.

To the best of our knowledge, only a few studies have explored the multi-layer drone-delivery network designs in healthcare logistics. \citet{enayati2023multimodal} addressed a multimodal vaccine distribution network that integrates traditional transportation modes (e.g., trucks, airplanes, and boats) with drones. Their problem determined the locations of distribution centers, drone bases, and drone relay stations to recharge the drones. This study also considered the limited flight range of two types of drones, as well as the cold chain time limit of vaccines. They developed two MILP formulations, including a compact model that tracks the aggregated travel time for vaccines, and a layered-network model that explicitly traces the vaccine flow associated with each origin-destination pair. The authors proposed a modified Benders decomposition algorithm to solve large-scale instances of the problem. However, in this study, the vaccine demand is modeled as a continuous quantity at the regional healthcare zones, rather than as discrete, package-specific deliveries that can capture distinct package weights, release times, and delivery due times. \citet{amirsahami2023hierarchical} presented a hierarchical model for strategic and operational planning in blood transportation with drones. The authors developed an MILP model, solved using an augmented $\epsilon$-constraint method, to minimize the total cost of a three-layer distribution network, including blood collection centers, blood transfusion centers, and hospitals. However, the flow of donated blood is not explicitly modeled in their study, and the demand is represented as the required number of drone trips at each facility. Additionally, neither study addressed the hierarchical H\&S network design with explicit inter-hub and intra-hub coordination, as well as routing and scheduling a heterogeneous drone fleet with different operational characteristics and delivery modes for delivering medical items with distinct weights, release times, and time windows.

Despite the valuable insights from the existing literature, several critical gaps remain unaddressed. \textbf{First}, the most important research gap is the lack of studies on the hierarchical H\&S drone delivery network design for delivering time-sensitive medical items with distinct package weights, release times, delivery destinations, and delivery due times. \textbf{Second}, existing studies often overlooked heterogeneous drone fleets with distinct costs and operational characteristics (i.e., battery capacity, package weight carrying capacity, speed, and power consumption), which perform multiple trips through different delivery modes (i.e., direct delivery vs. circular delivery). Moreover, continuously tracking the payload and energy consumption of each drone during each flight segment, accounting for the actual drone data and practical factors---package weight, distance, and drone type---with exogenous and endogenous battery swapping of the drones, remains largely unexplored. \textbf{Third}, the existing literature did not simultaneously consider fixed and mobile delivery destinations, each requiring a specific package delivery, which necessitates inter-hub and intra-hub coordination among drones across the delivery network. \textbf{Fourth}, current studies typically overlooked the need for efficient exact methods which tighten the underlying LP-relaxation through problem-specific reformulations and tailored dynamic cutting planes, leveraging the hierarchical H\&S structure. Unlike generic valid inequalities, these cuts arise from the structure of the hierarchical H\&S drone delivery problem itself, and require customized separation routines on the fractional LP solutions to tighten the LP-relaxation. Developing such efficient exact methods is essential for solving large-scale instances of the hierarchical H\&S drone delivery problem, which remain largely unexplored in the existing literature.

\vspace{-5pt}
\subsection{Contributions}

To fill the aforementioned gaps in the existing literature, we make the following contributions. \textbf{First,} we develop a comprehensive MILP model for a novel problem of designing a hierarchical H\&S network for the last-mile delivery of time-sensitive medical items in healthcare logistics. We address the following practical aspects: (1) a hierarchical H\&S network design with inter-hub and intra-hub coordination, where logistics hubs are intermediate locations (e.g., regional hospitals) for consolidation between a central depot and geographically distributed delivery destinations; (2) a heterogeneous fleet of drones with different delivery modes---direct delivery vs. circular delivery---and distinct characteristics, including cost, battery capacity, PWCC, speed, and power consumption; (3) different packages (i.e., medical items), each with a distinct package weight, release time at the central depot, delivery destination, and delivery due time; (4) simultaneous consideration of fixed (e.g., clinics) and mobile delivery destinations (e.g., ambulances); (5) continuous tracking of the energy consumption of drones across multiple trips, accounting for actual drone data and practical factors (e.g., package weight, delivery distance, and drone type); and (6) adaptive battery swapping with respect to drone type, where longer-range drones perform exogenous battery swaps (i.e., at the end of each trip), whereas shorter-range drones perform endogenous battery swaps (i.e., on an as-needed fashion during the delivery operation).

\textbf{Second,} we propose a customized exact solution method integrating the following set of enhancement strategies: (1) defining restricted sets and variable fixings to reduce the feasible region; (2) developing problem-specific reformulations to tighten the LP-relaxation; (3) proposing valid inequalities to reduce the solution space; and (4) designing dynamic cutting plane generation procedures to tighten the underlying LP-relaxation.

\textbf{Third,} we develop a heuristic solution algorithm, leveraging a simplified variant of the original problem, which integrates a set-partitioning model to decouple the hierarchical structure. The proposed heuristic algorithm solves large-scale problem instances faster, while providing high-quality near-optimal solutions.

\textbf{Fourth,} based on a real-life case study of whole blood delivery data and actual drone flight test data, we provide valuable managerial insights for healthcare logistics system owners into (1) the performance of different solution methods in solving the hierarchical H\&S drone delivery network design problem to deliver time-sensitive medical items across a regional healthcare network; (2) the effect of fleet type on the total cost and the required fleet composition of drones; and (3) the effect of maximum permissible delay, consolidation delay, and the multi-trip delivery mode of drones on the total cost and the required number of drones of each type.

\vspace{-10pt}
\section{Problem Description}\label{sec:Problem_Description}

In this paper, we study a hierarchical H\&S drone delivery network design for a heterogeneous fleet of drones to deliver time-sensitive medical items (e.g., blood products) across a regional healthcare system. The hierarchical structure in our study, comprises three functional layers and two operational echelons, as shown in Figure \ref{fig:Problem_Description}.

\begin{figure}[h!]
    \centering
    \includegraphics[width=0.75\linewidth]{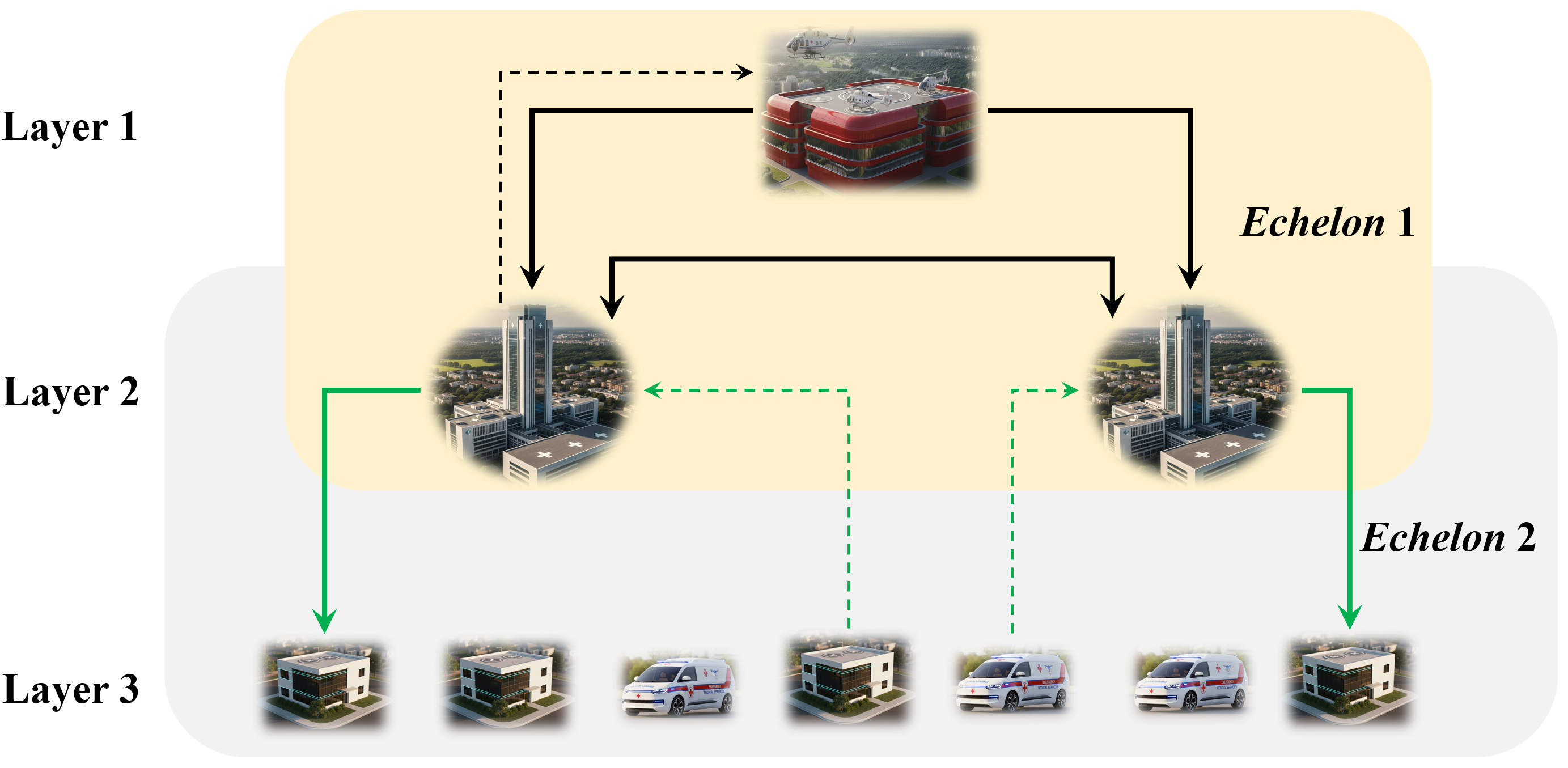}
    \caption{Visualization of the hierarchical network design to deliver time-sensitive medical items.}
    \label{fig:Problem_Description}
\end{figure}

\textbf{Layer 1} is the strategic supply layer, represented by a \textit{central depot} (CD), such as a central blood bank, that acts as the primary source of all medical items (e.g., whole blood). Each medical item (i.e., package) has distinct characteristics, including weight, release time at the CD, and the delivery due time at the corresponding delivery destination. \textbf{Layer 2} is the intermediate logistics layer, responsible for receiving bulk shipments from layer 1, temporarily storing medical items (if needed), and coordinating their secondary distribution to the next layer. We represent this layer by a set of \textit{logistics hubs} (LHs), which corresponds to regional medical centers (e.g., major hospitals). \textbf{Layer 3} represents the final delivery locations (i.e., delivery destinations). Each delivery location requires a single package delivery before its respective delivery due time to support emergency operations, such as blood transfusion to a patient. This layer includes two types of delivery locations: ($i$) a set of \textit{clinics} (CLs), representing stationary medical sites; and ($ii$) a set of \textit{ambulances} (AMBs), representing mobile medical sites. Each AMB is assigned to an \textit{ambulance location} (AL) selected from a set of candidate ALs, and receives its corresponding medical item at the designated AL.

The two echelons define the operational interactions among these three layers. \textbf{Echelon 1} manages the long-range transportation of the medical items from layer 1 (i.e., the CD) to layer 2 (i.e., the LHs), using a homogeneous fleet of large drones (LDs), as shown in Figure \ref{subfig:Echelon1}. In this echelon, LDs with long flight ranges and large package weight carrying capacities (e.g., Aero-200 long-range drones) perform multiple circular delivery trips. Each LD's trip starts from the CD by picking up multiple packages that are available (i.e., released). The LD then sequentially visits the LHs, and delivers multiple packages to each visited LH. The long flight ranges and large battery capacities of the LDs enable them to complete each trip without requiring en-route battery swapping. Upon delivering all the assigned packages in each trip, the LD returns to the CD to pick up the next available packages, swaps its battery, and conducts its next trip. Each LH in echelon 1 can be visited multiple times across different trips of the same LD, and/or by different LDs. Therefore, the multi-trip circular delivery of the LDs repeats until all the packages are transported from the CD to the LHs.

\begin{figure}[htbp]
    \centering
    \begin{subfigure}[b]{0.45\textwidth}
        \centering
        \includegraphics[width=\linewidth]{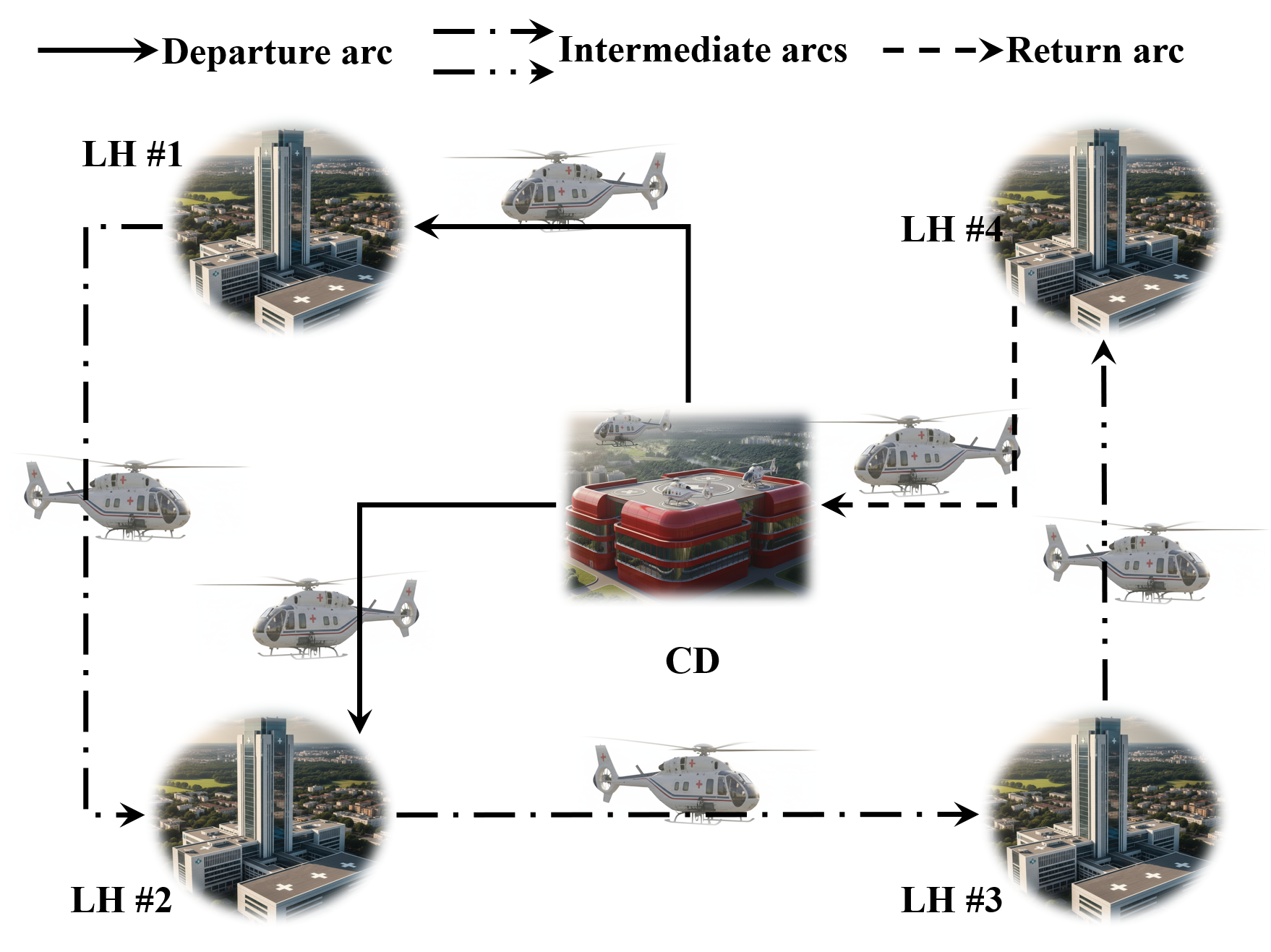}
        \caption{LDs in echelon 1.\vspace{-7pt}}
        \label{subfig:Echelon1}
    \end{subfigure}
    \hfill
    \begin{subfigure}[b]{0.45\textwidth}
        \centering
        \includegraphics[width=\linewidth]{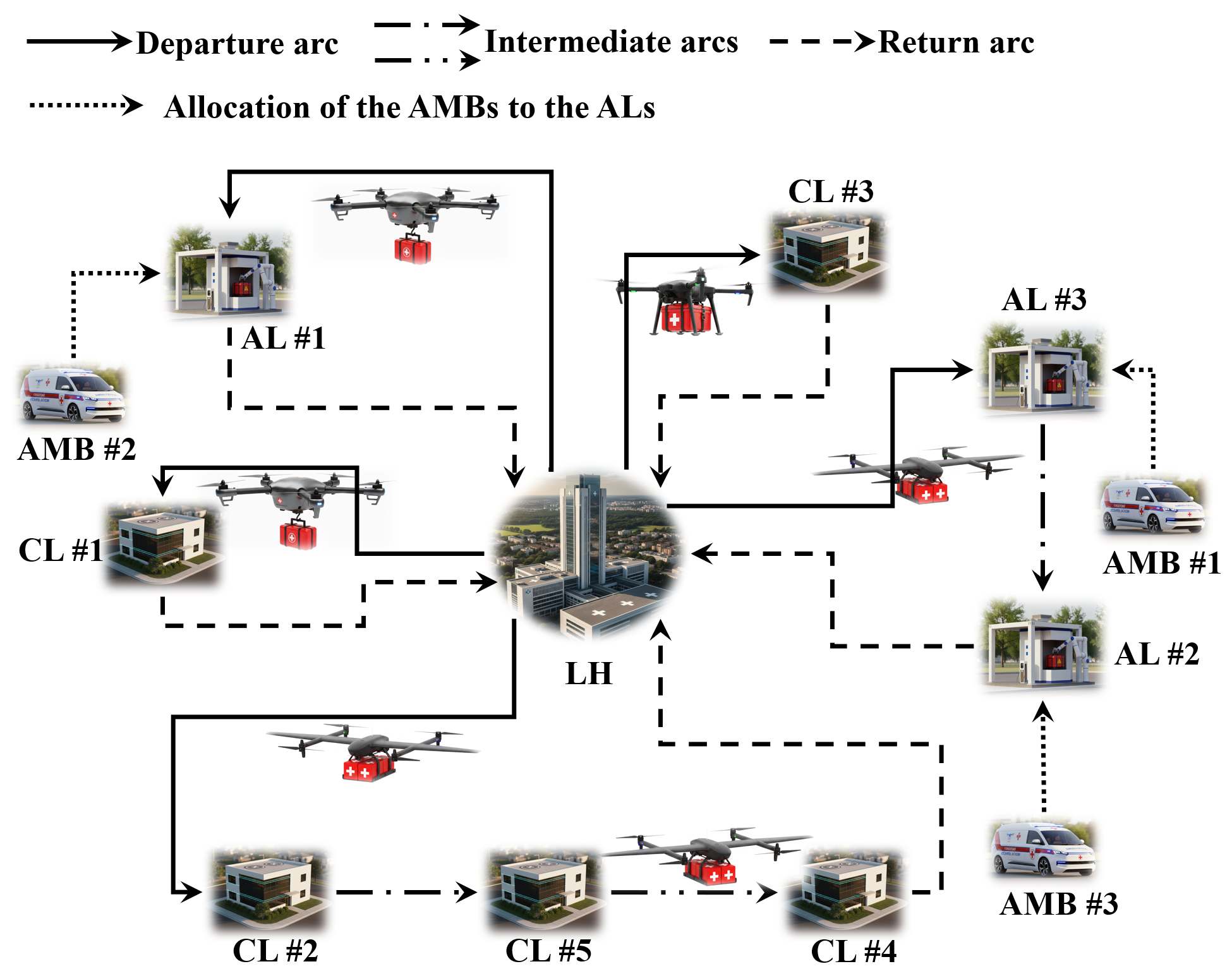}
        \caption{MDs and SDs in echelon 2.\vspace{-7pt}}
        \label{subfig:Echelon2}
    \end{subfigure}
    \caption{Visualization of the routing and scheduling of drones in the two echelons.}
    \label{fig:Echelon2_routing}
\end{figure}
Although we consider the LDs with long flight ranges and the ability to carry a large amount of packages in each trip, their battery capacity and package weight carrying capacity (PWCC) are limited. Therefore, the total package weight each LD carries in each trip must not exceed its PWCC. Moreover, to ensure the safe operation of drones, we explicitly account for the actual energy consumption of the LDs during the delivery process, based on real drone flight test data and practical factors (e.g., package weight, travel distance, and drone type). As package weight directly affects the energy consumption of drones during each trip, we track the payload (i.e., package weight) and energy consumption of each LD during each flight segment---traveling between different nodes in echelon 1 (i.e., the CD and the LHs)---to ensure that each LD can complete each trip with sufficient remaining battery energy and return to the CD to swap its battery.

Considering all the above practical aspects, the multi-trip circular delivery of each LD in echelon 1 accounts for: (1) the routing and scheduling of the LD between the CD and the corresponding LHs (i.e., the order and timing of visiting the LHs in each trip), (2) the packages that the LD carries in each trip, (3) the delivery location of each package (i.e., the corresponding LH that each package is delivered), and (4) the delivery time of each package to the corresponding LH.

\textbf{Echelon 2} conducts the localized distribution (i.e., last-mile delivery) of medical items from layer 2 (i.e., LHs) to layer 3 (i.e., delivery destinations---CLs and AMBs), as shown in Figure \ref{subfig:Echelon2}. At each LH, a homogeneous fleet of medium-range drones (MDs), and a heterogeneous fleet of short-range drones (SDs) are stationed. Given the larger PWCC and the longer flight range of the MDs (e.g., Wingcopter 198 vertical-takeoff-and-landing) compared to SDs (e.g., rotary drones---Tarot 650 and DJI Matrice 600 Pro), the MDs perform multi-trip circular deliveries, whereas the SDs conduct multiple direct delivery trips. Additionally, each SD type has distinct operational characteristics, including battery capacity, PWCC, power consumption, and speed. As each delivery destination requires a single package delivery, each CL/AMB is visited exactly once by a single drone. Furthermore, the delivery of the packages to the CLs and the AMBs is performed with separate drones, i.e., a drone serving the CLs does not deliver packages to the AMBs. This is because, packages destined for the AMBs typically require different packaging and transfer protocols than those delivered to the CLs. For instance, blood transfusion in ambulances requires maintaining a specific temperature range, which necessitates using drones equipped with this technology.

In the circular delivery mode, each MD performs multiple trips between an LH and the CLs/AMBs. In each trip, an MD departs an LH, carrying multiple available packages (i.e., packages that are delivered to the LH before the departure time of the MD). Then the MD visits a sequence of CLs/AMBs to deliver packages, and returns to the same LH to swap its battery and perform the next trip. Similar to the LDs, we account for the PWCC of the MDs, and track their payload and energy consumption during each trip. Additionally, the MDs have limited spatial capacity, where the number of packages each MD can carry in each trip is limited. Therefore, the circular delivery of each trip of an MD accounts for: (1) the departure time from the corresponding LH, (2) the set of packages that need to be delivered to the corresponding CLs/AMBs, (3) the sequence of delivering the packages, and (4) the delivery time of each package to the corresponding CL/AMB.

Unlike MDs, each SD performs multiple direct delivery trips between its associated LH and the CLs/AMBs, carrying a single package at a time. Each direct delivery trip starts from an LH by picking up an available package. The SD then flies to the corresponding CL/AMB to deliver the package, and immediately returns to the same LH to perform its next trip. In contrast to the MDs and the LDs, we consider endogenous battery swapping for the SDs to ensure safe and efficient operation of drones, where a drone's battery is swapped only if it is needed. For each SD, after returning to the LH from a CL/AMB, and before leaving the LH for the next CL/AMB, we compute the potential remaining energy in the drone's battery. The potential remaining energy is the amount of energy that will be remaining in the drone's battery upon its return to the LH, if the drone delivers the package to the next CL/AMB and returns. If the potential remaining energy falls below the minimum required energy in the drone's battery, we swap the SD's battery with a fully charged one at the LH before it departs for the next CL/AMB. Therefore, the multi-trip direct delivery schedule of each SD in echelon 2 accounts for: (1) the packages that need to be delivered across multiple trips, (2) the sequence of delivering the packages, (3) the delivery time of each package to the corresponding CL/AMB, and (4) the endogenous battery swaps.

The goal of the logistics system owner (i.e., decision-maker) in this problem is to determine the allocation of the AMBs to the ALs, as well as the efficient routing and scheduling of all the drones in the network. The decision-maker's goal is to efficiently make up a heterogeneous fleet of drones, as well as to minimize the total system cost, while delivering all medical items within their respective time windows---defined by the package release times and delivery due times. The total cost includes the investment and transportation energy consumption costs of drones, the total cost of the required drone batteries, and the allocation cost of the AMBs to the ALs. We refer to this problem as the hierarchical H\&S drone delivery network design for time-sensitive medical items with distinct release times and delivery due times (\textit{HDDNM}).

\vspace{-10pt}
\section{Mathematical Formulation}\label{sec:Model-Formulations}

In this section, we present the mathematical formulation of the \textit{HDDNM} problem, which we formulate as an arc flow MILP model. The mathematical formulation in this section, presented in Eqs. \eqref{eq:obj} - \eqref{E2-Const65}, is a complete polynomial-sized MILP model based on the classical Miller-Tucker-Zemlin (MTZ) formulation (referred to as \textit{BaM} throughout the remainder of the paper). We list the necessary sets, parameters, and variables in Tables \ref{tab:Sets}, \ref{tab:Parameters}, and \ref{tab:Variables}, respectively, that support the mathematical formulation presented in this section. \vspace{-5pt}

\begin{table}[H]
  \caption{Sets.\vspace{-5pt}}
  \label{tab:Sets}
  \begin{tabularx}{\linewidth}{lX}
  \toprule
    \textbf{Set} & \textbf{Description}\\
    \midrule
    \hline
    $\mathcal I$ & Set of CLs, indexed by $i$ and $j$ \\
    $\mathcal M$ & Set of AMBs and their corresponding packages, indexed by $m$ and $n$ \\
    $\mathcal P$ & Set of all packages for CLs and AMBs, indexed by $p$, where $\mathcal P = \mathcal I \cup \mathcal M$ \\
    $\mathcal D$ & Set of LDs, indexed by $d$ \\
    $\mathcal H$ & Set of LHs, indexed by $h$ and $k$ (Note: $h$ is the hub index. In echelon 2 of the network, we represent the dummy source and sink for each LH $h$, as $h$ and $h'$.)  \\
    $\mathcal H'$ & Set of all nodes in echelon 1, including the dummy source (CD) $0$ and sink $S$, where $\mathcal H' = \{0\} \cup \mathcal H \cup \{S\}$  \\
    $\mathcal G^{LD}$ & Set of trips of LDs, indexed by $g$ \\
    $\mathcal E$ & Set of arcs in echelon 1 of the network, where $\mathcal E = \{(0, h): h \in \mathcal H\} \cup \{(h, k) \in \mathcal H \times \mathcal H: h \neq k\} \cup \{(h, S): h \in \mathcal H\}$ \\
    $\mathcal L$ & Set of candidate locations for the AMBs (ALs), indexed by $l$ and $o$  \\
    $\mathcal T_h$ & Set of SD types stationed at LH $h$, indexed by $t$   \\
    $\mathcal V_h$ & Set of MDs stationed at LH $h$, indexed by $v$  \\
    $\mathcal G^{MD}$ & Set of trips of MDs, indexed by $g'$  \\
    $\mathcal A_h$ & Set of arcs in echelon 2 of the network between LH $h$ and the CLs, where $\mathcal A_h = \{(h,i):i\in\mathcal I \} \cup \{(i,j) \in\mathcal I \times \mathcal I: i \neq j \} \cup \{(i,h'): i\in\mathcal I \}$  \\
    $\mathcal B_h$ & Set of arcs in echelon 2 of the network between LH $h$ and the ALs, where $\mathcal B_h = \{(h,l):l\in\mathcal L \} \cup \{(l,o) \in\mathcal L \times \mathcal L: l \neq o \} \cup \{(l,h'): l\in\mathcal L \}$  \\
    \bottomrule
  \end{tabularx}
\end{table}

\begin{table}[H]
  \caption{Parameters.\vspace{-5pt}}
  \label{tab:Parameters}
  \begin{tabularx}{\linewidth}{lX}
  \toprule
    \textbf{Parameter} & \textbf{Description}\\
    \midrule
    \hline
    $F^{LD/MD/t}$ & Fixed cost of using an LD/MD/SD of type $t$ \\
    $C^E$ & Energy consumption cost (\$/kWh) \\
    $C_{bat}^{LD/MD/t}$ & Fixed cost of the battery of an LD/MD/SD of type $t$ \\
    $T_{hk}^{LD}$ & Time required for an LD to travel on arc $(h,k)\in\mathcal E$ \\
    $T_{ij/lo}^{MD}$ & Time required for an MD to travel on arc $(i,j)\in\mathcal A_h$/$(l,o)\in\mathcal B_h, \forall h\in\mathcal H$\\
    $T_{hi}^{t}, T_{ih'}^{t}$ & Time required for an SD of type $t$ to travel between LH $h$ and CL $i\in\mathcal I$ \\
    $T_{hl}^{t}, T_{lh'}^{t}$ & Time required for an SD of type $t$ to travel between LH $h$ and AL $l\in\mathcal L$ \\
    $T_{bat}^{LD/MD/t}$ & Battery swapping time for an LD/MD/SD of type $t$ \\
    $T_L^{LD/MD/t}$ & Package loading time for an LD/MD/SD of type $t$ \\
    $T_U^{LD/MD/t}$ & Package unloading time for an LD/MD/SD of type $t$ \\
    $R_{p}$ & Release time of package $p\in\mathcal P$ at the CD \\
    $L_{p}$ & Due time of delivering package $p$ to the corresponding CL/AMB \\
    $P_{p}$ & Weight of package $p$ \\
    $P_{max}^{LD/MD/t}$ & Package weight carrying capacity of an LD/MD/SD of type $t$ \\
    $P_N^{MD}$ & Maximum number of packages an MD can carry in each trip \\
    $E_{hk}^{0, LD}$ & Base energy consumption (kWh) of an LD to travel on arc $(h,k)\in\mathcal E$ \\
    $E_{hk}^{1, LD}$ & Rate of increment in the energy consumption (kWh/kg) of an LD to travel on arc $(h,k)\in\mathcal E$ with respect to the increment in the package weight \\
    $E_{ij/lo}^{0, MD}$ & Base energy consumption (kWh) of an MD to travel on arc $(i,j)\in\mathcal A_h$/$(l,o)\in\mathcal B_h, \forall h\in\mathcal H$\\
    $E_{ij/lo}^{1, MD}$ & Rate of increment in the energy consumption (kWh/kg) of an MD to travel on arc $(i,j)\in\mathcal A_h$/$(l,o)\in\mathcal B_h, \forall h\in\mathcal H$ with respect to the increment in the package weight \\
    $E_{hi/hl}^{0, t}$  & Base energy consumption (kWh) of an SD of type $t$ to travel on arc $(h,i)$/$(h,l)$ \\
    $E_{hi/hl}^{1, t}$ & Rate of increment in the energy consumption (kWh/kg) of an SD of type $t$ to travel on arc $(h,i)$/$(h,l)$ with respect to the increment in the package weight \\
    $B_{max}^{LD/MD/t}$ & The capacity of a fully charged battery of an LD/MD/SD of type $t$   \\
    $B_{min}^{LD/MD/t}$ & Minimum required energy in the battery of an LD/MD/SD of type $t$   \\
    $LC_l$ & Fixed cost of opening AL $l$ for serving an ambulance \\
    $AC_{ml}$ & Cost of allocating AMB $m\in\mathcal M$ to AL $l$ \\
    \bottomrule
  \end{tabularx}
\end{table}

\begin{table}[H]
  \caption{Variables.\vspace{-5pt}}
  \label{tab:Variables}
  \begin{tabularx}{\linewidth}{lX}
  \toprule
    \textbf{Variable} & \textbf{Description}\\
    \midrule
    \hline
    $w_{ph}^{dg}$ & 1 if package $p\in\mathcal P$ is delivered to LH $h\in \mathcal H$ by LD $d\in\mathcal D$ on its trip $g\in\mathcal G^{LD}$,
    0 otherwise \\
    $x_{hk}^{dg}$ & 1 if LH $h$ is visited immediately before LH $k\in\mathcal H$ by LD $d$ on its trip $g$, 0 otherwise \\
    $dt_0^{dg}$ & Departure time of LD $d$ on its trip $g$ from the CD  \\
    $rt_0^{dg}$ & Timing of when LD $d$ returns to the CD after concluding its trip $g$ \\
    $at_{h}^{dg}$ & Arrival time of LD $d$ to LH $h$ on its trip $g$ \\
    $\pi_p$ & Delivery time of package $p$ to an LH  \\
    $u^d$ & 1 if LD $d$ is used, 0 otherwise \\
    $u^{dg}$ & 1 if LD $d$ performs its trip $g$, 0 otherwise \\
    $r_h^{dg}$ & 1 if LD $d$ visits LH $h$ on its trip $g$, 0 otherwise \\
    $pw_{hk}^{dg}$ & Continuous variable representing the total weight of the packages carried by LD $d$ on its trip $g$, while traveling on arc $(h,k)\in\mathcal E$ \\
    $e_{hk}^{dg}$ & Continuous variable representing the energy consumption of LD $d$ on its trip $g$, while traveling on arc $(h,k)\in\mathcal E$\\
    $\alpha_l$ & 1 if AL $l$ is opened for serving an ambulance, 0 otherwise \\
    $\gamma_{ml}$ & 1 if AMB $m$ is allocated to AL $l$, 0 otherwise \\
    $\beta_{ph}$ & 1 if package $p$ is delivered to LH $h$, 0 otherwise \\
    $\theta_{ij/lo}^t$ & 1 if CL $i$/AL $l$ is served immediately before CL $j$/AL $o$ by an SD of type $t$, 0 otherwise \\
    $z_{ij/lo}^{vg'}$ & 1 if CL $i$/AL $l$ is visited immediately before CL $j$/AL $l$ by MD $v$ on its trip $g'$, 0 otherwise \\
    $\mu_h^v$ & 1 if MD $v$ from LH $h$ is used, 0 otherwise \\
    $\mu_h^{vg'}$ & 1 if MD $v$ from LH $h$ performs its trip $g'$, 0 otherwise \\
    $r_{ih/lh}^t$ & 1 if CL $i$/AL $l$ is served by an SD of type $t$ from LH $h$, 0 otherwise \\
    $y_{ih/lh}^{vg'}$ & 1 if CL $i$/AL $l$ is served by MD $v$ from LH $h$ on its trip $g'$, 0 otherwise \\
    $pw_{ij/lo}^{vg'}$ & Continuous variable representing the total weight of the packages carried by MD $v$ on its trip $g$, while traveling on arc $(i,j)\in\mathcal A_h$/$(l, o)\in\mathcal B_h, \forall h\in\mathcal H$  \\
    $\rho_{ih/lh}^t$ & Continuous variable representing the energy consumption of an SD of type $t$, while traveling on arc $(i,h)$/$(l,o)$  \\
    $q_{ij/lo}^{vg'}$ & Continuous variable representing the energy consumption of MD $v$ on its trip $g'$, while traveling on arc $(i,j)\in\mathcal A_h$/$(l,o)\in\mathcal B_h, \forall h\in\mathcal H$  \\
    $bt_h^{vg'}$ & Timing of when MD $v$ returns to LH $h$ after concluding its trip $g'$  \\
    $pt_h^{vg'}$ & Timing of when MD $v$ departs LH $h$ with carrying some packages on its trip $g'$  \\
    $s_{ih/lh}^t$ & 1 if the battery of an SD of type $t$ originating from LH $h$ is swapped after serving CL $i$/AL $l$ for visiting the next CL/AL, 0 otherwise \\
    $\delta_p$ & Delivery time of package $p$ to its corresponding CL/AMB  \\
    \bottomrule
  \end{tabularx}
\end{table}

\paragraph{Objective Function:}

\begin{align}
    \min & \; \textbf{C} = \textbf{C}^\text{fleet} + \textbf{C}^\text{batt} + \textbf{C}^\text{energy} + \textbf{C}^\text{alloc} \label{eq:obj}\\
    \text{where} & \; \textbf{C}^\text{fleet} = \sum_{d\in\mathcal D} F^{LD} \cdot u^d  + \sum_{h\in\mathcal H}\sum_{t\in\mathcal T_h}  F^t \Big(\sum_{i\in\mathcal I} \theta_{hi}^t + \sum_{l\in\mathcal L} \theta_{hl}^t \Big) + \sum_{h\in\mathcal H} \sum_{v\in\mathcal V_h} F^{MD} \cdot \mu_h^v \nonumber\\
     & \; \textbf{C}^\text{batt} = \sum_{d\in\mathcal D} \sum_{g\in\mathcal G^{LD}} C_{bat}^{LD}\cdot u^{dg} + \sum_{h\in\mathcal H}\sum_{t\in\mathcal T_h}  C_{bat}^t\Big(\sum_{i\in\mathcal I}s_{ih}^t + \sum_{l\in\mathcal L}s_{lh}^t \Big) + \sum_{h\in\mathcal H}\sum_{v\in\mathcal V_h} \sum_{g'\in\mathcal G^{MD}} C_{bat}^{MD}\cdot \mu_h^{vg'} \nonumber \\ 
     & \quad \textbf{C}^\text{energy} =  C^E \Bigg\{ \sum_{d\in\mathcal D} \sum_{g\in\mathcal G^{LD}} \sum_{(h,k)\in\mathcal E}e_{hk}^{dg}  + \sum_{h\in\mathcal H}\sum_{t\in\mathcal T_h} \Big(\sum_{i\in\mathcal I} \rho_{ih}^t + \sum_{l\in\mathcal L} \rho_{lh}^t \Big) \nonumber\\ 
     & \qquad \qquad \qquad+ \sum_{h\in\mathcal H}\sum_{v\in\mathcal V_h} \sum_{g'\in\mathcal G^{MD}} \Big(\sum_{(i,j)\in\mathcal A_h} q_{ij}^{vg'} +  \sum_{(l,o)\in\mathcal B_h} q_{lo}^{vg'} \Big) \Bigg\} \nonumber \\
     & \quad \textbf{C}^\text{alloc} = \sum_{l\in\mathcal L} LC_l \cdot \alpha_l + \sum_{l\in\mathcal L} \sum_{m\in\mathcal M}AC_{ml} \cdot \gamma_{ml} \nonumber
\end{align}

\vspace{-5pt}The objective function \eqref{eq:obj} aims to minimize the total system cost, denoted by $\textbf{C}$, to deliver all the packages from the CD to the final delivery locations (i.e., CLs and AMBs) within the specified time windows---defined by the package release times and delivery due times. Here, $\textbf{C}^\text{fleet}$ is the investment and operating costs of all the drones (i.e., LDs, MDs, and SDs). $\textbf{C}^\text{batt}$ is the investment cost of the required number of drone batteries. $\textbf{C}^\text{energy}$ is the total transportation energy consumption cost of all the drones in the two echelons of the network. Furthermore, $\textbf{C}^\text{alloc}$ is the total cost of opening the ALs and allocating the AMBs to the ALs. 

\vspace{-10pt}
\paragraph{Echelon 1:}

\allowdisplaybreaks
\begin{align}
    & \sum_{h\in\mathcal H} x_{0h}^{dg} = \sum_{k\in\mathcal H} x_{kS}^{dg} = u^{dg}, \quad \forall d\in\mathcal D,\, g\in\mathcal G^{LD} \label{eq:E1-Const1}\\
    & \sum_{k\in\mathcal H' \setminus \{S\}} x_{kh}^{dg} = \sum_{k\in\mathcal H' \setminus \{0\}} x_{hk}^{dg} = r_h^{dg}, \quad \forall h\in\mathcal H, \, d\in\mathcal D,\, g\in\mathcal G^{LD} \label{eq:E1-Const2} \\
    & w_{ph}^{dg} \leq r_h^{dg},\quad \forall p\in\mathcal P,\, h\in\mathcal H,\, d\in\mathcal D,\, g\in\mathcal G^{LD} \label{eq:E1-Const3} \\
    & r_h^{dg} \leq u^{dg},\quad \forall h\in\mathcal H,\, d\in\mathcal D,\, g\in\mathcal G^{LD} \label{eq:E1-Const4} \\
    & u^{dg} \leq u^d, \quad \forall d\in\mathcal D,\, g\in\mathcal G^{LD} \label{eq:E1-Const5} \\
    & \sum_{g\in\mathcal G^{LD}}\sum_{d\in\mathcal D}\sum_{h\in\mathcal H} w_{ph}^{dg} = 1, \quad \forall p\in\mathcal P  \label{eq:E1-Const6} \\
    & at_h^{dg} \geq dt_0^{dg}  + T_L^{LD} + T_U^{LD} + T_{0h}^{LD} - M_1 (1- x_{0h}^{dg}), \quad \forall h\in\mathcal H,\, d\in\mathcal D,\, g\in\mathcal G^{LD}  \label{eq:E1-Const8}\\
    & at_k^{dg} \geq at_h^{dg} + T_U^{LD} + T_{hk}^{LD} - M_1 (1 - x_{hk}^{dg}), \quad \forall h,k\in\mathcal H,\, h \neq k,\, d\in\mathcal D,\, g\in\mathcal G^{LD}  \label{eq:E1-Const9}\\
    & rt_0^{dg} \geq at_h^{dg} + T_{hS}^{LD} - M_1 (1 - x_{hS}^{dg}), \quad \forall h\in\mathcal H,\, d\in\mathcal D,\, g\in\mathcal G^{LD}  \label{eq:E1-Const10}\\
    & dt_0^{dg} \geq R_p \cdot \sum_{h\in\mathcal H}w_{ph}^{dg}, \quad \forall p\in\mathcal P,\, d\in\mathcal D,\, g\in\mathcal G^{LD}  \label{eq:E1-Const11}\\
    & dt_0^{dg} \geq rt_0^{d,g''} + T_{bat}^{LD},\quad \forall d\in\mathcal D,\, g,g''\in\mathcal G^{LD} , \, g''<g \label{eq:E1-Const12} \\
    & pw_{hk}^{dg} \leq M_2 \cdot x_{hk}^{dg}, \quad \forall (h,k)\in\mathcal E,\, d\in\mathcal D,\, g\in\mathcal G^{LD} \label{eq:E1-Const15}\\
    & -M_2 (1- x_{0h}^{dg}) \leq pw_{0h}^{dg} - \sum_{p\in\mathcal P}\sum_{k\in\mathcal H} P_p \cdot w_{pk}^{dg} \leq M_2 (1- x_{0h}^{dg}), \quad \forall h\in\mathcal H,\, d\in\mathcal D,\, g\in\mathcal G^{LD} \label{eq:E1-Const16}\\
    & -M_2 (1- r_{h}^{dg}) \leq \sum_{k': (k',h)\in\mathcal E} pw_{k'h}^{dg} - \sum_{p\in\mathcal P} P_p \cdot w_{ph}^{dg} - \sum_{k: (h,k)\in\mathcal E} pw_{hk}^{dg} \leq M_2 (1- r_{h}^{dg}), \label{eq:E1-Const17}\\
    & \qquad  \quad \forall h\in\mathcal H,\, d\in\mathcal D,\, g\in\mathcal G^{LD} \nonumber \\
    & \sum_{p\in\mathcal P}\sum_{h\in\mathcal H}P_p \cdot w_{ph}^{dg} \leq P_{max}^{LD}, \quad  d\in\mathcal D,\, g\in\mathcal G^{LD} \label{eq:E1-Const7} \\
    & e_{hk}^{dg} = E_{hk}^{0, LD} \cdot x_{hk}^{dg} + E_{hk}^{1, LD} \cdot pw_{hk}^{dg}, \quad \forall (h,k)\in\mathcal E,\, d\in\mathcal D,\, g\in\mathcal G^{LD} \label{eq:E1-Const18}\\
    & \sum_{(h,k)\in\mathcal E} e_{hk}^{dg} \leq B_{max}^{LD} - B_{min}^{LD}, \quad \forall d\in\mathcal D,\, g\in\mathcal G^{LD} \label{eq:E1-Const19} \\
    & u^d, u^{dg}, r_h^{dg}, w_{ph}^{dg}, x_{hk}^{dg} \in \{0,1\}, \, \qquad dt_0^{dg}, rt_{0}^{dg}, at_h^{dg}, pw_{hk}^{dg}, e_{hk}^{dg} \in \mathbb{R}^{+} \label{eq:E1-Const20}
\end{align}

\vspace{-5pt}Constraints \eqref{eq:E1-Const1} - \eqref{eq:E1-Const2} maintain the flow balance of the LDs traveling between the CD and the LHs.
Constraints \eqref{eq:E1-Const3} ensure that package $p\in\mathcal P$ can be delivered to LH $h$ on trip $g\in\mathcal G^{LD}$ of LD $d\in\mathcal D$ only if LD $d$ visits LH $h$ on its trip $g$. For brevity, we refer to trip $g$ of LD $d$ as \enquote{LD-trip} $dg$ throughout the remainder of the paper. LH $h$ can be visited by LD-trip $dg$ only if LD $d$ performs its trip $g$, as enforced by constraints \eqref{eq:E1-Const4}. Constraints \eqref{eq:E1-Const5} also ensure that LD $d$ can perform its trip $g$ only if LD $d$ is used in the echelon 1 delivery operations. Furthermore, each package must be delivered to exactly one LH by exactly one LD-trip (constraints \eqref{eq:E1-Const6}).

Constraints \eqref{eq:E1-Const8} - \eqref{eq:E1-Const9} compute the earliest possible time that each LH $h$ can be visited by an LD-trip $dg$. The computation of the earliest possible visit times accounts for: ($i$) the departure time of LD-trip $dg$ from the CD ($dt_0^{dg}$), ($ii$) the loading/unloading times of the LDs ($T_L^{LD}$ and $T_U^{LD}$), and ($iii$) the travel time between different nodes in echelon 1 ($T_{hk}^{LD}$). Constraints \eqref{eq:E1-Const10} compute the earliest time LD $d$ concludes its trip $g$, and returns to the CD. Here, $M_1$ needs to be greater than $\max_{p\in\mathcal P} \{L_p\}$. The earliest possible departure time of LD-trip $dg$ from the CD is determined according to constraints \eqref{eq:E1-Const11} - \eqref{eq:E1-Const12}. Here, constraints \eqref{eq:E1-Const11} - \eqref{eq:E1-Const12} ensure that each LD $d$ can start its trip $g$ only when the following two conditions hold: ($i$) all the packages assigned to LD-trip $dg$ are released (i.e., ready); and ($ii$) LD $d$ returns to the CD from its previous trips and swaps its battery at the CD.

Constraints \eqref{eq:E1-Const15} - \eqref{eq:E1-Const17} compute the package weight each LD-trip $dg$ carries while traveling on each arc within echelon 1 of the network. For each LD-trip $dg$, constraints \eqref{eq:E1-Const16} compute the total package weight while departing the CD, whereas constraints \eqref{eq:E1-Const17} maintain the flow conservation of package weights at each LH $h$. In constraints \eqref{eq:E1-Const15} - \eqref{eq:E1-Const17}, $M_2$ must be greater than $P_{max}^{LD}$. Furthermore, constraints \eqref{eq:E1-Const7} ensure that the total package weight LD $d$ carries on its trip $g$ does not exceed the PWCC of the LDs.

Each constraint \eqref{eq:E1-Const18} computes the energy consumption of an LD-trip $dg$ during travel through arc $(h,k)\in\mathcal E$. The energy consumption on each arc comprises: ($i$) the base energy consumption of an LD to travel on that arc, accounting for the drone type, drone speed, and travel distance; and ($ii$) the increment in the energy consumption with respect to the weight of the packages the drone carries on that arc. Furthermore, constraints \eqref{eq:E1-Const19} ensure that the total energy consumption in each LD-trip does not exceed the available battery capacity of the LDs (i.e., $B_{max}^{LD} - B_{min}^{LD}$). This ensures that the LDs have sufficient remaining energy in their batteries upon returning to the CD. Constraints \eqref{eq:E1-Const20} define the restrictions on the binary and continuous decision variables.

\vspace{-10pt}
\paragraph{Interconnection Between the Two Echelons:}

\begin{align}
    & \beta_{ph} = \sum_{d\in\mathcal D} \sum_{g\in\mathcal G^{LD}} w_{ph}^{dg}, \quad \forall p\in\mathcal P,\, h\in\mathcal H \label{eq:Int-Const1}\\
    & at_h^{dg} - M_1 (1 - w_{ph}^{dg}) \leq \pi_p \leq at_h^{dg} + M_1 (1 - w_{ph}^{dg}), \quad \forall p\in\mathcal P,\, h\in\mathcal H,\, d\in\mathcal D,\, g\in\mathcal G^{LD}  \label{eq:Int-Const2}\\
    & \beta_{ph} \in \{0, 1\}, \quad \pi_p \in \mathbb{R}^+
\end{align}

\vspace{-5pt}Constraints \eqref{eq:Int-Const1} determine the delivery destination of each package $p\in\mathcal P$ in echelon 1 of the network. Here, $\beta_{ph}=1$ represents that package $p$ is delivered to LH $h\in\mathcal H$. Furthermore, constraints \eqref{eq:Int-Const2} compute the delivery time of each package to an LH in layer 2 of the network. Therefore, $\beta_{ph}$ and $\pi_p$ provide the availability information of each package for echelon 2 of the network. $\beta_{ph}$ represents the origin of package $p$, and $\pi_p$ determines the earliest time that package $p$ is available for the last-mile delivery operations in echelon 2. 

\vspace{-10pt}
\paragraph{Opening the ALs and Allocating the AMBs to the ALs:}

\begin{align}
    & \sum_{m\in\mathcal M} \gamma_{ml} = \alpha_l, \quad \forall l\in\mathcal L \label{eq:LocAlloc-Const1}\\
    & \sum_{l\in\mathcal L} \gamma_{ml} = 1, \quad \forall m\in\mathcal M \label{eq:LocAlloc-Const2} \\
    & \alpha_l, \gamma_{ml} \in \{0, 1\}
\end{align}

\vspace{-5pt}Constraints \eqref{eq:LocAlloc-Const1} determine AL opening and ensure that exactly one AMB is allocated to each opened AL. Constraints \eqref{eq:LocAlloc-Const2} ensure that each AMB is allocated to exactly one AL.

\vspace{-10pt}
\paragraph{Echelon 2:}

\begin{align}
    & \sum_{i\in\mathcal I} \theta_{hi}^t = \sum_{j\in\mathcal I} \theta_{jh'}^t, \quad \forall h\in\mathcal H,\, t\in\mathcal T_h \label{E2-Const1}\\
    & \sum_{i\in\mathcal I \cup \{h\}} \theta_{ij}^t = \sum_{i\in\mathcal I \cup \{h'\}} \theta_{ji}^t = r_{jh}^t, \quad \forall h\in\mathcal H,\, t\in\mathcal T_h ,\, j\in\mathcal I  \label{E2-Const2} \\
    & \sum_{l\in\mathcal L} \theta_{hl}^t = \sum_{o\in\mathcal L} \theta_{oh'}^t, \quad \forall h\in\mathcal H,\, t\in\mathcal T_h \label{E2-Const3}\\
    & \sum_{l\in\mathcal L \cup \{h\}} \theta_{lo}^t = \sum_{l\in\mathcal L \cup \{h'\}} \theta_{ol}^t = r_{oh}^t, \quad \forall h\in\mathcal H,\, t\in\mathcal T_h,\, o\in\mathcal L \label{E2-Const4}\\
    & \sum_{i\in\mathcal I} z_{hi}^{vg'} = \sum_{j\in\mathcal I} z_{jh'}^{vg'} = \mu_h^{vg'}, \quad \forall h\in\mathcal H,\, v\in\mathcal V_h, \, g'\in\mathcal G^{MD} \label{E2-Const5}\\
    & \sum_{i\in\mathcal I \cup \{h\}} z_{ij}^{vg'} = \sum_{i\in\mathcal I \cup \{h'\}} z_{ji}^{vg'} = y_{jh}^{vg'}, \quad \forall h\in\mathcal H,\, j\in\mathcal I, \, v\in\mathcal V_h,\, g'\in\mathcal G^{MD} \label{E2-Const6}\\
    & \sum_{l\in\mathcal L} z_{hl}^{vg'} = \sum_{o\in\mathcal L} z_{oh'}^{vg'} = \mu_h^{vg'}, \quad \forall h\in\mathcal H,\, v\in\mathcal V_h, \, g'\in\mathcal G^{MD} \label{E2-Const7}\\
    & \sum_{l\in\mathcal L \cup \{h\}} z_{lo}^{vg'} = \sum_{l\in\mathcal L \cup \{h'\}} z_{ol}^{vg'} = y_{oh}^{vg'}, \quad \forall h\in\mathcal H,\, o\in\mathcal L, \, v\in\mathcal V_h,\, g'\in\mathcal G^{MD} \label{E2-Const8} \\
    & \sum_{h\in\mathcal H}\sum_{t\in\mathcal T_h} r_{ih}^t + \sum_{h\in\mathcal H} \sum_{v\in\mathcal V_h} \sum_{g'\in\mathcal G^{MD}} y_{ih}^{vg'} = 1,\quad \forall i\in\mathcal I \label{E2-Const9_0}\\
    & \sum_{h\in\mathcal H}\sum_{t\in\mathcal T_h} r_{lh}^t + \sum_{h\in\mathcal H} \sum_{v\in\mathcal V_h} \sum_{g'\in\mathcal G^{MD}} y_{lh}^{vg'} = \alpha_l,\quad \forall l\in\mathcal L \label{E2-Const9}\\
    & \sum_{t\in\mathcal T_h}r_{ih}^t + \sum_{v\in\mathcal V_h} \sum_{g'\in\mathcal G^{MD}}  y_{ih}^{vg'} = \beta_{ih}, \quad \forall i\in\mathcal I,\, h\in\mathcal H \label{E2-Const10}\\
    & \sum_{t\in\mathcal T_h}r_{lh}^t + \sum_{v\in\mathcal V_h} \sum_{g'\in\mathcal G^{MD}}  y_{lh}^{vg'} = \sum_{m\in\mathcal M} \beta_{mh} \cdot \gamma_{ml}, \quad \forall l\in\mathcal L,\, h\in\mathcal H \label{E2-Const11} \\
    & y_{ih}^{vg'} \leq \mu_h^{vg'}, \quad \forall h\in\mathcal H,\, v\in\mathcal V_h,\, g'\in\mathcal G^{MD},\, i\in\mathcal I \label{E2-Const16}\\
    & y_{lh}^{vg'} \leq \mu_h^{vg'}, \quad \forall h\in\mathcal H,\, v\in\mathcal V_h,\, g'\in\mathcal G^{MD},\, l\in\mathcal L \label{E2-Const17} \\
    & \mu_h^{vg'} \leq \mu_h^v, \quad \forall h\in\mathcal H,\, v\in\mathcal V_h,\, g'\in\mathcal G^{MD} \label{E2-Const18} \\
    & P_i \cdot r_{ih}^t \leq P_{max}^t, \quad \forall h\in\mathcal H,\, t\in\mathcal T_h,\, i \in \mathcal I \label{E2-Const19}\\
    & P_m \cdot r_{lh}^t \leq P_{max}^t + M (1 - \gamma_{ml}), \quad \forall h\in\mathcal H,\, t\in\mathcal T_h,\, l \in \mathcal L,\, m\in\mathcal M \label{E2-Const20} \\
    & \sum_{i\in\mathcal I} P_i \cdot y_{ih}^{vg'} \leq P_{max}^{MD}, \quad h\in\mathcal H,\, \forall v\in\mathcal V_h,\, g'\in\mathcal G^{MD} \label{E2-Const21}\\
    & \sum_{l\in\mathcal L} \sum_{m\in\mathcal M}P_m \cdot \gamma_{ml} \cdot y_{lh}^{vg'} \leq P_{max}^{MD}, \quad h\in\mathcal H,\, \forall v\in\mathcal V_h,\, g'\in\mathcal G^{MD} \label{E2-Const22} \\
    & \sum_{i\in\mathcal I} y_{ih}^{vg'} \leq P_N^{MD}, \quad h\in\mathcal H,\, v\in\mathcal V_h,\, g'\in\mathcal G^{MD} \label{E2-Const23}\\
    & \sum_{l\in\mathcal L} y_{lh}^{vg'} \leq P_N^{MD}, \quad h\in\mathcal H,\, v\in\mathcal V_h,\, g'\in\mathcal G^{MD} \label{E2-Const23-2} \\
    & \delta_j \geq \pi_j + T_{bat}^t \cdot s_{ih}^t + T_L^t + T_U^t + T_{hj}^t - M_1 (1- \theta_{ij}^t), \quad \forall h\in\mathcal H,\, t\in\mathcal T_h,\, i\in\mathcal I\cup h, \, j\in\mathcal I, \, i \neq j \label{E2-Const24}\\
    & \delta_m \geq \pi_m + T_{bat}^t \cdot s_{oh}^t + T_L^t + T_U^t + T_{hl}^t - M_1 (1- \theta_{ol}^t) - M_1 (1- \gamma_{ml}) - M_1 (1 - \gamma_{no}), \label{E2-Const26}\\
    & \qquad \quad \forall h\in\mathcal H,\, t\in\mathcal T_h,\,  l\in\mathcal L, \, o\in\mathcal L \cup h, \, l \neq o,\, m,n \in\mathcal M,\, m \neq n  \nonumber \\
    & \delta_j \geq \delta_i + T_{bat}^t \cdot s_{ih}^t + T_L^t + T_U^t + T_{ih}^t + T_{hj}^t - M_1 (1 - \theta_{ij}^t), \quad \forall h\in\mathcal H,\, t\in\mathcal T_h,\,  i,j\in\mathcal I,\, i \neq j \label{E2-Const25}\\
    & \delta_n \geq \delta_m + T_{bat}^t \cdot s_{lh}^t + T_L^t + T_U^t + T_{lh}^t + T_{ho}^t - M_1 (1 - \theta_{lo}^t) - M_1 (1 - \gamma_{ml}) - M_1 (1 - \gamma_{no}), \label{E2-Const27}\\
    & \qquad \quad \forall h\in\mathcal H,\, t\in\mathcal T_h,\,  l,o\in\mathcal L,\, l \neq o,\, m,n \in\mathcal M,\, m \neq n \nonumber \\
    & \delta_j \geq pt_h^{vg'}  + T_L^{MD} + T_U^{MD} + T_{hj}^{MD} - M_1 (1- z_{hj}^{vg'}), \quad \forall h\in\mathcal H,\, j\in\mathcal I,\, v\in\mathcal V_h,\, g'\in\mathcal G^{MD} \label{E2-Const28}\\
    & \delta_m \geq pt_h^{vg'} + T_L^{MD} + T_U^{MD} + T_{hl}^{MD} - M_1 (1- z_{hl}^{vg'}) - M_1 (1 - \gamma_{ml}), \label{E2-Const29}\\
    & \qquad \quad \forall h\in\mathcal H,\, l\in\mathcal L,\, m\in\mathcal M,\, v\in\mathcal V_h,\, g'\in\mathcal G^{MD} \nonumber \\
    & \delta_j \geq \delta_i + T_U^{MD} + T_{ij}^{MD} - M_1 (1 - z_{ij}^{vg'}), \quad \forall h\in\mathcal H,\,  i,j\in\mathcal I,\, i \neq j,\, v\in\mathcal V_h,\, g'\in\mathcal G^{MD} \label{E2-Const30}\\
    & \delta_n \geq \delta_m + T_U^{MD} + T_{lo}^{MD} - M_1 (1 - z_{lo}^{vg'}) - M_1 (1 - \gamma_{ml}) - M_1 (1 - \gamma_{no}), \label{E2-Const31}\\
    & \qquad  \quad \forall h\in\mathcal H,\,  l,o\in\mathcal L,\, l \neq o,\, m,n \in\mathcal M,\, m \neq n,\, v\in\mathcal V_h,\, g'\in\mathcal G^{MD} \nonumber \\
    & bt_h^{vg'} \geq \delta_i + T_{ih'}^{MD} - M_1 (1 - z_{ih'}^{vg'}), \quad \forall h\in\mathcal H,\, i\in\mathcal I,\, v\in\mathcal V_h,\, g'\in\mathcal G^{MD} \label{E2-Const32}\\
    & bt_h^{vg'} \geq \delta_m + T_{lh'}^{MD} - M_1 (1 - z_{lh'}^{vg'}) - M_1 (1 - \gamma_{ml}), \label{E2-Const33}\\
    & \qquad \quad \forall h\in\mathcal H,\, l\in\mathcal L,\, m\in\mathcal M,\, v\in\mathcal V_h,\, g'\in\mathcal G^{MD}  \nonumber \\
    & pt_h^{vg'}\geq \pi_i - M_1 (1-y_{ih}^{vg'}), \quad \forall h\in\mathcal H,\, i\in\mathcal I,\, v\in\mathcal V_h,\, g'\in\mathcal G^{MD} \label{E2-Const34}\\
    & pt_h^{vg'}\geq \pi_m - M_1 (1-y_{lh}^{vg'}) - M_1 (1 - \gamma_{ml}), \quad \forall h\in\mathcal H,\, l\in\mathcal L,\, m\in\mathcal M,\, v\in\mathcal V_h,\, g'\in\mathcal G^{MD} \label{E2-Const35}\\
    & pt_h^{vg'} \geq bt_h^{v,g''} + T_{bat}^{MD}, \quad \forall h\in\mathcal H,\,  v\in\mathcal V_h,\, g', g''\in\mathcal G^{MD},\, g'' < g' \label{E2-Const36} \\
    & \delta_p \leq L_p, \quad \forall p\in\mathcal P \label{E2-Const37} \\
    & pw_{ij}^{vg'} \leq M_3 \cdot z_{ij}^{vg'}, \quad \forall h\in\mathcal H,\, (i,j)\in\mathcal A_h,\, v\in\mathcal V_h,\, g'\in\mathcal G^{MD} \label{E2-Const38}\\
    & pw_{lo}^{vg'} \leq M_3 \cdot z_{lo}^{vg'}, \quad \forall h\in\mathcal H,\, (l,o)\in\mathcal B_h,\, v\in\mathcal V_h,\, g'\in\mathcal G^{MD} \label{E2-Const39}\\
    & -M_3 (1- z_{hi}^{vg'}) \leq pw_{hi}^{vg'} - \sum_{j\in\mathcal I}P_j \cdot y_{jh}^{vg'} \leq M_3 (1- z_{hi}^{vg'}), \quad \forall h\in\mathcal H,\, i\in\mathcal I,\, v\in\mathcal V_h,\, g'\in\mathcal G^{MD} \label{E2-Const40}\\
    & -M_3 (1- z_{hl}^{vg'}) \leq pw_{hl}^{vg'} - \sum_{o\in\mathcal L} \sum_{m\in\mathcal M} P_m \cdot \gamma_{mo} \cdot y_{oh}^{vg'} \leq M_3 (1- z_{hl}^{vg'}), \label{E2-Const41}\\
    & \qquad \quad \forall h\in\mathcal H,\, l\in\mathcal L,\, v\in\mathcal V_h,\, g'\in\mathcal G^{MD} \nonumber \\
    & -M_3 (1- y_{ih}^{vg'}) \leq \sum_{j': (j',i)\in\mathcal A_h} pw_{j'i}^{vg'} - P_i \cdot y_{ih}^{vg'} - \sum_{j: (i,j)\in\mathcal A_h} pw_{ij}^{vg'} \leq M_3 (1- y_{ih}^{vg'}), \label{E2-Const42}\\
    & \qquad  \quad \forall h\in\mathcal H,\, i\in\mathcal I,\, v\in\mathcal V_h,\, g'\in\mathcal G^{MD} \nonumber\\
    & -M_3 (1- y_{lh}^{vg'}) \leq \sum_{o': (o',l)\in\mathcal B_h} pw_{o'l}^{vg'} - \sum_{m\in\mathcal M} P_m \cdot \gamma_{ml} \cdot y_{lh}^{vg'} - \sum_{o: (l,o)\in\mathcal A_h} pw_{lo}^{vg'} \leq M_3 (1- y_{lh}^{vg'}), \label{E2-Const43}\\
    & \qquad  \quad \forall h\in\mathcal H,\, l\in\mathcal L,\, v\in\mathcal V_h,\, g'\in\mathcal G^{MD} \nonumber \\
    & q_{ij}^{vg'} = E_{ij}^{0,MD} \cdot z_{ij}^{vg'} + E_{ij}^{1,MD} \cdot pw_{ij}^{vg'}, \quad \forall h\in\mathcal H,\, (i,j)\in\mathcal A_h,\, v\in\mathcal V_h,\, g'\in\mathcal G^{MD} \label{E2-Const44}\\
    & q_{lo}^{vg'} = E_{lo}^{0, MD} \cdot z_{lo}^{vg'} + E_{lo}^{1, MD} \cdot pw_{lo}^{vg'}, \quad \forall h\in\mathcal H,\, (l,o)\in\mathcal B_h,\, v\in\mathcal V_h,\, g'\in\mathcal G^{MD} \label{E2-Const45}\\
    & \sum_{(i,j)\in\mathcal A_h} q_{ij}^{vg'} \leq B_{max}^{MD} - B_{min}^{MD}, \quad \forall h\in\mathcal H,\, v\in\mathcal V_h,\, g'\in\mathcal G^{MD} \label{E2-Const46}\\
    & \sum_{(l,o)\in\mathcal B_h} q_{lo}^{vg} \leq B_{max}^{MD} - B_{min}^{MD}, \quad \forall h\in\mathcal H,\, v\in\mathcal V_h,\, g'\in\mathcal G^{MD} \label{E2-Const47} \\
    & \rho_{ih}^t = \Big(2E_{ih}^{0,t} + E_{ih}^{1,t} \cdot P_i \Big) r_{ih}^t, \quad \forall h\in\mathcal H,\, t\in\mathcal T_h,\, i\in\mathcal I \label{E2-Const48}\\
    & \rho_{lh}^t = \Big(2E_{lh}^{0,t} + E_{lh}^{1,t} \cdot \sum_{m\in\mathcal M}P_m \cdot \gamma_{ml} \Big) r_{lh}^t, \quad \forall h\in\mathcal H,\, t\in\mathcal T_h,\, l\in\mathcal L \label{E2-Const49} \\
    & r_{ih}^t, r_{lh}^t, \theta_{ij}^t, \theta_{lo}^t, \mu_h^v, \mu_h^{vg'}, y_{ih}^{vg'}, y_{lh}^{vg'}, z_{ij}^{vg'}, z_{lo}^{vg'}, s_{ih}^t, s_{lh}^t \in \{0,1\} \label{E2-Const64}\\
    & \delta_p, bt_h^{vg'}, pt_h^{vg'}, \rho_{ih}^t, \rho_{lh}^t, pw_{ij}^{vg'}, pw_{lo}^{vg'}, q_{ij}^{vg'}, q_{lo}^{vg'} \in \mathbb{R}^{+} \label{E2-Const66} \\
    & \eqref{eq:End-CL-1} - \eqref{eq:End-AMB-7} \label{E2-Const65}
\end{align}

\vspace{-5pt}We model the multi-trip direct delivery schedule of the SDs as a VRP, where we leverage an \enquote{extended route} concept, and use arc flow variables $\theta_{ij}^t$ for each SD type $t$ (refer to \citet{bhuiyan2024aerial} for a detailed exposition of the extended route concept for the multi-trip direct delivery schedule of drones). Using this extended route concept, constraints \eqref{E2-Const1} - \eqref{E2-Const4} and \eqref{E2-Const5} - \eqref{E2-Const8} ensure the flow balance of the SDs and the MDs, respectively, in echelon 2 of the network. As explained in Section \ref{sec:Problem_Description}, each drone in echelon 2 either visits a set of CLs or a set of ALs across its multiple trips. Therefore, we consider separate flow balance constraints for the CLs and the ALs. Specifically, constraints \eqref{E2-Const1} - \eqref{E2-Const2} and \eqref{E2-Const5} - \eqref{E2-Const6} correspond to the CLs, while constraints \eqref{E2-Const3} - \eqref{E2-Const4} and \eqref{E2-Const7} - \eqref{E2-Const8} correspond to the ALs.

Constraints \eqref{E2-Const9_0} ensure that each CL is visited exactly once by a drone. Similarly, if AL $l\in\mathcal L$ is opened for serving the AMBs, AL $l$ must be visited exactly once, as enforced by constraints \eqref{E2-Const9}. Furthermore, constraints \eqref{E2-Const10} - \eqref{E2-Const11} ensure that each CL/AL is served by a single drone from LH $h\in\mathcal H$ if and only if the corresponding package is available at LH $h$. In constraints \eqref{E2-Const11}, the bilinear term $(\beta_{mh} \cdot \gamma_{ml})$ determines whether package $m\in\mathcal M$ is available at LH $h$, and AMB $m$ is allocated to AL $l$. We linearize this bilinear term using the McCormick linearization technique. Constraints \eqref{E2-Const16} - \eqref{E2-Const17} ensure that CL $i$ or AL $l$ can be served via trip $g'\in\mathcal G^{MD}$ of MD $v\in\mathcal V_h$ from LH $h$ only if MD $v$ performs its trip $g'$ from LH $h$. For brevity, we refer to trip $g'$ of MD $v$ as \enquote{MD-trip} $vg'$ throughout the remainder of the paper. Furthermore, MD $v$ can perform its trip $g'$ only if MD $v$ is used in echelon 2, enforced by constraints \eqref{E2-Const18}.

Constraints \eqref{E2-Const19} - \eqref{E2-Const20} ensure that, if the package weight exceeds the PWCC of an SD of type $t$, the corresponding CL/AL cannot be served by this drone type. Constraints \eqref{E2-Const20} also account for the allocation of the AMBs to the ALs, where $P_m \cdot r_{lh}^t \leq P_{max}^t + M (1 - \gamma_{ml})$ is the linear representation of $P_m \cdot r_{lh}^t \cdot \gamma_{ml} \leq P_{max}^t$. Here, $r_{lh}^t \cdot \gamma_{ml} = 1$ represents that the package for AMB $m$ is delivered to AL $l$ by an SD of type $t$.
Constraints \eqref{E2-Const21} - \eqref{E2-Const22} ensure that the total package weight in MD-trip $vg'$ does not exceed the PWCC of the MDs. In constraints \eqref{E2-Const22}, $(\gamma_{ml} \cdot y_{lh}^{vg'})$ is a bilinear term representing whether AMB $m$ is assigned to AL $l$, and MD-trip $vg'$ visits AL $l$ from LH $h$. We linearize this bilinear term using the McCormick linearization technique. Constraints \eqref{E2-Const23} - \eqref{E2-Const23-2} enforce that the number of packages each MD-trip carries does not exceed the spatial capacity of the MDs, $P_N^{MD}$.

Constraints \eqref{E2-Const24} - \eqref{E2-Const27} determine the earliest possible delivery time of the packages carried by the SDs. As SDs perform direct deliveries, in constraints \eqref{E2-Const24} and \eqref{E2-Const26}, the earliest possible delivery time of a package from LH $h$ by an SD of type $t$ accounts for: ($i$) the earliest time the package is available in layer 2 (i.e., the LHs), denoted by $\pi_j$ for CL $j$ and $\pi_m$ for AMB $m$; ($ii$) the time it takes to load and unload the package, denoted by $T_L^t + T_U^t$; ($iii$) the time it takes to swap the battery of an SD of type $t$, if necessary, denoted by $T_{bat}^t$; and ($iv$) the time it takes for an SD of type $t$ to travel from LH $h$ to the corresponding destination, denoted by $T_{hj}^t$ and $T_{hl}^t$ for the CLs and the ALs, respectively. Moreover, constraints \eqref{E2-Const25} and \eqref{E2-Const27} compute the earliest time that consecutive packages carried by the SDs can be delivered to their corresponding destinations. Here, in addition to the loading, unloading, and the battery swapping time, the delivery time of the successive packages in the delivery sequence of each SD of type $t$ accounts for: ($i$) the delivery time of the preceding CL/AMB; and ($ii$) the time it takes for an SD of type $t$ to return to LH $h$ from the preceding CL/AMB, and travel to the subsequent CL/AMB from LH $h$.

Constraints \eqref{E2-Const28} - \eqref{E2-Const36} determine the earliest possible time that the packages carried by the MDs can be delivered to the corresponding delivery destinations (i.e., CLs and AMBs). Unlike SDs, MDs perform circular delivery, carrying multiple packages in each trip. Therefore, constraints \eqref{E2-Const28} and \eqref{E2-Const29} compute the earliest possible delivery time for the first CL and AMB, respectively, in the delivery sequence of MD-trip $vg'$. Constraints \eqref{E2-Const28} and \eqref{E2-Const29} ensure that the delivery time of CL $j$ and AMB $m$---the first delivery destination in the delivery sequence---accounts for: ($i$) the time MD-trip $vg'$ departs LH $h$, denoted by $pt_h^{vg'}$; ($ii$) the time it takes to load the packages, denoted by $T_L^{MD}$; ($iii$) the time it takes for MD $v$ to travel from LH $h$ to the corresponding delivery destination, denoted by $T_{hj}^{MD}$ and $T_{hl}^{MD}$ for the CLs and the ALs, respectively; and ($iv$) the time it takes for an MD to unload the package to the corresponding destination, denoted by $T_U^{MD}$. Furthermore, constraints \eqref{E2-Const30} and \eqref{E2-Const31} compute the earliest possible delivery time of the successive CLs/AMBs in the delivery sequence of each MD-trip $vg'$. As MDs perform circular delivery, the delivery time of the subsequent packages only accounts for: ($i$) the delivery time of the preceding CL/AMB in the delivery sequence, ($ii$) the direct travel time between the two consecutive delivery locations, and ($iii$) the unloading time of the MDs. 

To ensure the time consistency between the two consecutive trips of each MD, constraints \eqref{E2-Const32} and \eqref{E2-Const33} compute the earliest time MD $v$ concludes its trip $g'$ and returns to LH $h$. Then constraints \eqref{E2-Const34} - \eqref{E2-Const36} determine the earliest time MD $v$ can depart LH $h$ for its trip $g'$. Here, constraints \eqref{E2-Const34} and \eqref{E2-Const35} ensure that the departure time of MD-trip $vg'$ from LH $h$ occurs after all the packages assigned to MD-trip $vg'$ are available at LH $h$. Moreover, constraints \eqref{E2-Const36} ensure that MD $v$ departs LH $h$ for its trip $g'$ after MD $v$ returns to LH $h$ from its previous trips and swaps its battery at LH $h$.
Constraints \eqref{E2-Const37} ensure that each package $p$ is delivered to its corresponding delivery destination before its delivery due time, $L_p$.

Constraints \eqref{E2-Const38} - \eqref{E2-Const43} compute the package weight each MD-trip $vg'$ carries. Here, constraints \eqref{E2-Const40} - \eqref{E2-Const41} compute the total package weight MD-trip $vg'$ carries while departing LH $h$. Moreover, similar to constraints \eqref{eq:E1-Const17}, constraints \eqref{E2-Const42} - \eqref{E2-Const43} ensure the flow conservation of package weights across intermediate delivery locations (i.e., CLs and ALs). Here, $M_3$ needs to be greater than $P_{max}^{MD}$. Constraints \eqref{E2-Const44} - \eqref{E2-Const45} compute the energy consumption of the MDs while traveling on each arc within echelon 2 of the network. Then constraints \eqref{E2-Const46} - \eqref{E2-Const47} ensure that the MDs maintain the minimum required energy (i.e., $B_{min}^{MD}$) upon returning to the LHs.

Unlike MDs, SDs perform direct deliveries by serving each delivery location and returning to the LH at each trip. Therefore, constraints \eqref{E2-Const48} and \eqref{E2-Const49} compute the total energy each SD of type $t$ consumes while serving a CL and an AMB, respectively. Here, the terms $2E_{ih}^{0,t}$ and $2E_{lh}^{0,t}$ account for returning empty to the LH after delivering a package to the corresponding CL/AL. Moreover, we linearize the bilinear term $(\gamma_{ml} \cdot r_{lh}^t)$ using the McCormick linearization technique. Constraints \eqref{E2-Const64} and \eqref{E2-Const66} define the restrictions on the binary and continuous decision variables, respectively.

As explained in Section \ref{sec:Problem_Description}, in contrast to the MDs that swap their batteries each time they return to the LHs, SDs swap their batteries on an as-needed fashion (i.e., only when needed). For brevity, the mathematical expressions of the endogenous battery swapping of the SDs are presented in Eqs. \eqref{eq:End-CL-1} - \eqref{eq:End-AMB-7} in Appendix~\ref{Appendix_Endogenous}. These constraints compute the potential remaining energy in the drone batteries, and determine the battery swapping decisions for the SDs (i.e., $s_{ih}^t$ and $s_{lh}^t$).

\vspace{-10pt}
\section{Solution Method}\label{sec:Solution_Method}

Our \textit{HDDNM} problem is NP-hard, as proved in Proposition \ref{prop:NP-hard}. To efficiently solve this computationally difficult problem, we propose a customized exact solution method integrating the following set of enhancement strategies: ($i$) pre-processing and search space reduction, ($ii$) problem-specific reformulations, ($iii$) valid inequalities, and ($iv$) dynamic cutting planes generation. We also develop a heuristic algorithm, which leverages a simplified variant of the \textit{HDDNM} problem to provide good-quality solutions within a reasonable runtime.

\begin{prop}\label{prop:NP-hard}
    The \textit{HDDNM} problem is NP-hard.
\end{prop}
\vspace{-5pt}
The proofs of Propositions \ref{prop:NP-hard}, \ref{prop:Return_Time}, \ref{prop:Symmetry-Breaking}, \ref{prop:Last_Hub}, and \ref{prop:Conservative} stated in this section are provided in Appendix~\ref{Appendix_Proof}.

\vspace{-5pt}
\subsection{Pre-Processing and Search Space Reduction}\label{subsec:preprocess}

To improve the computational efficiency in solving the \textit{BaM} (i.e., Eqs. \eqref{eq:obj} - \eqref{E2-Const65}), we apply a set of pre-processing procedures to reduce the problem size (i.e., number of variables and constraints). 

\vspace{-5pt}
\subsubsection{Restricted Sets}\label{subsubsec:Restricted_Sets}

The physical characteristics of each drone type---speed, battery capacity, and PWCC---limit its ability to deliver certain packages from some LHs. Therefore, we define $\mathcal I_h^{a} = \{ i\in\mathcal I: P_i \leq P_{max}^{a}, \; 2E_{ih}^{0, a} + E_{ih}^{1, a} P_i \leq B_{max}^{a} - B_{min}^{a}, \; T_L^{a}+T_{hi}^{a}+T_U^{a}+T_L^{LD}+T_{0h}^{LD}+T_U^{LD} \leq L_i - R_i\}$ and $\mathcal L_h^{a} = \{ l\in\mathcal L: \bar{P}_m \leq P_{max}^{a}, \; 2E_{lh}^{0, a} + E_{lh}^{1, a} \bar{P}_m \leq B_{max}^{a} - B_{min}^{a}, \; T_L^{a}+T_{hl}^{a}+T_U^{a}+T_L^{LD}+T_{0h}^{LD}+T_U^{LD} \leq \max_{m\in\mathcal M} (L_m - R_m)\}$ as the restricted sets of CLs and ALs, respectively, that can be feasibly served by drone type $a\in \{\mathcal T_h\cup MD\}$ from LH $h$. A detailed explanation for these two restricted sets is provided in Appendix~\ref{Appendix_Restricted_Sets}. Using these restricted sets $\mathcal I_h^{a}$ and $\mathcal L_h^{a}$, we then define the restricted sets of arcs in echelon 2 as follows: $\mathcal A_h^{a} = \{(h,i):i\in\mathcal I_h^{a} \} \cup \{(i,j) \in\mathcal I_h^{a} \times \mathcal I_h^{a}: i \neq j \} \cup \{(i,h'): i\in\mathcal I_h^{a} \}$, and $\mathcal B_h^{a} = \{(h,l):l\in\mathcal L_h^{a} \} \cup \{(l,o) \in\mathcal L_h^{a} \times \mathcal L_h^{a}: l \neq o \} \cup \{(l,h'): l\in\mathcal L_h^{a} \}$. We directly incorporate these sets (i.e., $\mathcal I_h^a$, $\mathcal L_h^a$, $\mathcal A_h^a$, and $\mathcal B_h^a$) in constraints \eqref{E2-Const1} - \eqref{E2-Const49} to reduce the number of constraints.

\vspace{-5pt}
\subsubsection{Variable Fixings}\label{subsubsec:variable_fixing}

To reduce the number of variables, we introduce the following problem-specific variable fixings, which reduces the solution space, thus improving the computational efficiency. 

\vspace{-10pt}
\paragraph{Infeasible Hub Destinations:}

For each CL $i\in\mathcal I$ and LH $h\in\mathcal H$, if $i\notin \cup_{a\in \{\mathcal T_h\cup MD\}} \mathcal I_h^a$, i.e., delivering package $i$ from LH $h$ is not possible with any drone in echelon 2, we set $\beta_{ih}=0$ and $w_{ih}^{dg}=0, \forall d\in\mathcal D,\, g\in\mathcal G^{LD}$. We note that, as the AMB locations are not pre-determined delivery locations, this variable fixing does not apply to the ALs and the AMBs.

\vspace{-10pt}
\paragraph{Infeasible Hub Origins:}

If delivering package $i$ from LH $h$ to the corresponding CL is not possible with drone type $a\in \{\mathcal T_h\cup MD\}$, LH $h$ cannot be the origin of package $i$ in echelon 2 for drone type $a$. Therefore, we apply variable fixings \eqref{VF-HO-CL-1} - \eqref{VF-HO-AL-4} in Appendix~\ref{Appendix_VF}.

\vspace{-10pt}
\paragraph{Infeasible CL Pairs:}

We also identify CL pairs $(i,j)$ that cannot be feasibly served consecutively by a single drone, using Algorithm~\ref{alg:Feasibility_Check} in Appendix~\ref{Appendix_Algorithms}. For each drone type $a\in \{\mathcal T_h\cup MD\}$ and LH $h\in\mathcal H$, we compute the earliest possible delivery times of each CL pair $(i,j)\in\mathcal A_h^a$ in a given path $h \to i \to j \to h'$. If the earliest possible delivery times violate the delivery due times and/or the energy consumption exceeds the drone's battery capacity, CL pair $(i,j)$ is marked as infeasible for drone type $a$ from LH $h$. We then fix $\theta_{ij}^t=0$ and/or $z_{ij}^{vg'}=0, \forall v\in\mathcal V_h,\, g'\in\mathcal G^{MD}$, if CL pair $(i,j)$ cannot be feasibly served consecutively with an SD of type $t$ and/or an MD, respectively, from LH $h$. This variable fixing is also only applicable to the CLs.

\vspace{-5pt}
\subsection{Problem-Specific Reformulations}\label{subsec:reformulation}

We introduce new reformulations to replace selected big-M constraints of the \textit{BaM} (i.e., Eqs. \eqref{eq:obj} - \eqref{E2-Const65}) with big-M-free equations, which tighten the LP-relaxation.

\vspace{-5pt}
\subsubsection{Return Time Reformulation}

As discussed in Section \ref{sec:Problem_Description}, LDs and MDs perform multiple circular delivery trips, where constraints \eqref{eq:E1-Const10} and \eqref{E2-Const32} of the \textit{BaM} determine the return time of each trip of the LDs and the MDs, respectively. However, these constraints include big-M values resulting in a weaker LP-relaxation.

\vspace{-5pt}
\begin{prop} \label{prop:Return_Time}
    For any drone performing multiple trips, a feasible schedule in which the drone returns to either the CD (in echelon 1) or the origin LH (in echelon 2) immediately after serving all the locations in the sequence dominates an identical schedule with an idle time during the trip.
\end{prop}

\vspace{-5pt}
According to Proposition \ref{prop:Return_Time}, the return time of any drone-trip can be expressed as the earliest possible completion time of the trip, without loss of optimality. Therefore, for the LDs in echelon 1, we reformulate constraints \eqref{eq:E1-Const10} to constraints \eqref{eq:E1-Const10-New}. Similarly, for the MDs in echelon 2, we reformulate constraints~\eqref{E2-Const32} to constraints~\eqref{eq:E2-Const32-New}. Here, $T_L^{LD} + \sum_{(h,k)\in\mathcal A, k \neq S} (T_{hk}^{LD} + T_U^{LD}) \cdot x_{hk}^{dg}$ and $T_L^{MD} + \sum_{(i, j)\in\mathcal A_h, j \neq h'} (T_{ij}^{MD} + T_U^{MD} ) \cdot z_{ij}^{v g'}$ compute the total completion time of LD-trip $dg$ and MD-trip $vg'$, respectively.

\begin{align}
    & rt_0^{dg} = dt_0^{dg} + T_L^{LD} + \sum_{(h,k)\in\mathcal E, k \neq S} (T_{hk}^{LD} + T_U^{LD}) \cdot x_{hk}^{dg} + \sum_{h\in\mathcal H} x_{hS}^{dg} \cdot T_{hS}^{LD}, \quad \forall d\in\mathcal D,\, g\in\mathcal G \label{eq:E1-Const10-New} \\
    & bt_h^{v g'} = pt_h^{v g'} + T_L^{MD} + \sum_{(i, j)\in\mathcal A_h, j \neq h'} (T_{ij}^{MD} + T_U^{MD} ) \cdot z_{ij}^{v g'} + \sum_{i\in\mathcal I} z_{i h'}^{vg'} \cdot T_{i h'}^{MD}, \label{eq:E2-Const32-New} \\
    &  \qquad \quad \forall h\in\mathcal H,\, v\in\mathcal V_h,\, g' \in \mathcal G^{MD}  \nonumber
\end{align}

\vspace{-20pt}\subsubsection{Network Flow Reformulation for Echelon 1}\label{subsubsec:NF-Reformulation}

In this section, we introduce a multi-commodity network flow formulation for the LDs in echelon 1. We first define binary variable $f_{hkp}^{dg}$, representing whether package (i.e., commodity) $p$ flows through arc $(h, k)\in\mathcal E$ by LD-trip $dg$. We then reformulate constraints~\eqref{eq:E1-Const15} - \eqref{eq:E1-Const7} to constraints~\eqref{NF-1} - \eqref{NF-6}.

\begin{align}
    & \sum_{h\in\mathcal H}f_{0hp}^{dg} = \sum_{h\in\mathcal H}w_{ph}^{dg}, \quad \forall p\in\mathcal P, \; d\in\mathcal D,\; g\in\mathcal G^{LD} \label{NF-1}\\
    & \sum_{k':(k',h)\in\mathcal E}f_{k'hp}^{dg} - \sum_{k:(h,k)\in\mathcal E}f_{hkp}^{dg} = w_{ph}^{dg}, \quad \forall h\in\mathcal H,\; p\in\mathcal P,\; d\in\mathcal D,\; g\in\mathcal G^{LD} \label{NF-2}\\
    & f_{hkp}^{dg} \leq x_{hk}^{dg}, \quad \forall p\in\mathcal P,\; (h,k)\in\mathcal E,\;d\in\mathcal D,\; g\in\mathcal G^{LD} \label{NF-3}\\
    & x_{hk}^{dg} \leq \sum_{p\in\mathcal P}f_{hkp}^{dg},\quad \forall (h,k)\in\mathcal E,\; d\in\mathcal D,\; g\in\mathcal G^{LD} \label{NF-4} \\
    & pw_{hk}^{dg} = \sum_{p\in\mathcal P} P_p \cdot f_{hkp}^{dg}, \quad \forall (h,k)\in\mathcal E,\; d\in\mathcal D,\; g\in\mathcal G^{LD} \label{NF-5}\\
    & pw_{hk}^{dg} \leq P_{max}^{LD}, \quad \forall (h,k)\in\mathcal E,\; d\in\mathcal D,\; g\in\mathcal G^{LD} \label{NF-6} 
\end{align}

Constraints \eqref{NF-1} compute the outgoing flow of the packages from the CD for each LD-trip $dg$. Constraints \eqref{NF-2} ensure the flow balance of the packages at each LH. Furthermore, constraints \eqref{NF-3} - \eqref{NF-4} ensure that if a package is carried over arc $(h, k)$ by LD-trip $dg$, this arc must be traversed by LD-trip $dg$. Constraints \eqref{NF-5} compute the package weight each LD-trip $dg$ carries on each arc $(h,k)$, whereas constraints \eqref{NF-6} ensure that the PWCC of the LDs is satisfied.

\vspace{-5pt}
\subsubsection{Commodity-Flow Reformulation for Payload Constraints in Echelon 2}

Introducing the network flow reformulation (described in Section~\ref{subsubsec:NF-Reformulation}) for the MDs in echelon 2 results in a significantly large number of decision variables. Therefore, we introduce a new reformulation technique for the MDs in echelon 2, where we reformulate big-M-based constraints \eqref{E2-Const40} - \eqref{E2-Const43} to big-M-free equations \eqref{PL-1} - \eqref{PL-4}. 

\begin{align}
    & \sum_{i\in\mathcal I}pw_{hi}^{vg'} = \sum_{j\in\mathcal I}P_j \cdot y_{jh}^{vg'}, \quad \forall h\in\mathcal H,\, v\in\mathcal V_h,\, g'\in\mathcal G^{MD} \label{PL-1}\\
    & \sum_{l\in\mathcal L}pw_{hl}^{vg'} = \sum_{o\in\mathcal L} \sum_{m\in\mathcal M}P_m \cdot \gamma_{mo} \cdot y_{oh}^{vg'}, \quad \forall h\in\mathcal H,\, v\in\mathcal V_h,\, g'\in\mathcal G^{MD} \label{PL-2}\\
    & \sum_{j': (j',i)\in\mathcal A_h} pw_{j'i}^{vg'} - \sum_{j: (i,j)\in\mathcal A_h} pw_{ij}^{vg'} = P_i \cdot y_{ih}^{vg'}, \quad \forall h\in\mathcal H,\, i\in\mathcal I,\, v\in\mathcal V_h,\, g'\in\mathcal G^{MD} \label{PL-3}\\
    & \sum_{o': (o',l)\in\mathcal B_h} pw_{o'l}^{vg'}  - \sum_{o: (l,o)\in\mathcal A_h} pw_{lo}^{vg'} = \sum_{m\in\mathcal M} P_m \cdot \gamma_{ml} \cdot y_{lh}^{vg'}, \quad \forall h\in\mathcal H,\, l\in\mathcal L,\, v\in\mathcal V_h,\, g'\in\mathcal G^{MD} \label{PL-4}
\end{align}

\vspace{-5pt}Constraints \eqref{PL-1} - \eqref{PL-2} compute the package weight each MD carries while departing an LH. Moreover, constraints \eqref{PL-3} - \eqref{PL-4} compute the package weight of the MDs along different arcs while serving the CLs and the AMBs, respectively.

\vspace{-5pt}
\subsection{Valid Inequalities}\label{subsec:valid-inequality}

We introduce several problem-specific valid inequalities to tighten the feasible region.

\vspace{-5pt}
\subsubsection{Symmetry-Breaking Constraints}

In our \textit{HDDNM} problem, the following two types of symmetry exist, which negatively affect the computational performance of the \textit{BaM} (i.e., Eqs. \eqref{eq:obj} - \eqref{E2-Const65}): ($i$) there are multiple drones of the same type, and ($ii$) each drone has multiple available identical trips to perform. Therefore, we introduce symmetry-breaking constraints \eqref{SB-1} - \eqref{SB-4}, which are valid inequalities that improve the computational efficiency of the \textit{BaM} (i.e., Eqs. \eqref{eq:obj} - \eqref{E2-Const65}), as stated in Proposition \ref{prop:Symmetry-Breaking}.
\vspace{-8pt}
\begin{align}
    & u^d \leq u^{d-1}, \quad \forall d\in\mathcal D,\, d > 1 \label{SB-1} \\
    & u^{dg} \leq u^{d, g-1}, \quad \forall d\in\mathcal D,\, g\in\mathcal G^{LD},\, g>1 \label{SB-2} \\
    & \mu_h^v \leq \mu_h^{v-1}, \quad \forall h\in\mathcal H,\, v\in\mathcal V_h, \, v>1 \label{SB-3} \\
    & \mu_h^{vg'} \leq \mu_h^{v, g'-1}, \quad \forall h\in\mathcal H,\, v\in\mathcal V_h,\, g'\in\mathcal G^{MD},\, g' > 1 \label{SB-4}
\end{align}

\vspace{-15pt}\begin{prop}\label{prop:Symmetry-Breaking}
    The \textit{BaM} model augmented with symmetry-breaking constraints \eqref{SB-1} - \eqref{SB-4} remains as a valid model, and its feasible set is a proper subset of that of \textit{BaM}.
\end{prop}

\vspace{-15pt}
\subsubsection{Echelon 1 Delivery Time}

Once a package is delivered to an LH in echelon 1, it needs to be transported to its corresponding final delivery destination in echelon 2. Although constraints~\eqref{E2-Const37} enforce the delivery due times of the packages at their corresponding delivery destinations in echelon 2, they do not directly restrict the delivery time of packages to the LHs in echelon 1. Accordingly, the delivery time of each package in echelon 1 must account for the minimum time required for the package to be delivered by a drone in echelon 2 to its final destination. To model this condition, we introduce valid inequalities \eqref{VI-1} and \eqref{VI-2} for the CLs and the AMBs, respectively.

\begin{align}
    & \pi_i \leq L_i - \beta_{ih} \cdot \min_{a\in \{\mathcal T_h \cup MD\}} \Big(T_L^{a} + T_U^{a} + T_{hi}^{a} \Big), \quad \forall i\in\mathcal I, \; h\in\mathcal H \label{VI-1} \\
    & \pi_m \leq L_m - \gamma_{ml} \cdot \beta_{mh} \cdot \min_{a\in \{\mathcal T_h \cup MD\}} \Big(T_L^{a} + T_U^{a} + T_{hl}^{a} \Big), \quad \forall m\in\mathcal M,\, l\in\mathcal L,\, h\in\mathcal H \label{VI-2}
\end{align}

\vspace{-20pt}
\subsection{Dynamic Cutting Planes}\label{subsec:separation}

In this section, we introduce a set of dynamic cutting plane generation procedures. These dynamic cuts eliminate fractional solutions that satisfy the relaxed constraints of the \textit{BaM} (i.e., Eqs. \eqref{eq:obj} - \eqref{E2-Const65}) but violate the integral feasibility. We dynamically detect and remove such \textit{LP-feasible, integer-infeasible} solutions using separation procedures as the algorithm proceeds through the branch-and-bound (B\&B) tree. These procedures gradually tighten the underlying LP-relaxation without explicitly enumerating all valid inequalities.

\vspace{-5pt}
\subsubsection{Capacity Cuts}\label{subsubsec:Cap_Cut}

In an LP-relaxed solution, binary variables may take fractional values (e.g., $0 \leq f_{hkp}^{dg} \leq 1$). Therefore, the computed payload, $pw_{hk}^{dg}=\sum_{p\in\mathcal P}P_p \cdot f_{hkp}^{dg}$, in an LP-relaxed solution may underestimate the actual weight of the packages that the fractional solution assigns to LD-trip $dg$. This can result in LP-feasible solutions, in which the assignment of the packages to the LD-trips violates the PWCC of the LDs. Let $C^{dg}$ be a subset of packages assigned to LD-trip $dg$ in a fractional solution, such that $\sum_{p\in C^{dg}} P_p > P_{max}^{LD}$. This subset of packages cannot be carried simultaneously by a single LD-trip. Therefore, to eliminate fractional solutions that assign all the packages in $C^{dg}$ to the same LD-trip, we introduce constraints~\eqref{Cap_Cut}.

\begin{equation} \label{Cap_Cut}
    \sum_{h\in\mathcal H} \sum_{p\in C^{dg}} w_{ph}^{dg} \leq u^{dg} \Big(|C^{dg}| - 1 \Big), \quad \forall d\in\mathcal D,\, g\in\mathcal G^{LD},\, C^{dg} \in \tilde{\mathcal C}
\end{equation}

\vspace{-5pt}Here, $\tilde{\mathcal C}$ is the set of all subsets (i.e., $C^{dg}$) of packages that violate the PWCC of the LDs. However, a full enumeration of $\tilde{\mathcal C}$ is computationally intractable. Therefore, we perform a separation procedure on the fractional solutions to dynamically update $\tilde{\mathcal C}$, and add constraints~\eqref{Cap_Cut} to the \textit{BaM} on an as-needed basis.
The exact separation procedure requires solving model~\eqref{Exact_Cap_Cut_Separation}, presented in Appendix~\ref{Appendix_Separation}, to find the maximally violated subset. However, finding the maximally violated subset by solving model~\eqref{Exact_Cap_Cut_Separation} is an NP-hard problem~\citep{kaparis2010separation}. Therefore, we develop a fast heuristic separation procedure to identify the violated subsets, as shown in Algorithm~\ref{alg:Capacity_Cut} in Appendix~\ref{Appendix_Algorithms}. Here, for each LD-trip $dg$, we first determine whether LD $d$ performs its trip $g$. We then identify the set of packages, $C^{dg}$, that have non-zero assignments to that LD-trip. If this set of packages violates the capacity cut~\eqref{Cap_Cut}, we add the corresponding cut to the \textit{BaM}.

\vspace{-5pt}
\subsubsection{Last Hub Cuts}

As discussed in Section \ref{sec:Problem_Description}, each LD-trip visits a sequence of LHs and delivers multiple packages to each visited LH. This operational characteristic specifies that the relationship between variables $x_{hk}^{dg}$, $f_{hkp}^{dg}$, and $w_{ph}^{dg}$ must follow special conditions in an integer-feasible solution, as stated in Proposition~\ref{prop:Last_Hub}. 

\begin{prop}\label{prop:Last_Hub}
    Let $\tilde{K}_h^{dg}:= \{p\in\mathcal P: \sum_{k\in\mathcal H \cup \{0\}} f_{khp}^{dg}=1\}$ and $\tilde{D}_h^{dg}:= \{p\in \tilde{K}_h^{dg}: w_{ph}^{dg}=1\}$ denote the set of packages carried and delivered to LH $h$, respectively, by LD-trip $dg$. Then, in every integer-feasible solution of the \textit{BaM}, the following two conditions hold: $(i)$ $\sum_{k\in\mathcal H} x_{hk}^{dg}=1 \Rightarrow \tilde{D}_h^{dg} \subset \tilde{K}_h^{dg}$; and $(ii)$ $\tilde{D}_h^{dg} = \tilde{K}_h^{dg} \Rightarrow \sum_{k\in\mathcal H} x_{hk}^{dg}=0,\, x_{hS}^{dg}=1,\, \sum_{k\in\mathcal H} \sum_{p\in\mathcal P} f_{hkp}^{dg}=0$.
\end{prop}

However, in an LP-relaxed solution, packages can be split fractionally across the LHs and the LD-trips. Therefore, the conditions presented in Proposition~\ref{prop:Last_Hub} may not necessarily hold in LP-relaxed solutions. To eliminate such \textit{LP-feasible, integer-infeasible} solutions, we introduce constraints~\eqref{eq:Last_Hub_Cut} for each $K_h^{dg}= \{p\in\mathcal P: \sum_{k\in\mathcal H \cup \{0\}} f_{khp}^{dg} > \epsilon\}$ and $D_h^{dg}= \{p\in K_h^{dg}: w_{ph}^{dg} > \epsilon\}$. Dynamically adding constraints~\eqref{eq:Last_Hub_Cut} to the \textit{BaM} requires a separation procedure. As the exact separation is computationally challenging, we propose a heuristic separation procedure, as shown in Algorithm~\ref{alg:LastHubCut} in Appendix~\ref{Appendix_Algorithms}. Here, for each LD-trip $dg$ and LH $h$, we first determine whether LH $h$ is visited by LD-trip $dg$. We then identify the subsets $K_h^{dg}$ and $D_h^{dg}$, and evaluate the two conditions presented in Proposition~\ref{prop:Last_Hub}. If these conditions are violated and the identified subsets violate the cut~\eqref{eq:Last_Hub_Cut} in the current fractional solution, we add the corresponding cut to the \textit{BaM}.

\begin{equation}\label{eq:Last_Hub_Cut}
    \sum_{k : (k,h) \in \mathcal A} \sum_{p \in K_h^{dg}} f_{khp}^{dg} - \sum_{p \in D_h^{dg}} w_{ph}^{dg} \geq \sum_{k \in \mathcal H} x_{hk}^{dg}, \quad \forall d\in\mathcal D,\, g\in\mathcal G^{LD},\, h\in\mathcal H
\end{equation}

\vspace{-15pt}
\subsection{A Fast Heuristic Solution Algorithm}\label{subsec:Heuristic}

In this section, we propose a fast heuristic algorithm to efficiently solve real-life large-scale instances of the \textit{HDDNM} problem while providing good-quality solutions. To that end, we first define a simplified variant of the \textit{HDDNM} problem as follows. We consider the hierarchical H\&S network described in Section~\ref{sec:Problem_Description} with a single commodity and identical packages, where all delivery destinations are fixed. We specify the demand at each CL as the required quantity (measured in weight) of the commodity that needs to be satisfied before the respective due time. The full supply is released at the CD at the beginning of the planning horizon, so that all LDs depart the CD simultaneously at a given time during the planning horizon. Only MDs operate in echelon 2, and both LDs and MDs perform single-trip circular deliveries, eliminating the need to consider battery swapping decisions. Moreover, we discretize the amounts of commodity carried by the LDs over fixed intervals. We also exclude the energy consumption costs from the objective function, effectively simplifying the problem's goal to minimizing the fleet size. Additionally, we assume that the departure of the MDs from each LH occurs after all LD deliveries to that LH are completed.

To efficiently solve this simplified problem variant, we propose a heuristic algorithm, as shown in Algorithm~\ref{alg:Heuristic}. Here, $\hat{P}_i$ is the discretized demand quantity (i.e., package weight) of CL $i$, and $T_E$ is the earliest time that all the LDs depart the CD simultaneously. In Step 1, we decouple the two echelons by solving the set partitioning model~\eqref{Heuristic_Set_Partition}. The solution of model~\eqref{Heuristic_Set_Partition} determines the assignment of the CLs to the LHs, $\beta_{ih}$. In step 2, we solve the echelon 2 problem (i.e., routing and scheduling MDs) using the greedy Algorithm~\ref{alg:Greedy} in Appendix~\ref{Appendix_Algorithms}. Then in Step 3, we determine the latest possible departure time of the MDs from the LHs, while satisfying the delivery due times of all the CLs (Algorithm~\ref{alg:Late_Departure} in Appendix~\ref{Appendix_Algorithms}).
In Step 4, for each LH $h$, we then determine ($i$) the required quantity for satisfying the demand of assigned CLs, $s_h$; and ($ii$) the due (i.e., latest) time that the LD operations in echelon 1 must be completed for that LH, $\pi_h$.
Then in Step 5, we determine the routing and scheduling of the LDs in echelon 1 by solving the set covering model~\eqref{Set_Covering}. Here, we refer to the set $\mathcal Q$ as all the \enquote{pattern-paths} for the LDs in echelon 1, where each \enquote{pattern-path} $q\in\mathcal Q$ represents the sequence of visiting the LHs, as well as the amount of commodity (in discrete values) delivered to each LH $h$, $W_h^q$. For each pattern-path, we ensure that the battery capacity and the PWCC of the LDs are satisfied. Moreover, we ensure that each pattern-path $q$ satisfies the due time of the LHs, i.e., $AT_h^q \leq \pi_h$, where $AT_h^q$ is the time that pattern-path $q$ visits LH $h$.

\begin{algorithm}[H]
\caption{Fast heuristic algorithm for the simplified problem variant.}
\label{alg:Heuristic}
\begin{algorithmic}[1]
\Require $\Delta = \{T_E, T_L^{LD/MD}, T_U^{LD/MD},  T_{hk}^{LD}, T_{ij}^{MD}, L_i, \hat{P}_i, B_{min/max}^{LD/MD}, P_{max}^{LD/MD}, P_N^{MD}\}$
\Ensure $\mathcal R^{MD}$ and $\mathcal R^{LD}$
\Procedure{SimplifiedHeuristic}{$\Delta$}
    \State{\textbf{Step 1:}}
    \State $\mathcal H_i \gets \{ h\in\mathcal H: T_E + T_L^{LD} + T_{0h}^{LD} + T_U^{LD} + T_L^{MD} + T_{hi}^{MD} + T_U^{MD} \leq L_i\}, \; \forall i\in\mathcal I$
    \State $\beta_{ih} \gets$ Solve model~\eqref{Heuristic_Set_Partition}\vspace{-5pt}
        \begin{subequations}
        \begin{align}
        \min & \sum_{i\in\mathcal I} \sum_{h\in\mathcal H_i}  \Big(T_{hi}^{MD} \cdot \beta_{ih}\Big) \\
        \text{s.t.} & \sum_{h\in\mathcal H_i} \beta_{ih} = 1, \forall i\in\mathcal I \\
        & \beta_{ih} \in \{0, 1\}
        \end{align}\label{Heuristic_Set_Partition}
        \end{subequations}\vspace{-15pt}
    \State{\textbf{Step 2:}}
    \For{$h\in\mathcal H$}
        \State $ActiveTrips_h \gets \Call{TripAssignment}{h, \beta_{ih}}$
    \EndFor
    \State $\mathcal R^{MD} \gets \cup_{h\in\mathcal H} ActiveTrips_h$: Routing and scheduling of all the MDs from all the LHs
    \State{\textbf{Step 3:}}
    \For{$h\in\mathcal H$ and $T_h \in ActiveTrips_h$}
        \State $\text{Dep}(T_h) \gets \Call{LatestDeparture}{T_h}$
    \EndFor
    \State{\textbf{Step 4:}}
    \For{$h \in \mathcal H$}
        \State $\pi_h \gets \min_{ActiveTrips_h} \{\text{Dep}(T_h)\}$, and $s_h \gets \sum_{i \in \mathcal I} \hat{P}_i : \beta_{ih} = 1$ 
    \EndFor
    \State{\textbf{Step 5:}}
    \State $\mathcal Q \gets$ Enumerate all the pattern-paths
    \State $\mathcal R^{LD} \gets$ Solve model~\eqref{Set_Covering}\vspace{-5pt}
        \begin{subequations}
        \begin{align}
            \min & \sum_{q\in\mathcal Q} x_q \\
            \text{s.t.} & \sum_{q\in\mathcal Q} W_h^{q} \cdot x_q \geq s_h, \quad \forall h\in\mathcal H \label{SC-Const}\\
            & x_q \in \mathbb{Z}^+
        \end{align}\label{Set_Covering}
        \end{subequations}\vspace{-15pt}
\EndProcedure
\end{algorithmic}
\end{algorithm}

\begin{prop}\label{prop:Conservative}
    Denoting $PU$ be the package weight unit in discretized demand, any solution obtained by Algorithm~\ref{alg:Heuristic} is a feasible solution to the \textit{HDDNM} problem restricted to the fixed delivery destinations (i.e., CLs) if the following conditions hold: $(i)$ The release times of all the packages are earlier than the departure time of the LDs in the simplified problem variant, i.e., $T_E \geq \max_{i\in\mathcal I} R_i$. $(ii)$ The discretized package weight quantities of the CLs in the simplified problem variant overestimate the actual package weights, i.e., $\hat{P}_i = PU \cdot \lceil \frac{P_i}{PU} \rceil$.
\end{prop}
\vspace{-5pt}
Proposition~\ref{prop:Conservative} demonstrates that, under certain conditions, the solution obtained by Algorithm~\ref{alg:Heuristic} for the simplified problem variant can be mapped to a feasible solution to the \textit{HDDNM} problem. Therefore, we develop Algorithm~\ref{alg:Full_Heuristic} to provide a good-quality feasible solution to the \textit{HDDNM} problem restricted to fixed delivery destinations. First, under aforementioned conditions ($i$) and ($ii$), we obtain an initial routing and scheduling for the LDs and the MDs, $\tilde{\mathcal R}^{LD}$ and $ \tilde{\mathcal R}^{MD}$, respectively, using Algorithm~\ref{alg:Heuristic}. We then apply two post-processing procedures as follows. (1) For each $r\in\tilde{\mathcal R}^{MD}$, we perform an intra-route 2-OPT procedure to reduce the total energy consumption. During this local neighborhood search, we only retain the route permutations that satisfy the delivery due times and do not delay the latest departure times. (2) We solve model~\eqref{MILP-LD} to assign individual packages to the discretized single-commodity LD pattern-paths, $\tilde{\mathcal R}^{LD}$. In model~\eqref{MILP-LD}, $A_h^r$ represents whether pattern-path $r\in\tilde{\mathcal R}^{LD}$ visits LH $h$. The solution to model~\eqref{MILP-LD} determines the routing and scheduling of the LDs to deliver individual packages from the CD to the LHs, while minimizing the required number of LDs. We then compute the total cost, accounting for the cost of the required fleet size, as well as the total energy consumption costs of the LDs and the MDs.

\begin{algorithm}[H]
\caption{Fast heuristic algorithm for the \textit{HDDNM} problem.}
\label{alg:Full_Heuristic}
\begin{algorithmic}[1]
\Require PU
\Ensure $\textbf{C}$, $\mathcal R^{LD}$, $\mathcal R^{MD}$

\Procedure{Heuristic}{$PU$}
    \State $\hat{P}_i \gets PU \cdot \lceil \frac{P_i}{PU} \rceil$ and $T_E \gets \max_{i\in\mathcal I} R_i$
    \Comment{The two conditions in Proposition~\ref{prop:Conservative}}
    
    \State $\tilde{\mathcal R}^{LD}, \tilde{\mathcal R}^{MD} \gets \Call{SimplifiedHeruistic}{\Delta}$
    \Comment{Initial solution}

    \State $\mathcal R^{MD} = \emptyset$
    \Comment{Improving echelon 2 solution}
    \For{$r\in \tilde{\mathcal R}^{MD}$}
        \State Perform intra-route 2-OPT on $r$, determine the best route, and add it to $\mathcal R^{MD}$
    \EndFor
    \State $\textbf{C}^{MD} = FC^{MD} \cdot |\mathcal R^{MD}| + C^E \cdot \text{Total Energy Consumption of MDs}$

    \State $\mathcal R^{LD} \gets$ Solve model~\eqref{MILP-LD}
    \Comment{Improving echelon 1 solution}\vspace{-7pt}
        \begin{subequations}
        \begin{align}
            \min & \sum_{r\in \tilde{\mathcal R}^{LD}} y^r \\
            \text{s.t.} & \sum_{r\in \tilde{\mathcal R}^{LD}}  x_{ih}^r = \beta_{ih}, \quad \forall i\in\mathcal I,\, h\in\mathcal H \\
                        & x_{ih}^r \leq A_h^r , \quad \forall i\in\mathcal I, \, r\in \tilde{\mathcal R}^{LD},\, h\in\mathcal H \\
                        & \sum_{i\in\mathcal I} P_i \cdot x_{ih}^r \leq W_h^r,\quad \forall r\in \tilde{\mathcal R}^{LD},\, h\in\mathcal H \\
                        & \sum_{i\in\mathcal I} \sum_{h\in\mathcal H} x_{ih}^r \leq M \cdot y^r,\quad \forall r\in \tilde{\mathcal R}^{LD} \\
                        & x_{ih}^r , y^r \in \{0, 1\}
        \end{align}\label{MILP-LD}
        \end{subequations}\vspace{-15pt}
    \State $\textbf{C}^{LD} = FC^{LD} \cdot |\mathcal R^{LD}| + C^E \cdot \text{Total Energy Consumption of LDs}$
    \State \Return $\textbf{C} = \textbf{C}^{LD} + \textbf{C}^{MD}$, $\mathcal R^{LD}$, $\mathcal R^{MD}$

\EndProcedure
\end{algorithmic}
\end{algorithm}

\vspace{-20pt}
\section{Numerical Results and Managerial Insights}
\vspace{-5pt}

In this section, we assess the efficacy of the proposed models and algorithms, as well as compare the performance of different solution algorithms in solving the \textit{HDDNM} problem. Moreover, based on the real-life case study presented in Section~\ref{subsec:case-study}, we conducted a series of numerical experiments to provide practical insights for emergency healthcare logistics system owners (i.e., decision-makers) into the following questions: (1) How do different solution methods, Gurobi, enhancement strategies in our proposed exact method, and the heuristic approach for the simplified problem variant, compare to each other in terms of runtime and solution quality (i.e., optimality gap)? (2) How do the total cost and the required fleet composition vary from a heterogeneous fleet of drones to a homogeneous fleet of MDs in echelon 2? and (3) How much impact do the changes in maximum permissible delay, consolidation delay at the LHs, and the multi-trip delivery mode of LDs and MDs, have on the total cost and the required number of drones from each drone type?

\vspace{-5pt}
\subsection{Case Study}\label{subsec:case-study}

In this section, we present a real-life case study to evaluate the performance of the proposed model and solution methods in solving the \textit{HDDNM} problem, as well as to provide key managerial insights into efficiently using a heterogeneous fleet of drones for delivering time-sensitive products across a regional hierarchical H\&S network. We used a two-hour Spatio-temporal whole blood delivery data from the Interpath Laboratory, Inc.---a healthcare logistics company---located in Pendleton, Oregon, USA, which is a fully integrated, for-profit, clinical, and anatomic pathology medical laboratory. This delivery data contains: ($i$) the central blood bank's geographic location (i.e., latitude and longitude) as the CD, ($ii$) the geographic locations (i.e., latitude and longitude) of five regional hospitals (LHs), ($iii$) the geographic locations of 35 candidate ambulance locations (ALs), and ($iv$) the geographic locations, package weights, and the timing of a package being ready for pickup at the CD, for 100 clinics (CLs) and 23 mobile ambulances (AMBs). The geographic distribution of the CD, LHs, CLs, ALs, and AMBs is shown in Figure~\ref{fig:Case_Study}. Based on our close collaboration with the Interpath Laboratory, Inc., each unit of whole blood weighs approximately 0.453 kilograms (kg) (i.e., 1 lb), and each package contains between 1 and 10 units of whole blood. Therefore, the package weights vary from 0.453 to 4.53 kg (i.e., from 1 to 10 lbs).

\begin{figure}[h!]
    \centering
    \includegraphics[width=0.9\linewidth]{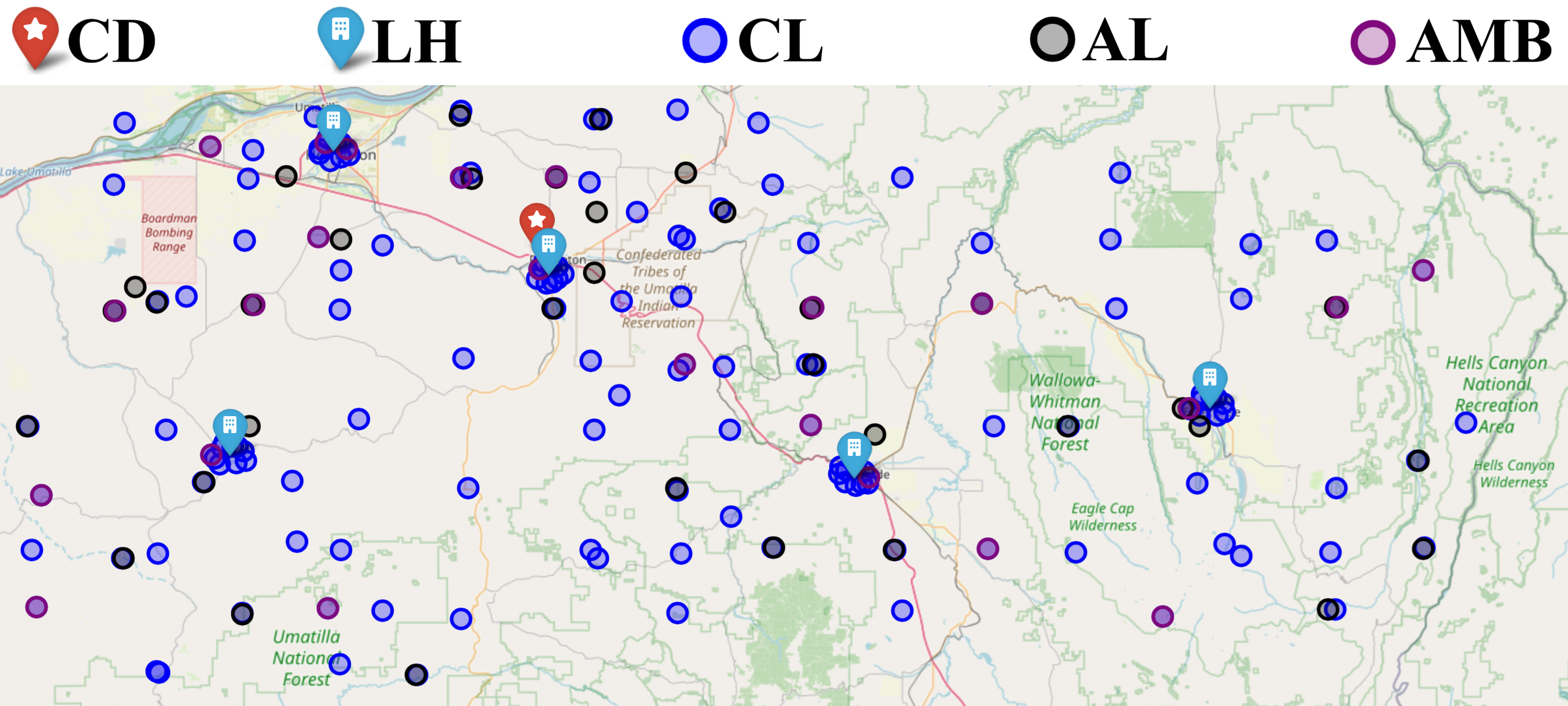}
    \caption{Geographic locations of the CD, LHs, CLs, ALs, and AMBs.\vspace{-10pt}}
    \label{fig:Case_Study}
\end{figure}

We considered a heterogeneous fleet of drones across the two echelons of the hierarchical H\&S network, comprising one LD type, one MD type, and two SD types. We used rotary quadcopter (Tarot 650) and rotary hexacopter (DJI Matrice 600 Pro) as the two SD types, and fixed-wing vertical-takeoff-and-landing (VTOL) (Wingcopter 198) as the MD type in echelon 2. Each Wingcopter 198 can carry three packages simultaneously (i.e., $P_N^{MD}=3$). Furthermore, we used the battery capacity, speed, PWCC, cost, power consumption, loading/unloading time, and battery replacement time for these three drone types from~\citep{hosseini2026branch}, where all drone data in~\citep{hosseini2026branch} are obtained from drone flight tests conducted at the Idaho National Laboratory test site~\citep{bhuiyan2024aerial, bhuiyan2025optimization}. We used Aero-200 long-range drone as the LD type in echelon 1. All drone data for Aero-200 long-range drones, including ownership cost, battery cost, speed, PWCC, battery capacity, and unloaded and fully loaded flight range (i.e., power consumption), are obtained from the Interpath Laboratory, Inc.

In our case study, fixed cost of opening AL $l$ (i.e., $LC_l$) varies between \$2,000 and \$5,000, representing the cost of the necessary equipment (e.g., a medical-grade refrigerator and a calibrated temperature monitoring system) to store blood within specific temperature ranges (e.g., $1^{\circ}\text{C} - 6^{\circ}\text{C}$ for whole blood)~\citep{Aegis2022Blood, Blood2025Temp}. The allocation cost of the AMBs to the ALs represents the fuel and energy consumption cost of ambulances to travel from their origin to the designated delivery location. We consider a city-driving-condition fuel efficiency of 11.26 km (7 miles) per gallon (MPG) for an ambulance~\citep{Fuel2011AMB}. Moreover, the gasoline price and energy conversion factor for each gallon of gasoline are set to \$2.2, and 121.32 MJ (33.7 kWh), respectively~\citep{hosseini2026branch}. Using the aforementioned values, as well as an energy cost of \$0.036/MJ (i.e., \$0.13/kWh)~\citep{hosseini2026branch}, the combined fuel and energy consumption cost of an ambulance traveling from its origin to an AL is \$0.938/mile.

\vspace{-5pt}
\subsection{Experimental Setup}\label{subsec:Experimental_Setup}

We implemented our models and solution methods in Python 3.12.7. We used the Gurobi solver ver. 12.0.1~\citep{gurobi2023} to solve the \textit{BaM} (i.e., Eqs~\eqref{eq:obj} - \eqref{E2-Const65}), model~\eqref{Heuristic_Set_Partition}, model~\eqref{Set_Covering}, and model~\eqref{MILP-LD}. We set the optimality gap and the time limit to 1\% and 1 hour, respectively, in all numerical experiments. We ran the numerical experiments on the compute nodes in the University of Texas at San Antonio High Performance Computing (HPC) Cluster, ARC, using a 64-bit Linux operating system. A compute node has two Intel Xeon Gold 6248 CPUs, each with @2.5 GHz and 40 physical cores with 376 GB of installed RAM. In the \textit{BaM} (i.e., Eqs.~\eqref{eq:obj}--\eqref{E2-Const65}), we used $M_1 = 2 \times \max_{p\in\mathcal P} L_p = 43080$, $M_2 = 2 \times P_{max}^{LD} = 88$, and $M_3 = 2 \times P_{max}^{MD} = 27$. 

As presented in Table~\ref{tab:problem-instance}, numerical experiments are conducted for four distinct network sizes (i.e., the number of LHs, CLs, AMBs, and ALs). For each network size, we generated different cases by varying the number of LHs, CLs, AMBs, and ALs. Moreover, in our numerical experiments, we used different values for the time windows of final delivery destinations (i.e., CLs and AMBs), including three, four, and five hours.

\begin{table}[htbp!]
    \centering
    \caption{Network sizes used in our numerical experiments.\vspace{-5pt}}
    \label{tab:problem-instance}
    \begin{tabular}{ccccc}
    \hline
        Network Size & Number of LHs & Number of CLs & Number of ALs & Number of AMBs\\
        \hline \hline
         Small & 2 & 10 & 3 &2\\
         Medium & 3 & 20 & 8 & 5 \\
         Large & 4 & 50 & 15 & 10 \\
         Very Large & 5 & 100 & 35 & 23 \\
         \hline
    \end{tabular}
\end{table}

\vspace{-5pt}
\subsection{Runtime and Solution Quality of the Solution Methods}

In this section, we compare the performance (i.e., runtime and optimality gap) of the proposed solution methods as well as evaluate their individual performance in solving the \textit{HDDNM} problem for different network sizes (i.e., different numbers of LHs, CLs, ALs, and AMBs) described in Section~\ref{subsec:Experimental_Setup} and different time windows. For brevity and readability, we define a set of acronyms for the solution methods, as shown in Table~\ref{tab:Acronyms}.

\begin{table}[H]
\centering
\caption{Acronyms for different solution methods.\vspace{-5pt}}
\label{tab:Acronyms}
\begin{tabular}{c c p{12.1cm}}
\toprule
Acronym & \makecell[c]{Method \\ Type} & Description \\
\hline \hline
\textit{M-B} & Exact & Solving the \textit{BaM} (i.e., Eqs.~\eqref{eq:obj}--\eqref{E2-Const65}) using Gurobi \\
\textit{M-P} & Exact & Pre-processing and search space reduction in solving the \textit{BaM} (described in Section~\ref{subsec:preprocess}) \\
\textit{M-R} & Exact & Problem-specific reformulations in solving the \textit{BaM} (described in Section~\ref{subsec:reformulation}) \\
\textit{M-V} & Exact & Valid inequalities in solving the \textit{BaM} (described in Section~\ref{subsec:valid-inequality}) \\
\textit{M-S} & Exact & Dynamic cutting planes in solving the \textit{BaM} (described in Section~\ref{subsec:separation}) \\
\textit{M-F} & Exact & Applying all enhancement strategies in solving the \textit{BaM} \\
\textit{M-H} & Heuristic & Using the heuristic solution algorithm presented in Section~\ref{subsec:Heuristic} \\
\bottomrule
\end{tabular}
\end{table}

\vspace{-5pt}We demonstrate the performance of different solution methods by varying the number of LHs, CLs, ALs, and AMBs as the four network sizes described in Section \ref{subsec:Experimental_Setup}, as well as the delivery time window from three hours to five hours. To account for variations in problem settings and to ensure consistent runtime performance evaluation, we ran four replications of each problem instance by changing the sets of the LHs, AMBs, ALs, and CLs across the full dataset described in Section~\ref{subsec:case-study}, and then presented the average runtime and optimality gap of different solution methods for different network sizes and time windows across the four replications in Table~\ref{tab:runtime}. For \textit{M-H}, the optimality gap is calculated as the percentage by which its objective value exceeds the objective value of \textit{M-F} for the simplified variant of the \textit{HDDNM} problem (discussed in Section~\ref{subsec:Heuristic}).

\begin{table}[htbp!]
    \caption{Runtime and solution quality comparison among different solution methods. \enquote{-} and \enquote{${*}$} represent that the corresponding solution method is not able to compute a feasible solution within the one-hour time limit, and the solution quality for the corresponding problem instance cannot be reported, respectively.\vspace{-5pt}}\label{tab:runtime}
    \resizebox{\textwidth}{!}{
    \centering
    \setlength{\tabcolsep}{2.5pt}
        \begin{tabular}{cccccccccccccccc}
            \toprule
            \multirow{3}{*}{\makecell[c]{Network \\ Size}} & \multirow{3}{*}{\makecell[c]{Time \\ Window \\ (hours)}} 
            & \multicolumn{7}{c}{\makecell[c]{Runtime \\ (seconds)}} 
            & \multicolumn{7}{c}{\makecell[c]{Optimality \\Gap (\%)}} \\
            \cmidrule(lr){3-9} \cmidrule(lr){10-16}
            & & \textit{M-B}
            & \textit{M-P}
            & \textit{M-R}
            & \textit{M-V}
            & \textit{M-S}
            & \textit{M-F}
            & \textit{M-H}
            & \textit{M-B}
            & \textit{M-P}
            & \textit{M-R}
            & \textit{M-V}
            & \textit{M-S}
            & \textit{M-F}
            & \textit{M-H} \\
            \midrule \hline
            \multirow{3}{*}{Small} & 5 
                                        & 12.8 & 3.23 & 7.19 & 2.91 & 3.30 & 1.23 & 0.02 
                                        & 0.0 & 0.0 & 0.0 & 0.0 & 0.0 & 0.0 & 0.0 \\
                                    & 4
                                        & 19.3 & 6.74 & 11.3 & 3.97 & 4.63 & 1.74 & 0.01 
                                        & 0.0 & 0.0 & 0.0 & 0.0 & 0.0 & 0.0 & 0.0 \\
                                    & 3
                                        & 69.6 & 11.2 & 25.5 & 5.93 & 5.34 & 3.43 & 0.01 
                                        & 0.0 & 0.0 & 0.0 & 0.0 & 0.0 & 0.0 & 0.0 \\ \hline
            \multirow{3}{*}{Medium} & 5 
                                        & 2386 & 1163 & 1059 & 536 & 712 & 29.0 & 0.05 
                                        & 0.0 & 0.0 & 0.0 & 0.0 & 0.0 & 0.0 & 0.0 \\
                                    & 4
                                        & 3600 & 1329 & 1415 & 973 & 1063 & 48.5 & 0.05 
                                        & 3.1 & 0.0 & 0.0 & 0.0 & 0.0 & 0.0 & 2.4 \\
                                    & 3
                                        & 3600 & 2863 & 3086 & 1944 & 2487 & 486 & 0.03 
                                        & 18.7 & 0.0 & 0.0 & 0.0 & 0.0 & 0.0 & 3.2 \\ \hline
            \multirow{3}{*}{Large} & 5 
                                        & - & - & - & - & - & 3600 & 0.23 
                                        & - & - & - & - & - & 8.6 & 2.8 \\
                                    & 4
                                        & - & - & - & - & - & 3600 & 0.23 
                                        & - & - & - & - & - & 27.8 & 4.2 \\
                                    & 3
                                        & - & - & - & - & - & - & 0.25  
                                        & - & - & - & - & - & - & * \\ \hline
            \multirow{3}{*}{\makecell[c]{Very \\ Large}} & 5 
                                        & - & - & - & - & - & - & 3.24 
                                        & - & - & - & - & - & - & * \\
                                    & 4
                                        & - & - & - & - & - & - & 3.13 
                                        & - & - & - & - & - & - & * \\
                                    & 3
                                        & - & - & - & - & - & - & 2.91  
                                        & - & - & - & - & - & - & * \\ 
            \bottomrule
        \end{tabular}%
    }
\end{table}

We see from Table~\ref{tab:runtime} that the performance of different solution methods varies significantly across different network sizes and time windows. For small-sized networks, all the exact solution methods (i.e., \textit{M-B}, \textit{M-P}, \textit{M-R}, \textit{M-V}, \textit{M-S}, and \textit{M-F}) are able to compute the optimal solution within the one-hour time limit. As demonstrated in Proposition~\ref{prop:NP-hard}, our \textit{HDDNM} problem is NP-hard and thus has an exponential runtime. Therefore, as the network size increases to the medium scale (i.e., three LHs, 20 CLs, eight ALs, and five AMBs), \textit{M-B} (i.e., solving the \textit{BaM} directly with Gurobi) faces computational difficulty and cannot solve the problem to optimality for all the time windows. However, applying the enhancement strategies proposed in Section~\ref{sec:Solution_Method} results in a considerably better computational performance compared to \textit{M-B}, as evident by ($i$) their ability to solve larger-sized networks with different time windows, and ($ii$) their significantly lower average runtime across different network sizes and time windows. Table~\ref{tab:runtime} shows that, each exact enhancement strategy described in Section~\ref{sec:Solution_Method}, when applied individually (i.e., \textit{M-P}, \textit{M-R}, \textit{M-V}, and \textit{M-S}) can provide the optimal solution for the medium-sized network (which is a moderately large network comprising 36 nodes) across different time windows. Moreover, on average, across all network sizes and time windows, \textit{M-P}, \textit{M-R}, \textit{M-V}, and \textit{M-S} are \textbf{3.8, 2.1, 6.4, and 6.1 times faster}, respectively, than \textit{M-B}, considering the problem instances \textit{M-B} converges to optimality within the one-hour time limit.

\begin{remark}
    Valid inequalities and dynamic cutting planes are the most effective individual enhancements for improving computational efficiency in solving the \textit{HDDNM} problem.
\end{remark}

Furthermore, Table~\ref{tab:runtime} demonstrates that our proposed \textit{M-F}, which incorporates all enhancement strategies described in Section~\ref{sec:Solution_Method} (i.e., pre-processing, reformulations, valid inequalities, and dynamic cutting planes), significantly outperforms the commercial solver (i.e., \textit{M-B}). On average, across all network sizes and time windows, \textit{M-F} is \textbf{31 times faster} than \textit{M-B}, considering the problem instances \textit{M-B} converges to optimality within the one-hour time limit. Furthermore, \textit{M-F} can solve more problem instances (i.e., network sizes and time windows) within the one-hour time limit compared to \textit{M-B}. For instance, \textit{M-B} cannot find a feasible solution (as shown with \enquote{-} in Table~\ref{tab:runtime}) for any replication of the large-sized network with five hours time window within the one-hour time limit, whereas \textit{M-F} converges to a less than 10\% average optimality gap among different replications of the same problem instance within the one-hour time limit. The significantly better performance of \textit{M-F} over \textit{M-B} highlights the effectiveness of our proposed exact enhancement strategies---pre-processing, reformulations, valid inequalities, and dynamic cutting planes---in solving the \textit{HDDNM} problem. The proposed enhancement strategies reduce the solution space and tighten the LP-relaxation, thus improving the computational performance of \textit{M-F} compared to \textit{M-B}.

However, we observe from Table~\ref{tab:runtime} that \textit{M-F}---the fastest exact solution method---cannot solve large-sized and very-large-sized networks to optimality within the one-hour time limit. Therefore, to solve large-scale instances of the \textit{HDDNM} problem faster, we used \textit{M-H}, which ($i$) consistently outperforms the other solution methods in terms of runtime while maintaining a high-quality solution (compared to the exact solution of the simplified variant of the \textit{HDDNM} problem), and ($ii$) can solve the largest-sized network (i.e., very large network) comprising more than 100 nodes within a very short runtime (i.e., less than five seconds). Our proposed heuristic approach (i.e., \textit{M-H} described in Algorithm~\ref{alg:Full_Heuristic}) significantly reduces the computational burden, especially for larger-sized networks, where exact solution methods (e.g., \textit{M-F}) cannot find the optimal solution within the one-hour time limit. On average, across all network sizes and time windows, \textit{M-H} provides high-quality near-optimal solutions and is \textbf{3,375 times faster} than \textit{M-F}, considering the problem instances \textit{M-F} converges to optimality within the one-hour time limit. Therefore, our proposed \textit{M-F} and \textit{M-H} can serve as viable decision-support tools for healthcare logistics system owners, enabling them to make fast and reliable decisions in efficiently designing a hierarchical H\&S drone delivery network for delivering time-sensitive medical items across large-scale regionally distributed healthcare networks.

Another noteworthy observation is the effect of the time window on the computational performance of different solution methods. We see from Table \ref{tab:runtime} that although the time window does not affect the runtime of \textit{M-H}, the computational runtime of all the exact solution methods is strongly affected by the time window. As the time window decreases (i.e., the delivery due time becomes more restrictive), each drone can deliver fewer packages without violating the delivery due times, thus requiring a larger number of drones in both echelons to satisfy the tighter time windows. Due to the explicit drone indexing in the \textit{BaM}, this increased required number of drones introduces additional non-zero decision variables and requires more branching operations within the B\&B tree. Therefore, decreasing the time window makes solving the \textit{HDDNM} problem computationally more challenging for all the exact solution methods. For instance, the runtime of \textit{M-B}, \textit{M-P}, \textit{M-R}, \textit{M-V}, \textit{M-S}, and \textit{M-F} increases by 442\%, 247\%, 255\%, 104\%, 62\%, and 179\%, respectively, for the small-sized network as the time window decreases from five hours to three hours.

As explained above, \textit{M-H} solves large-scale instances of the \textit{HDDNM} problem in a very short time while providing high-quality solutions with an average optimality gap of less than 4.2\%. To further demonstrate the solution quality of \textit{M-H} across different network sizes and time windows, Table~\ref{tab:Heuristic_Gap} shows the best and average optimality gaps of \textit{M-H} among different replications of the network sizes and time windows for which the exact method, \textit{M-F}, can find a feasible solution to make it possible to compute the optimality gap of \textit{M-H}. \vspace{-5pt}

\begin{table}[H]
    \centering
    \caption{The best and average optimality gap of \textit{M-H} for different network sizes and time windows among different replications.\vspace{-5pt}}
    \begin{tabular}{cccc}
        \toprule
        Network Size & Time Window (hours) & Best Optimality Gap (\%) & Average Optimality Gap (\%) \\
        \hline \hline
        \multirow{3}{*}{Small}  & 5 & 0.0 & 0.0 \\
                                & 4 & 0.0 & 0.0 \\
                                & 3 & 0.0 & 0.0 \\ \hline
        \multirow{3}{*}{Medium} & 5 & 0.0 & 0.0 \\
                                & 4 & 0.0 & 2.4 \\
                                & 3 & 2.5 & 3.2 \\ \hline
        \multirow{2}{*}{Large}  & 5 & 0.0 & 2.8 \\
                                & 4 & 3.2 & 4.2 \\ \hline
    \end{tabular}
    \label{tab:Heuristic_Gap}
\end{table}
\vspace{-5pt}

We see from Table~\ref{tab:Heuristic_Gap} that our \textit{M-H} consistently provides high-quality near-optimal solutions, with the best optimality gap of less than 3.2\%. Moreover, Table~\ref{tab:Heuristic_Gap} shows that the best optimality gap obtained by \textit{M-H} varies across different network sizes and time windows. For small-sized networks, \textit{M-H} consistently converges to the optimal solution across all time windows. For larger-sized networks (i.e., medium-sized and large-sized networks), the best solution quality (i.e., optimality gap) improves as the time window increases (i.e., the delivery due time becomes more relaxed). This is because, as the time window increases, the temporal constraints on the MD-trips are less restrictive, which expands the feasible space for MDs. In our \textit{M-H}, the routing and scheduling of MDs in echelon 2 is solved using the greedy Algorithm~\ref{alg:Greedy} in Appendix~\ref{Appendix_Algorithms}, which sequentially assigns CLs to MDs based on ($i$) the remaining drone capacity (i.e., PWCC and $P_N^{MD}$), and ($ii$) the delivery due times. Therefore, increasing the time window allows the greedy Algorithm~\ref{alg:Greedy} to prioritize the remaining drone capacity over strict delivery due times (as due times are relaxed with longer time windows) when assigning CLs to the MDs. As a result, the sequential assignment of the sorted CLs to the MDs in Algorithm~\ref{alg:Greedy} is more likely to converge to the true optimal solution for the routing and scheduling of MDs in echelon 2. Moreover, from our numerical experiments, we observe that the solution quality of \textit{M-H} improves as the network becomes more sparse (i.e., CLs are closer to a single LH rather than accessible by all LHs). This is because spatial sparsity between CLs and LHs reduces the overlap between the service areas of different LHs. This reduces the potential for obtaining a suboptimal CL-to-LH assignment in Step 1 of Algorithm~\ref{alg:Heuristic} compared to the optimal assignment of CLs to the LHs obtained from the exact solution. Proposition~\ref{prop:Exact_Heuristic} provides the sufficient conditions under which \textit{M-H} is guaranteed to converge to the optimal solution of the simplified variant of the \textit{HDDNM} problem. The proof of Proposition~\ref{prop:Exact_Heuristic} is provided in Appendix~\ref{Appendix_Proof}.\vspace{-5pt}

\begin{prop}\label{prop:Exact_Heuristic}
    The solution obtained by Algorithm~\ref{alg:Full_Heuristic} is the exact optimal solution to the simplified variant of the \textit{HDDNM} problem if the following conditions hold: $(i)$ $\mathcal{H}_i \cap \mathcal{H}_j = \emptyset, \forall i,j \in \mathcal I$; $(ii)$ $L_i \gg T_E + \sum_{(h,k) \in \mathcal E} T_{hk}^{LD} + |\mathcal H|T_U^{LD} + T_L^{LD} + \sum_{(i,j) \in \mathcal A}T_{ij}^{MD} + P_N^{MD} T_U^{MD} + T_L^{MD}$; $(iii)$ $\max_{\mathcal{S} \subseteq \mathcal{I}, |\mathcal{S}| = P_N^{MD}} \left( \sum_{i \in \mathcal{S}} P_i \right) \leq P_{max}^{MD}$; and $(iv)$ $P_i/PU \in \mathbb{Z}^+, \forall i \in \mathcal I$ (i.e., $PU$ is the common divisor of all package weights).
\end{prop}

\vspace{-5pt}

\vspace{-5pt}
\subsection{Effect of Fleet Type in Echelon 2 on the Total Cost and the Required Fleet Composition}

In this section, we present insights into how much cost-efficient a heterogeneous fleet of MDs and SDs with distinct characteristics---cost, battery capacity, power consumption, PWCC, speed, and delivery mode---is compared to a homogeneous fleet of MDs in echelon 2 in delivering time-sensitive medical items across a regional hierarchical H\&S network. Table~\ref{tab:Effect_Fleet_Composition} shows the effect of fleet type in echelon 2 on the total cost, and the required fleet composition and fleet size. In this analysis, we used the medium-sized network containing three LHs, 20 CLs, eight ALs, and five AMBs with a five-hour time window and solved using \textit{M-F}.\vspace{-5pt}

\begin{table}[H]
    \centering
    \caption{Effect of fleet type in echelon 2 on the total cost and the required fleet composition. \enquote{-} denotes that the corresponding drone type is not present in the corresponding fleet.\vspace{-5pt}}
    \label{tab:Effect_Fleet_Composition}
    \resizebox{\textwidth}{!}{
        \begin{tabular}{ccccccccc}
            \hline
            \multirow{2}{*}{\makecell[c]{Fleet \\ Type}} & \multicolumn{2}{c}{Cost} & \multicolumn{2}{c}{\makecell[c]{Required Number \\ of LD-Trips}} & \multicolumn{2}{c}{\makecell[c]{Required Number \\ of MD-Trips}} & \multicolumn{2}{c}{\makecell[c]{Required Number \\ of SDs}} \\
            \cmidrule(lr){2-3} \cmidrule(lr){4-5} \cmidrule(lr){6-7} \cmidrule(lr){8-9}
                                                & \makecell[c]{Total \\ Cost (\$)}  & \makecell[c]{Percentage \\ Increase} & \makecell[c]{Total \\ LDs}  &   \makecell[c]{Total \\ Trips}  &  \makecell[c]{Total \\ VTOLs}  &  \makecell[c]{Total \\ Trips}      & DJI    & Tarot  \\ \hline \hline
            Heterogeneous               & 669,912        &  0.0\%       &  1     &   3    &  2     &   6    &    3   &  3 \\
            Homogeneous                         & 765,603        &  12.5\%      &  1     &   3    &  4     &   10    &   -   &  -     \\ \hline
        \end{tabular}
    }
\end{table}
\vspace{-5pt}
We used different fleet types in echelon 2, including: ($i$) a heterogeneous fleet of drones (considered as the baseline fleet) comprising rotary quadcopter (i.e., Tarot) and rotary hexacopter (i.e., DJI) as SD types, and VTOL drones as the MD type; and ($ii$) a homogeneous fleet of VTOL drones (i.e., only MDs). The percentage increase in the total cost due to changing the fleet type in Table~\ref{tab:Effect_Fleet_Composition} is calculated compared to the total cost of the baseline fleet. We see from Table~\ref{tab:Effect_Fleet_Composition} that the total cost increases as the fleet type becomes less diverse (i.e., the fleet type comprises less number of drone types). This is because, due to their physical characteristics, rotary quadcopter (i.e., Tarot) and rotary hexacopter (i.e., DJI) are more cost- and energy-efficient than VTOL drones in delivering light packages to short distances, as demonstrated in~\citep{bhuiyan2024aerial}. Therefore, for final delivery locations closer to the LHs and requiring delivery of lighter package weights, using a shorter-range drone---a rotary quadcopter (i.e., Tarot) or a rotary hexacopter (i.e., DJI)---is more cost-efficient than using a VTOL drone. As a result, reducing the number of drone types in the fleet forces the optimal solution of the \textit{HDDNM} problem to use a larger number of less cost-efficient drones (i.e., VTOLs) for shorter-range deliveries, which increases the total cost. Notably, we see from Table~\ref{tab:Effect_Fleet_Composition} that the required number of MDs increases by 100\% as the fleet in echelon 2 changes from a heterogeneous fleet of MDs and SDs to a homogeneous fleet of MDs. These insights can benefit logistics system owners in designing a cost-efficient drone fleet within a hierarchical H\&S structure to deliver time-sensitive medical items.

\vspace{-5pt}
\subsection{Effect of Maximum Permissible Delay on the Total Cost and the Required Fleet Composition}

As mentioned in Section~\ref{sec:Problem_Description}, each medical item has a distinct time window defined by its release time at the CD and its delivery due time at the corresponding delivery destination (i.e., CL or AMB). The time difference between the release time and delivery due time is determined based on a maximum permissible delay, which differs across different businesses and product types based on the time-criticality of the products (e.g., defibrillator vs. blood sample). Therefore, in this section, we present insights into how the maximum permissible delay between the release time of the packages at the CD and their delivery due times at the corresponding delivery destinations affects the total cost, and the required fleet composition and fleet size. Table~\ref{tab:Effect_Time_Window} shows the effect of increasing the maximum permissible delay on the total cost and the required number of drones of each type. In this analysis, we used the medium-sized network containing three LHs, 20 CLs, eight ALs, and five AMBs and solved using \textit{M-F}.

\begin{table}[h!]
    \centering
    \caption{Effect of maximum permissible delay on the total cost and the required fleet composition.\vspace{-5pt}}
    \label{tab:Effect_Time_Window}
    \resizebox{\textwidth}{!}{
        \begin{tabular}{ccccccccc}
            \hline
            \multirow{2}{*}{\makecell[c]{Maximum \\ Permissible \\ Delay \\ (minutes)}} & \multicolumn{2}{c}{Cost} & \multicolumn{2}{c}{\makecell[c]{Required Number \\ of LD-Trips}} & \multicolumn{2}{c}{\makecell[c]{Required Number \\ of MD-Trips}} & \multicolumn{2}{c}{\makecell[c]{Required Number \\ of SDs}} \\
            \cmidrule(lr){2-3} \cmidrule(lr){4-5} \cmidrule(lr){6-7} \cmidrule(lr){8-9}
                                                & \makecell[c]{Total \\ Cost (\$)}  & \makecell[c]{Percentage \\ Decrease} & \makecell[c]{Total \\ LDs}  &   \makecell[c]{Total \\ Trips}  &  \makecell[c]{Total \\ VTOLs}  &  \makecell[c]{Total \\ Trips}      & DJI    & Tarot  \\ \hline \hline
            76                & 2,573,944        &  0.0\%       &  5     &   7    &  5     &   9     &   4    &  0     \\ 
            90                & 1,802,711        &  30.0\%      &  3     &   6    &  6     &   11    &   3    &  1     \\ 
            120               & 1,747,582        &  32.1\%      &  3     &   6    &  5     &   10    &   3    &  3     \\
            150               & 1,590,083        &  38.2\%      &  3     &   5    &  3     &   5     &   4    &  3     \\
            180               & 1,232,906        &  52.1\%      &  2     &   4    &  4     &   6     &   4    &  4     \\
            210               &   912,959        &  64.5\%      &  1     &   3    &  6     &   9     &   2    &  3     \\
            240               &   797,172        &  69.0\%      &  1     &   3    &  4     &   8     &   3    &  3     \\
            270               &   737,885        &  71.3\%      &  1     &   3    &  3     &   7     &   4    &  2     \\
            300               &   669,912        &  74.0\%      &  1     &   3    &  2     &   6     &   3    &  3     \\
            330               &   660,398        &  74.3\%      &  1     &   3    &  2     &   7     &   2    &  2     \\
            360               &   599,798        &  76.7\%      &  1     &   3    &  1     &   7     &   2    &  2     \\ \hline
        \end{tabular}
    }
\end{table}

We gradually increase the maximum permissible delay starting from the most restricted maximum permissible delay (i.e., 76 minutes in this experiment), considered as the baseline, to the most relaxed one (i.e., 360 minutes in this experiment). The most restricted maximum permissible delay is the shortest delay with which the delivery due time of all delivery locations (i.e., CLs and AMBs) can be satisfied (i.e., the problem remains feasible), whereas the most relaxed is the one from which further increment in the delay does not affect the optimal solution. The percentage decrease in the total cost due to increasing the maximum permissible delay in Table \ref{tab:Effect_Time_Window} is calculated compared to the total cost corresponding to the baseline maximum permissible delay (i.e., 76-minute maximum permissible delay).

We see from Table~\ref{tab:Effect_Time_Window} that the total cost and the required fleet size (i.e., the total number of drones) decrease as the maximum permissible delay increases. Notably, as the maximum permissible delay increases from the most restricted maximum permissible delay (i.e., 76 minutes) to the most relaxed one (i.e., 360 minutes), the total cost and the required fleet size decrease by 76.7\% and 57.1\%, respectively. This is because, as the maximum permissible delay increases, the delivery due times become less restrictive. Therefore, each drone can deliver a larger number of packages to the corresponding delivery destinations. Furthermore, increasing the maximum permissible delay allows LDs and MDs (performing multi-trip circular deliveries) to perform additional trips while maintaining the delivery due times. This results in reducing the required number of drones of each type. Notably, we see from Table~\ref{tab:Effect_Time_Window} that the average number of trips performed by each LD and MD increases by 114\% and 289\%, respectively, as the maximum permissible delay increases from the most restricted maximum permissible delay (i.e., 76 minutes) to the most relaxed maximum permissible delay (i.e., 360 minutes). These findings can benefit logistics system owners in designing their H\&S drone delivery network, depending on the time-criticality of their medical items.

\vspace{-5pt}
\subsection{Effect of Consolidation Delay at the LHs on the Total Cost and the Required Fleet Composition}

As mentioned in Section~\ref{sec:Problem_Description}, we assume that once a package is delivered to an LH, it can be picked up by a drone in echelon 2 to perform the last-mile delivery. However, in practical healthcare logistics, the transfer of medical items between the two echelons may face delays due to logistical limitations (e.g., clinical protocols and specialized repackaging of temperature-sensitive medical items). Therefore, in this section, we introduce a consolidation delay parameter, $T_C$, to account for the time delay in transferring medical items at the LHs, and present insights into how the consolidation delay at the LHs affects the total cost and the required fleet composition. To model the consolidation delay at the LHs, we modify constraints~\eqref{eq:Int-Const2} of the \textit{BaM} (i.e., Eqs. \eqref{eq:obj} - \eqref{E2-Const65}) as constraints~\eqref{eq:Int-Const2-Consld}.
\vspace{-5pt}
\begin{align}
    & at_h^{dg} - M_1 (1 - w_{ph}^{dg}) \leq \pi_p - T_C \leq at_h^{dg} + M_1 (1 - w_{ph}^{dg}), \quad \forall p\in\mathcal P,\, h\in\mathcal H,\, d\in\mathcal D,\, g\in\mathcal G^{LD}  \label{eq:Int-Const2-Consld}
\end{align}

Figure~\ref{fig:Effect_Consolidation_Delay} shows the effect of increasing the consolidation delay at the LHs on the total cost, and the required fleet composition and fleet size. In this analysis, we used the medium-sized network containing three LHs, 20 CLs, eight ALs, and five AMBs with a five-hour time window and solved using \textit{M-F}.

\begin{figure}[h!]
    \centering
    \includegraphics[width=0.75\linewidth]{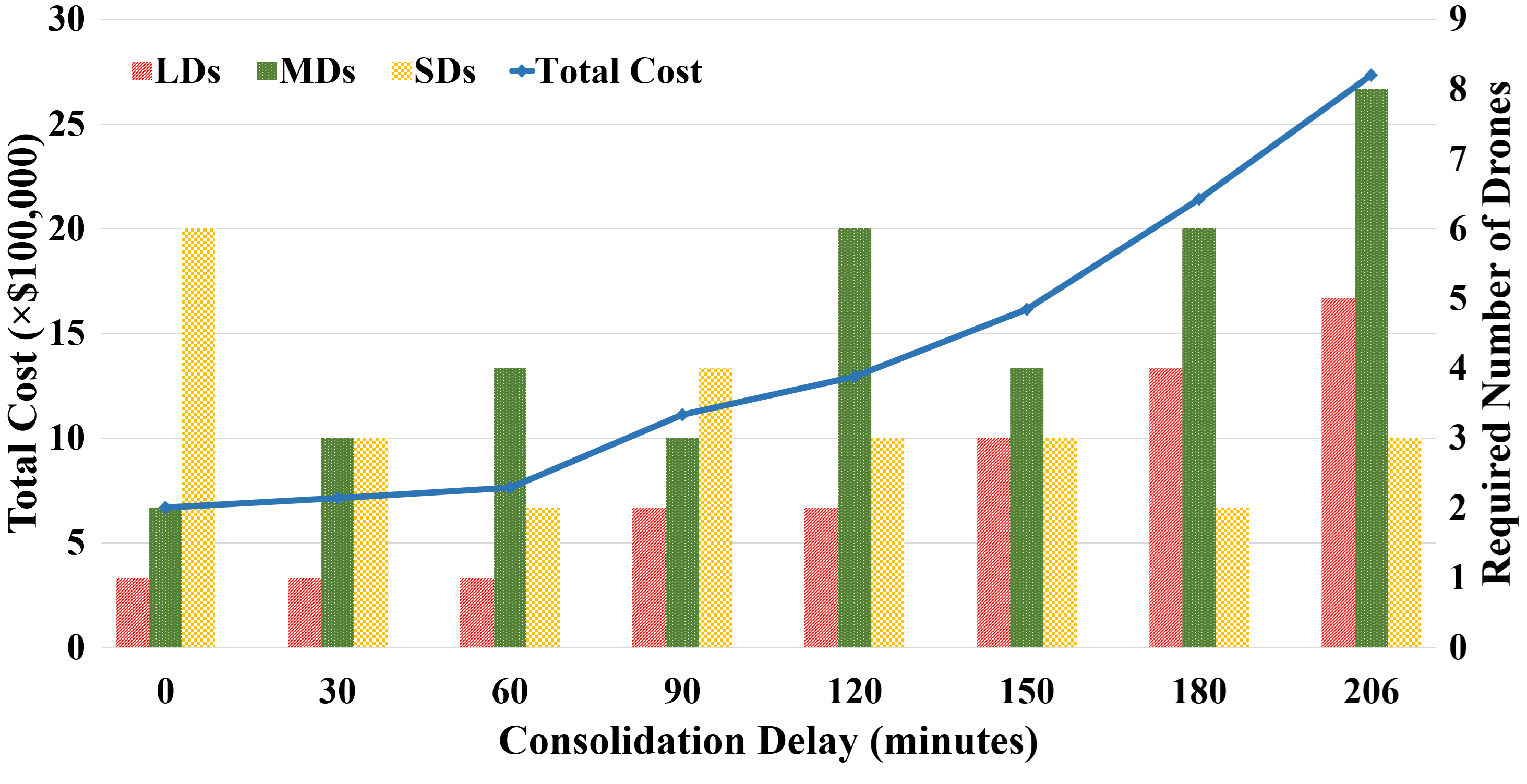}
    \vspace{-10pt}\caption{Effect of consolidation delay on the total cost and the required fleet composition.\vspace{-10pt}}
    \label{fig:Effect_Consolidation_Delay}
\end{figure}

We gradually increase the consolidation delay at the LHs starting from no consolidation delay (i.e., $T_C=0$) until the maximum possible consolidation delay (i.e., 206 minutes in this experiment) beyond which the delivery due time of all delivery locations (i.e., CLs and AMBs) cannot be satisfied. Figure~\ref{fig:Effect_Consolidation_Delay} demonstrates that the total cost increases as the consolidation delay increases. Notably, as the consolidation delay at the LHs increases from no consolidation delay (i.e., 0 minutes) to the maximum possible consolidation delay (i.e., 206 minutes), the total cost increases by 308\%. This is because, as the consolidation delay increases, packages require longer times to be ready before being picked up by the drones in echelon 2. Therefore, the total time required to deliver a package from the CD to the final delivery destination increases. To compensate this longer consolidation delay, the optimal solution of the \textit{HDDNM} problem can either ($i$) use a larger number of LDs to deliver packages earlier to the LHs in echelon 1, or ($ii$) use a larger number of MDs---the fastest drone type in echelon 2---to satisfy the delivery due times, or ($iii$) both. Notably, we see from Figure~\ref{fig:Effect_Consolidation_Delay} that the required number of LDs and MDs increases by 400\% and 300\%, respectively, as the consolidation delay at the LHs increases from no consolidation delay (i.e., 0 minutes) to the maximum possible consolidation delay (i.e., 206 minutes). These insights benefit practitioners in designing their H\&S drone delivery network, depending on their logistical limitations that impose consolidation delays in the delivery process.

\vspace{-5pt}
\subsection{Benefit of Modeling Multi-Trip Delivery Mode of LDs and MDs on the Total Cost and the Required Fleet Composition}

In this section, we present insights into how much is the benefit of modeling the multi-trip delivery mode for LDs and MDs compared to the single-trip delivery mode. To demonstrate the benefit of our comprehensive multi-trip modeling of MDs and LDs, Table~\ref{tab:Effect_Trip_Number} shows the effect of reducing the number of possible trips of LDs and MDs (multi-trip vs. single-trip delivery mode) on the total cost and the required number of drones of each type. In this analysis, we used the medium-sized network containing three LHs, 20 CLs, eight ALs, and five AMBs with a five-hour time window and solved using \textit{M-F}.

\begin{table}[h!]
    \centering
    \caption{Effect of the number of possible trips of LDs and MDs (multi-trip vs. single-trip delivery mode) on the total cost and the required fleet composition. \enquote{-} denotes that there is no limitation on the number of possible trips of LDs and MDs.\vspace{-5pt}}
    \label{tab:Effect_Trip_Number}
    \resizebox{\textwidth}{!}{
        \begin{tabular}{ccccccccc}
            \hline
            \multirow{2}{*}{\makecell[c]{Number of \\ Possible \\ Trips}} & \multicolumn{2}{c}{Cost} & \multicolumn{2}{c}{\makecell[c]{Required Number \\ of LD-Trips}} & \multicolumn{2}{c}{\makecell[c]{Required Number \\ of MD-Trips}} & \multicolumn{2}{c}{\makecell[c]{Required Number \\ of SDs}} \\
            \cmidrule(lr){2-3} \cmidrule(lr){4-5} \cmidrule(lr){6-7} \cmidrule(lr){8-9}
                                                & \makecell[c]{Total \\ Cost (\$)}  & \makecell[c]{Percentage \\ Increase} & \makecell[c]{Total \\ LDs}  &   \makecell[c]{Total \\ Trips}  &  \makecell[c]{Total \\ VTOLs}  &  \makecell[c]{Total \\ Trips}      & DJI    & Tarot  \\ \hline \hline
            -                & 669,912          &  0.0\%       &  1     &   3    &  2     &   6     &   3   &  3 \\
            2                & 1,147,944        &  71.3\%      &  2     &   3    &  3     &   4     &   4    &  3     \\ 
            1                & 1,626,528        &  142.8\%     &  3     &   3    &  4     &   4     &   5    &  2     \\  \hline
        \end{tabular}
    }
\end{table}
We decrease the number of possible trips of LDs and MDs starting from no restriction on the number of possible trips (considered as the baseline number of trips) until the least number of possible trips (i.e., single-trip delivery mode for LDs and MDs). The percentage increase in the total cost due to reducing the number of possible trips in Table~\ref{tab:Effect_Trip_Number} is calculated compared to the total cost of the baseline number of trips. We see from Table~\ref{tab:Effect_Trip_Number} that the total cost and the required fleet size increase as the number of possible trips of LDs and MDs decreases. Notably, as the number of possible trips of LDs and MDs decreases from no restriction (i.e., the multi-trip delivery mode) to one trip (i.e., the single-trip delivery mode), the total cost and the total required number of drones increase by 142.8\% and 55.6\%, respectively. This is because reducing the number of possible trips of LDs and MDs forces the optimal solution of the \textit{HDDNM} problem to use a larger number of LDs and MDs (i.e., a larger required fleet size) to deliver packages satisfying the delivery due times, which increases the total cost.

\vspace{-10pt}
\section{Conclusion}

In this paper, we studied a hierarchical H\&S drone delivery network design with inter-hub and intra-hub coordination to deliver time-sensitive medical items with distinct characteristics---package weights, release times, and delivery due times---to regionally distributed fixed and mobile delivery destinations while minimizing the total investment and operational costs. We considered a heterogeneous fleet of drones with distinct costs, PWCCs, battery capacities, speeds, power consumption, and delivery modes (direct delivery or circular delivery), performing multi-trip delivery operations across the two echelons of the H\&S network. We also modeled adaptive battery swapping for drones, where drones with larger battery capacities and flight ranges perform exogenous battery swaps, whereas drones with smaller battery capacities and flight ranges endogenously decide on battery swaps on an as-needed basis. 

We formulated the problem with these practical aspects as a novel MILP model. To efficiently solve this problem, we proposed an exact solution method integrating a set of new enhancement strategies to reduce the search space and tighten the LP-relaxation, including (1) defining restricted sets and variable fixings, (2) developing problem-specific reformulations, (3) proposing valid inequalities, and (4) designing dynamic cutting plane generation procedures. We also proposed a fast heuristic algorithm to solve large-scale instances of the problem within a reasonable time while sacrificing the solution quality by a small amount. We applied the proposed models and solution methods to a real-life case study of whole blood delivery data from the Interpath Laboratory, Inc.---a healthcare logistics company---located in Pendleton, Oregon, USA. The numerical results based on actual drone flight test data from tests conducted at the Idaho National Laboratory test site demonstrate that our exact solution method is \textit{31} times faster than the Gurobi solver. Moreover, our heuristic solution algorithm is \textit{3,375} times faster than the exact solution method, and solves a real-life problem size of a healthcare network comprising more than 100 nodes within a very short time while providing high-quality solutions.

Moreover, results demonstrate that using a heterogeneous fleet of drones in echelon 2 is 12.5\% more cost-efficient than using a homogeneous fleet of drones. Numerical results show that the total cost and the required fleet size decrease by 76.7\% and 57.1\%, respectively, as the maximum permissible delay increases from 76 minutes to 360 minutes. Furthermore, results indicate that the consolidation delay at regional hospitals due to logistical limitations increases the total cost and the required number of drones. As the consolidation delay increases from \enquote{no delay} to 206 minutes, the total cost, and the required number of LDs and MDs increase by 308\%, 400\%, and 300\%, respectively. We also studied the benefit of modeling the multi-trip delivery mode of LDs and MDs on the total cost and the required fleet composition. Our numerical results demonstrate that, as the delivery mode of LDs and MDs changes from multi-trip to single-trip, the required number of drones and total cost increase by 55.6\% and 142.8\%, respectively.

The outcomes of this research and the insights provided by the numerical experiments are pivotal for the practical implementation of a heterogeneous drone fleet in delivering time-sensitive products across a hierarchical H\&S network. Our proposed exact and heuristic solution methods can serve as fast and reliable decision-support tools for practitioners and logistics system owners in leveraging a heterogeneous fleet of drones to deliver time-sensitive medical items with distinct weights, release times, and delivery due times to regionally distributed fixed and mobile delivery destinations. While the optimization models and the solution methods were developed for the proposed hierarchical H\&S drone delivery network design with novel practical aspects studied in this paper, logistics system owners can use the proposed algorithms on their own data, network configuration, and heterogeneous vehicle fleet for designing a hierarchical H\&S network to deliver time-sensitive medical items to regionally distributed delivery destinations.

This study can be further extended to account for the failure contingencies of drones during the delivery operations. A two-stage robust optimization model can be used to model the problem under $N-k$ failure contingencies, where the first-stage decisions determine the required number of drones of each type, whereas the routing and scheduling of drones are determined in the second stage under the worst-case contingencies. Another possible future research direction is to include the routing of mobile ambulances to the ambulance locations in the hierarchical network, and study the problem as a three-echelon network structure.

\vspace{-10pt}
\section*{Acknowledgment}
\vspace{-5pt}
We acknowledge the support of Interpath Laboratory Inc. for providing us the whole blood delivery data. We acknowledge the support of Mr. Victor Walker, Principal Research Scientist, Idaho National Laboratory (INL), for providing us the drone flight test data. The drone flight test data were obtained from a previous collaborative project between INL and The University of Texas at San Antonio funded by the U.S. Department of Energy Vehicle Technologies Office under the Systems and Modeling for Accelerated Research in Transportation Mobility (SMART) Laboratory Consortium, an Initiative of the Energy Efficient Mobility Systems Program under the Department of Energy Idaho Operations Office Contract No. DE-AC07-05ID14517.

\vspace{-10pt}
\section*{Data Availability}
\vspace{-5pt}
The data that support the findings of this study will be made available by the corresponding author upon reasonable request.

\vspace{-10pt}
\setlength{\bibsep}{1.5pt} 
\bibliographystyle{apalike}
\bibliography{References}

@article{bhuiyan2024aerial,
  title={Aerial drone fleet deployment optimization with endogenous battery replacements for direct delivery of time-sensitive products},
  author={Bhuiyan, Tanveer Hossain and Walker, Victor and Roni, Mohammad and Ahmed, Imtiaz},
  journal={Expert Systems with Applications},
  volume={252},
  pages={124172},
  year={2024},
  publisher={Elsevier},
  note = {\url{https://doi.org/10.1016/j.eswa.2024.124172}}
}

@article{li2025hierarchical,
  title={A hierarchical hub location model for the integrated design of urban and rural logistics networks under demand uncertainty},
  author={Li, Zhi-Chun and Bing, Xue and Fu, Xiaowen},
  journal={Annals of Operations Research},
  volume={348},
  number={2},
  pages={1087--1108},
  year={2025},
  publisher={Springer},
  note = {\url{https://doi.org/10.1007/s10479-023-05189-6}}
}

@article{sun2024multi,
  title={A multi-emission-driven efficient network design for green hub-and-spoke airline networks},
  author={Sun, Mengyuan and Tian, Yong and Dong, Xingchen and Lv, Yangyang and Zhang, Naizhong and Li, Zhixiong and Li, Jiangchen},
  journal={IET Intelligent Transport Systems},
  volume={18},
  number={2},
  pages={346--376},
  year={2024},
  publisher={Wiley Online Library},
  note = {\url{https://doi.org/10.1049/itr2.12455}}
}

@article{lin2008integral,
  title={An integral constrained generalized hub-and-spoke network design problem},
  author={Lin, Cheng-Chang and Chen, Sheu-Hua},
  journal={Transportation Research Part E: Logistics and Transportation Review},
  volume={44},
  number={6},
  pages={986--1003},
  year={2008},
  publisher={Elsevier},
  note = {\url{https://doi.org/10.1016/j.tre.2008.02.001}}
}

@article{carlsson2013euclidean,
  title={Euclidean hub-and-spoke networks},
  author={Carlsson, John Gunnar and Jia, Fan},
  journal={Operations Research},
  volume={61},
  number={6},
  pages={1360--1382},
  year={2013},
  publisher={INFORMS},
  note = {\url{https://doi.org/10.1287/opre.2013.1219}}
}

@article{correia2011hub,
  title={Hub and spoke network design with single-assignment, capacity decisions and balancing requirements},
  author={Correia, Isabel and Nickel, Stefan and Saldanha-da-Gama, Francisco},
  journal={Applied Mathematical Modelling},
  volume={35},
  number={10},
  pages={4841--4851},
  year={2011},
  publisher={Elsevier},
  note = {\url{https://doi.org/10.1016/j.apm.2011.03.046}}
}

@article{dukkanci2017routing,
  title={Routing and scheduling decisions in the hierarchical hub location problem},
  author={Dukkanci, Okan and Kara, Bahar Y},
  journal={Computers \& Operations Research},
  volume={85},
  pages={45--57},
  year={2017},
  publisher={Elsevier},
  note = {\url{https://doi.org/10.1016/j.cor.2017.03.013}}
}

@article{shang2020stochastic,
  title={Stochastic hierarchical multimodal hub location problem for cargo delivery systems: Formulation and algorithm},
  author={Shang, Xiaoting and Yang, Kai and Wang, Weiqiao and Wang, Weiping and Zhang, Haifeng and Celic, Selena},
  journal={IEEE Access},
  volume={8},
  pages={55076--55090},
  year={2020},
  publisher={IEEE},
  note = {\url{https://doi.org/10.1109/ACCESS.2020.2981669}}
}

@article{chou1990hierarchical,
  title={The hierarchical-hub model for airline networks},
  author={Chou, Yue-Hong},
  journal={Transportation Planning and Technology},
  volume={14},
  number={4},
  pages={243--258},
  year={1990},
  publisher={Taylor \& Francis},
  note = {\url{https://doi.org/10.1080/03081069008717429}}
}

@article{lin2004hierarchical,
  title={The hierarchical network design problem for time-definite express common carriers},
  author={Lin, Cheng-Chang and Chen, Sheu-Hua},
  journal={Transportation Research Part B: Methodological},
  volume={38},
  number={3},
  pages={271--283},
  year={2004},
  publisher={Elsevier},
  note = {\url{https://doi.org/10.1016/S0191-2615(03)00013-4}}
}

@article{santos2015branch,
  title={A branch-and-cut-and-price algorithm for the two-echelon capacitated vehicle routing problem},
  author={Santos, Fernando Afonso and Mateus, Geraldo Robson and da Cunha, Alexandre Salles},
  journal={Transportation Science},
  volume={49},
  number={2},
  pages={355--368},
  year={2015},
  publisher={INFORMS},
  note = {\url{https://doi.org/10.1287/trsc.2013.0500}}
}

@article{wu2023branch,
  title={A branch-and-price algorithm for two-echelon electric vehicle routing problem},
  author={Wu, Zhiguo and Zhang, Juliang},
  journal={Complex \& Intelligent Systems},
  volume={9},
  number={3},
  pages={2475--2490},
  year={2023},
  publisher={Springer},
  note = {\url{https://doi.org/10.1007/s40747-021-00403-z}}
}

@article{marques2022branch,
  title={A branch-cut-and-price approach for the single-trip and multi-trip two-echelon vehicle routing problem with time windows},
  author={Marques, Guillaume and Sadykov, Ruslan and Dupas, R{\'e}my and Deschamps, Jean-Christophe},
  journal={Transportation Science},
  volume={56},
  number={6},
  pages={1598--1617},
  year={2022},
  publisher={INFORMS},
  note = {\url{https://doi.org/10.1287/trsc.2022.1136}}
}

@article{mhamedi2022branch,
  title={A branch-price-and-cut algorithm for the two-echelon vehicle routing problem with time windows},
  author={Mhamedi, Tayeb and Andersson, Henrik and Cherkesly, Maril{\`e}ne and Desaulniers, Guy},
  journal={Transportation Science},
  volume={56},
  number={1},
  pages={245--264},
  year={2022},
  publisher={Informs},
  note = {\url{https://doi.org/10.1287/trsc.2021.1092}}
}

@article{dellaert2021multi,
  title={A multi-commodity two-Echelon capacitated vehicle routing problem with time windows: Model formulations and solution approach},
  author={Dellaert, Nico and Van Woensel, Tom and Crainic, Teodor Gabriel and Saridarq, Fardin Dashty},
  journal={Computers \& Operations Research},
  volume={127},
  pages={105154},
  year={2021},
  publisher={Elsevier},
  note = {\url{https://doi.org/10.1016/j.cor.2020.105154}}
}

@article{baldacci2013exact,
  title={An exact algorithm for the two-echelon capacitated vehicle routing problem},
  author={Baldacci, Roberto and Mingozzi, Aristide and Roberti, Roberto and Calvo, Roberto Wolfler},
  journal={Operations Research},
  volume={61},
  number={2},
  pages={298--314},
  year={2013},
  publisher={INFORMS},
  note = {\url{https://doi.org/10.1287/opre.1120.1153}}
}

@article{santos2013branch,
  title={Branch-and-price algorithms for the two-echelon capacitated vehicle routing problem},
  author={Santos, Fernando Afonso and da Cunha, Alexandre Salles and Mateus, Geraldo Robson},
  journal={Optimization Letters},
  volume={7},
  number={7},
  pages={1537--1547},
  year={2013},
  publisher={Springer},
  note = {\url{https://doi.org/10.1007/s11590-012-0568-3}}
}

@article{dellaert2019branch,
  title={Branch-and-price-based algorithms for the two-echelon vehicle routing problem with time windows},
  author={Dellaert, Nico and Dashty Saridarq, Fardin and Van Woensel, Tom and Crainic, Teodor Gabriel},
  journal={Transportation Science},
  volume={53},
  number={2},
  pages={463--479},
  year={2019},
  publisher={INFORMS},
  note = {\url{https://doi.org/10.1287/trsc.2018.0844}}
}

@article{thomadsen2005hierarchical,
  title={Hierarchical ring network design using branch-and-price},
  author={Thomadsen, Tommy and Stidsen, Thomas},
  journal={Telecommunication Systems},
  volume={29},
  number={1},
  pages={61--76},
  year={2005},
  publisher={Springer},
  note = {\url{https://doi.org/10.1007/s11235-005-6631-y}}
}

@article{wang2018two,
  title={Two-echelon location-routing optimization with time windows based on customer clustering},
  author={Wang, Yong and Assogba, Kevin and Liu, Yong and Ma, Xiaolei and Xu, Maozeng and Wang, Yinhai},
  journal={Expert Systems with Applications},
  volume={104},
  pages={244--260},
  year={2018},
  publisher={Elsevier},
  note = {\url{https://doi.org/10.1016/j.eswa.2018.03.018}}
}

@article{pinto2020network,
  title={A network design model for a meal delivery service using drones},
  author={Pinto, Roberto and Zambetti, Michela and Lagorio, Alexandra and Pirola, Fabiana},
  journal={International Journal of Logistics Research and Applications},
  volume={23},
  number={4},
  pages={354--374},
  year={2020},
  publisher={Taylor \& Francis},
  note = {\url{https://doi.org/10.1080/13675567.2019.1696290}}
}

@article{jiu2024benders,
  title={Benders decomposition for robust distribution network design and operations in online retailing},
  author={Jiu, Song and Wang, Dan and Ma, Zujun},
  journal={European Journal of Operational Research},
  volume={315},
  number={3},
  pages={1069--1082},
  year={2024},
  publisher={Elsevier},
  note = {\url{https://doi.org/10.1016/j.ejor.2024.01.046}}
}

@article{zou2024delivery,
  title={Delivery network design of a locker-drone delivery system},
  author={Zou, Bipan and Wu, Siqing and Gong, Yeming and Yuan, Zhe and Shi, Yuqian},
  journal={International Journal of Production Research},
  volume={62},
  number={11},
  pages={4097--4121},
  year={2024},
  publisher={Taylor \& Francis},
  note = {\url{https://doi.org/10.1080/00207543.2023.2254402}}
}

@article{pinto2022point,
  title={Point-to-point drone-based delivery network design with intermediate charging stations},
  author={Pinto, Roberto and Lagorio, Alexandra},
  journal={Transportation Research Part C: Emerging Technologies},
  volume={135},
  pages={103506},
  year={2022},
  publisher={Elsevier},
  note = {\url{https://doi.org/10.1016/j.trc.2021.103506}}
}

@article{hu2024drone,
  title={Drone-based instant delivery hub-and-spoke network optimization},
  author={Hu, Zhi-Hua and Huang, Yan-Ling and Li, Yao-Na and Bao, Xiao-Qiong},
  journal={Drones},
  volume={8},
  number={6},
  pages={247},
  year={2024},
  publisher={MDPI},
  note = {\url{https://doi.org/10.3390/drones8060247}}
}

@article{gao2022optimizing,
  title={Optimizing the hub-and-spoke network with drone-based traveling salesman problem},
  author={Gao, Chao-Feng and Hu, Zhi-Hua and Wang, Yao-Zong},
  journal={Drones},
  volume={7},
  number={1},
  pages={6},
  year={2022},
  publisher={MDPI},
  note = {\url{https://doi.org/10.3390/drones7010006}}
}

@article{enayati2023multimodal,
  title={Multimodal vaccine distribution network design with drones},
  author={Enayati, Shakiba and Li, Haitao and Campbell, James F and Pan, Deng},
  journal={Transportation Science},
  volume={57},
  number={4},
  pages={1069--1095},
  year={2023},
  publisher={INFORMS},
  note = {\url{https://doi.org/10.1287/trsc.2023.1205}}
}

@article{kar2024optimal,
  title={Optimal multimodal multi-echelon vaccine distribution network design for low and medium-income countries with manufacturing infrastructure during healthcare emergencies},
  author={Kar, Biswajit and Jenamani, Mamata},
  journal={International Journal of Production Economics},
  volume={273},
  pages={109282},
  year={2024},
  publisher={Elsevier},
  note = {\url{https://doi.org/10.1016/j.ijpe.2024.109282}}
}

@article{gentili2022locating,
  title={Locating platforms and scheduling a fleet of drones for emergency delivery of perishable items},
  author={Gentili, Monica and Mirchandani, Pitu B and Agnetis, Alessandro and Ghelichi, Zabih},
  journal={Computers \& Industrial Engineering},
  volume={168},
  pages={108057},
  year={2022},
  publisher={Elsevier},
  note = {\url{https://doi.org/10.1016/j.cie.2022.108057}}
}

@article{ghelichi2021logistics,
  title={Logistics for a fleet of drones for medical item delivery: A case study for {Louisville}, {KY}},
  author={Ghelichi, Zabih and Gentili, Monica and Mirchandani, Pitu B},
  journal={Computers \& Operations Research},
  volume={135},
  pages={105443},
  year={2021},
  publisher={Elsevier},
  note = {\url{https://doi.org/10.1016/j.cor.2021.105443}}
}

@article{kim2017drone,
  title={Drone-aided healthcare services for patients with chronic diseases in rural areas},
  author={Kim, Seon Jin and Lim, Gino J and Cho, Jaeyoung and C{\^o}t{\'e}, Murray J},
  journal={Journal of Intelligent \& Robotic Systems},
  volume={88},
  number={1},
  pages={163--180},
  year={2017},
  publisher={Springer},
  note = {\url{https://doi.org/10.1007/s10846-017-0548-z}}
}

@article{rabta2018drone,
  title={A drone fleet model for last-mile distribution in disaster relief operations},
  author={Rabta, Boualem and Wankm{\"u}ller, Christian and Reiner, Gerald},
  journal={International Journal of Disaster Risk Reduction},
  volume={28},
  pages={107--112},
  year={2018},
  publisher={Elsevier},
  note = {\url{https://doi.org/10.1016/j.ijdrr.2018.02.020}}
}

@article{wen2016multi,
  title={Multi-objective algorithm for blood supply via unmanned aerial vehicles to the wounded in an emergency situation},
  author={Wen, Tingxi and Zhang, Zhongnan and Wong, Kelvin KL},
  journal={Plos One},
  volume={11},
  number={5},
  pages={e0155176},
  year={2016},
  publisher={Public Library of Science San Francisco, CA USA},
  note = {\url{https://doi.org/10.1371/journal.pone.0155176}}
}

@article{hemmelmayr2012adaptive,
  title={An adaptive large neighborhood search heuristic for two-echelon vehicle routing problems arising in city logistics},
  author={Hemmelmayr, Vera C and Cordeau, Jean-Fran{\c{c}}ois and Crainic, Teodor Gabriel},
  journal={Computers \& Operations Research},
  volume={39},
  number={12},
  pages={3215--3228},
  year={2012},
  publisher={Elsevier},
  note = {\url{https://doi.org/10.1016/j.cor.2012.04.007}}
}

@article{ulmer2018same,
  title={Same-day delivery with heterogeneous fleets of drones and vehicles},
  author={Ulmer, Marlin W and Thomas, Barrett W},
  journal={Networks},
  volume={72},
  number={4},
  pages={475--505},
  year={2018},
  publisher={Wiley Online Library},
  note = {\url{https://doi.org/10.1002/net.21855}}
}

@article{cokyasar2021designing,
  title={Designing a drone delivery network with automated battery swapping machines},
  author={Cokyasar, Taner and Dong, Wenquan and Jin, Mingzhou and Verbas, {\.I}smail {\"O}mer},
  journal={Computers \& Operations Research},
  volume={129},
  pages={105177},
  year={2021},
  publisher={Elsevier},
  note = {\url{https://doi.org/10.1016/j.cor.2020.105177}}
}

@article{amirsahami2023hierarchical,
  title={A hierarchical model for strategic and operational planning in blood transportation with drones},
  author={Amirsahami, Amirali and Barzinpour, Farnaz and Pishvaee, Mir Saman},
  journal={Plos One},
  volume={18},
  number={9},
  pages={e0291352},
  year={2023},
  publisher={Public Library of Science San Francisco, CA USA},
  note = {\url{https://doi.org/10.1371/journal.pone.0291352}}
}

@article{palmer2024unique,
  title={Unique Health Care Delivery Considerations in Rural {America}},
  author={Palmer, Kaitlyn and Cochran, Jill and McGinley, Marisa},
  journal={International Journal of MS Care},
  volume={27},
  number={Theme},
  pages={T2--T6},
  year={2024},
  publisher={The Consortium of Multiple Sclerosis Centers},
  note = {\url{https://doi.org/10.7224/1537-2073.2024-081}}
}

@article{koshta2022evaluating,
  title={Evaluating barriers to the adoption of delivery drones in rural healthcare supply chains: Preparing the healthcare system for the future},
  author={Koshta, Nitin and Devi, Yashoda and Chauhan, Chetna},
  journal={IEEE Transactions on Engineering Management},
  volume={71},
  pages={13096--13108},
  year={2022},
  publisher={IEEE},
  note = {\url{https://doi.org/10.1109/TEM.2022.3210121}}
}

@book{national2022building,
  title={Building resilience into the nation's medical product supply chains},
  author={{National Academies of Sciences, Engineering, and Medicine}},
  year={2022},
  publisher = {National Academies Press},
  note = {\url{https://doi.org/10.17226/26420}}
}

@misc{AmericaBlood2022,
  author       = {{America's Blood Centers}},
  title        = {Prioritize Blood Donation as a National Imperative},
  howpublished = {\url{https://americasblood.org/wp-content/uploads/2022/09/Prioritize-Blood-Donation-as-a-National-Imperative.pdf}},
  year         = {2022},
  note         = {Accessed: November 20, 2025},
}

@misc{FDA2025EmergencyDevices,
  author       = {{U.S. Food and Drug Administration (FDA)}},
  title        = {Emergency Preparedness and Medical Devices: Supply Chain Recommendations for Health Care Providers and Manufacturers},
  howpublished = {\url{https://www.fda.gov/medical-devices/emergency-situations-medical-devices/emergency-preparedness-and-medical-devices-supply-chain-recommendations-health-care-providers-device}},
  year         = {2025},
  note         = {Accessed: November 20, 2025},
}

@article{varela2019transportation,
  title={Transportation barriers to access health care for surgical conditions in {Malawi} a cross sectional nationwide household survey},
  author={Varela, Carlos and Young, Sven and Mkandawire, Nyengo and Groen, Reinou S and Banza, Leonard and Viste, Asgaut},
  journal={BMC Public Health},
  volume={19},
  number={1},
  pages={264},
  year={2019},
  publisher={Springer},
  note = {\url{https://doi.org/10.1186/s12889-019-6577-8}}
}

@article{eisner2024transportation,
  title={Transportation and equipment needs for emergency medical services development in low- and middle-income countries},
  author={Eisner, Zachary J and Smith, Nathanael J and Wylie, Craig},
  journal={Surgery},
  volume={176},
  number={2},
  pages={521--523},
  year={2024},
  publisher={Elsevier},
  note = {\url{https://doi.org/10.1016/j.surg.2024.03.050}}
}

@article{kellermann2020drones,
  title={Drones for parcel and passenger transportation: A literature review},
  author={Kellermann, Robin and Biehle, Tobias and Fischer, Liliann},
  journal={Transportation Research Interdisciplinary Perspectives},
  volume={4},
  pages={100088},
  year={2020},
  publisher={Elsevier},
  note = {https://doi.org/10.1016/j.trip.2019.100088}
}

@article{mohsan2023unmanned,
  title={Unmanned aerial vehicles ({UAV}s): Practical aspects, applications, open challenges, security issues, and future trends},
  author={Mohsan, Syed Agha Hassnain and Othman, Nawaf Qasem Hamood and Li, Yanlong and Alsharif, Mohammed H and Khan, Muhammad Asghar},
  journal={Intelligent Service Robotics},
  volume={16},
  number={1},
  pages={109--137},
  year={2023},
  publisher={Springer},
  note = {\url{https://doi.org/10.1007/s11370-022-00452-4}}
}

@article{jairoun2025evolution,
  title={The evolution of medication delivery via drones: Revolutionizing healthcare logistics},
  author={Jairoun, Ammar A and Al-Hemyari, Sabaa S and Shahwan, Moyad and Al-Ghananeem, Abeer M and El-Dahiyat, Faris and Al-Salmi, Sondos and Babar, Zaheer-Ud-Din},
  journal={Journal of Pharmaceutical Policy and Practice},
  volume={18},
  number={1},
  pages={2519137},
  year={2025},
  publisher={Taylor \& Francis},
  note = {\url{https://doi.org/10.1080/20523211.2025.2519137}}
}

@misc{wingcopter2020vanuatu,
  author       = {Wingcopter},
  title        = {Vaccine delivery service in {Vanuatu}},
  year         = {2020},
  howpublished = {\url{https://wingcopter.com/project/story-vanuatu}},
  note         = {Accessed: August 20, 2024}
}

@article{ackerman2019blood,
  title={The blood is here: Zipline's medical delivery drones are changing the game in {Rwanda}},
  author={Ackerman, Evan and Koziol, Michael},
  journal={IEEE Spectrum},
  volume={56},
  number={5},
  pages={24--31},
  year={2019},
  publisher={IEEE},
  note = {\url{https://doi.org/10.1109/MSPEC.2019.8701196}}
}

@misc{flyzipline_website,
  author       = {Zipline},
  title        = {FlyZipline Official Website},
  year = {2025},
  howpublished = {\url{https://www.flyzipline.com/}},
  note         = {Accessed: March 21, 2025}
}

@misc{Simons_Rwanda,
    author = {Alexander Onukwue},
    title = {Zipline needs {Nigeria} to support its drone delivery medical service---as {Rwanda} and {Ghana} did},
    year = {2022},
    howpublished = {\url{https://qz.com/africa/2086752/what-zipline-brings-to-nigeria-after-5-years-in-rwanda-and-ghana}},
    note = {Accessed: November 26, 2025}
}

@misc{pwc2024drone,
  title        = {Drone Deliveries: Taking Retail and Logistics to New Heights},
  author       = {PwC Drone Powered Solutions},
  year         = 2024,
  howpublished = {\url{https://cee.pwc.com/drone-powered-solutions/drone-deliveries-taking-retail-and-logistics-to-new-heights.html}},
  note         = {Accessed: March 15, 2025}
}

@misc{karsten2018drones,
  author = {Jack Karsten and Darrell M. West},
  title = {How emergency responders are using drones to save lives},
  year = {2018},
  howpublished = {\url{https://www.brookings.edu/articles/how-emergency-responders-are-using-drones-to-save-lives/}},
  note = {Accessed: October 14, 2025}
}

@article{hosseini2026branch,
title = {A branch-and-cut algorithm for routing a heterogeneous drone-ground vehicle fleet to deliver time-sensitive products},
journal = {Transportation Research Part C: Emerging Technologies},
author = {Sayed Hamid {Hosseini Dolatabadi} and Tanveer Hossain Bhuiyan and Waquar Kaleem and Anirudh Subramanyam and Taner Cokyasar},
volume = {183},
pages = {105454},
year = {2026},
issn = {0968-090X},
doi = {https://doi.org/10.1016/j.trc.2025.105454},
url = {https://www.sciencedirect.com/science/article/pii/S0968090X25004589},
note = {\url{https://doi.org/10.1016/j.trc.2025.105454}}
}

@article{murray2015flying,
  title={The flying sidekick traveling salesman problem: Optimization of drone-assisted parcel delivery},
  author={Murray, Chase C and Chu, Amanda G},
  journal={Transportation Research Part C: Emerging Technologies},
  volume={54},
  pages={86--109},
  year={2015},
  publisher={Elsevier},
  note = {\url{https://doi.org/10.1016/j.trc.2015.03.005}}
}

@article{li2021ground,
  title={Ground-vehicle and unmanned-aerial-vehicle routing problems from two-echelon scheme perspective: A review},
  author={Li, Hongqi and Chen, Jun and Wang, Feilong and Bai, Ming},
  journal={European Journal of Operational Research},
  volume={294},
  number={3},
  pages={1078--1095},
  year={2021},
  publisher={Elsevier},
  note = {\url{https://doi.org/10.1016/j.ejor.2021.02.022}}
}

@article{hong2018range,
  title={A range-restricted recharging station coverage model for drone delivery service planning},
  author={Hong, Insu and Kuby, Michael and Murray, Alan T},
  journal={Transportation Research Part C: Emerging Technologies},
  volume={90},
  pages={198--212},
  year={2018},
  publisher={Elsevier},
  note = {\url{https://doi.org/10.1016/j.trc.2018.02.017}}
}

@article{cokyasar2021optimization,
  title={Optimization of battery swapping infrastructure for e-commerce drone delivery},
  author={Cokyasar, Taner},
  journal={Computer Communications},
  volume={168},
  pages={146--154},
  year={2021},
  publisher={Elsevier},
  note = {\url{https://doi.org/10.1016/j.comcom.2020.12.015}}
}

@article{mirzapour2024optimizing,
  title={Optimizing last-mile delivery services: A robust truck-drone cooperation model and hybrid metaheuristic algorithm},
  author={Mirzapour Al-e-Hashem, Seyed Mohammad Javad and Hejazi, Taha-Hossein and Haghverdizadeh, Ghazal and Shidpour, Mohsen},
  journal={Annals of Operations Research},
  pages={1--31},
  year={2024},
  publisher={Springer},
  note = {\url{https://doi.org/10.1007/s10479-024-06164-5}}
}

@article{moshref2021comparative,
  title={A comparative analysis of synchronized truck-and-drone delivery models},
  author={Moshref-Javadi, Mohammad and Hemmati, Ahmad and Winkenbach, Matthias},
  journal={Computers \& Industrial Engineering},
  volume={162},
  pages={107648},
  year={2021},
  publisher={Elsevier},
  note = {\url{https://doi.org/10.1016/j.cie.2021.107648}}
}

@article{hassan2022charging,
  title={Charging station distribution optimization using drone fleet in a disaster},
  author={Hassan, Zohaib and Ali Shah, Syed Irtiza and Sarwar Rana, Ahsan},
  journal={Journal of Robotics},
  volume={2022},
  number={1},
  pages={7329346},
  year={2022},
  publisher={Wiley Online Library},
  note = {\url{https://doi.org/10.1155/2022/7329346}}
}

@article{glick2022case,
  title={Case study of drone delivery reliability for time-sensitive medical supplies with stochastic demand and meteorological conditions},
  author={Glick, Travis B and Figliozzi, Miguel A and Unnikrishnan, Avinash},
  journal={Transportation Research Record},
  volume={2676},
  number={1},
  pages={242--255},
  year={2022},
  publisher={SAGE Publications Sage CA: Los Angeles, CA},
  note = {\url{https://doi.org/10.1177/03611981211036685}}
}

@article{kaparis2010separation,
  title={Separation algorithms for 0-1 knapsack polytopes},
  author={Kaparis, Konstantinos and Letchford, Adam N},
  journal={Mathematical Programming},
  volume={124},
  number={1},
  pages={69--91},
  year={2010},
  publisher={Springer},
  note = {\url{https://doi.org/10.1007/s10107-010-0359-5}}
}

@manual{gurobi2023,
  title        = {Gurobi Optimizer Reference Manual},
  author       = {{Gurobi}},
  year         = {2025},
  note         = {\url{https://docs.gurobi.com/projects/optimizer/en/current/index.html}},
}

@misc{unisHS,
  author = {{UNIS}},
  title = {Transportation Modeling vs Hub and Spoke Network: A Comprehensive Comparison},
  year = {2025},
  howpublished = {\url{https://www.unisco.com/comparison/hub-and-spoke-network-vs-transportation-modeling}},
  note = {Accessed: December 29, 2025}
}

@article{ogazon2025designing,
  title={Designing hub-based regional transportation networks with service level constraints},
  author={Ogaz{\'o}n, Esteban and Anaya-Arenas, Ana Mar{\'\i}a and Ruiz, Angel},
  journal={Transportation Research Part E: Logistics and Transportation Review},
  volume={195},
  pages={103991},
  year={2025},
  publisher={Elsevier},
  note = {\url{https://doi.org/10.1016/j.tre.2025.103991}}
}

@misc{Aegis2022Blood,
  author = {Aegis Scientific},
  title = {Blood Banks: Storage Requirements},
  year = {2022},
  howpublished = {\url{https://www.aegisfridge.com/2022/10/03/blood-banks-storage-requirements/}},
  note = {Accessed: January 20, 2026}
}

@misc{Fuel2011AMB,
  author = {Jeremy J. Hess and 
Lawrence A. Greenberg},
  title = {Fuel Use in a Large, Dynamically Deployed Emergency Medical Services System},
  year = {2011},
  howpublished = {\url{https://www.cambridge.org/core/journals/prehospital-and-disaster-medicine/article/fuel-use-in-a-large-dynamically-deployed-emergency-medical-services-system/89B4B55BF605C9B594631E8C90CF3030}},
  note = {Accessed: January 20, 2026}
}

@misc{Blood2025Temp,
  author = {O.J. Bradshaw},
  title = {Ensuring Proper Storage and Temperatures for the Blood Bank},
  year = {2025},
  howpublished = {\url{https://achc.org/ensuring-proper-storage-and-temperatures-for-the-blood-bank/}},
  note = {Accessed: January 20, 2026}
}

@article{bhuiyan2025optimization,
  title={Optimization of a Mixed Fleet of Aerial Drones for Medical Supplies: A Case Study of Blood Delivery Logistics},
  author={Bhuiyan, Tanveer Hossain and Dolatabadi, Sayed Hamid Hosseini and Uddin, Md Jalal and Walker, Victor},
  journal={Transportation Research Record},
  pages={03611981251394973},
  year={2025},
  publisher={SAGE Publications Sage CA: Los Angeles, CA},
  note = {\url{https://doi.org/10.1177/03611981251394973}}
}

\appendix 
\newpage
\setcounter{page}{1}
\renewcommand{\thefigure}{\Alph{section}.\arabic{figure}}
\renewcommand{\thetable}{\Alph{section}.\arabic{table}}
\renewcommand{\theequation}{\Alph{section}.\arabic{equation}}
\renewcommand{\thepage}{A-\arabic{page}}
\renewcommand{\thealgorithm}{\Alph{section}.\arabic{algorithm}}
\newpage
\section*{Appendices}

\setcounter{equation}{0}
\setcounter{figure}{0}
\setcounter{table}{0}
\section{Endogenous Battery Swapping Constraints}
\label{Appendix_Endogenous}

Constraints \eqref{eq:End-CL-1} - \eqref{eq:End-CL-7} and \eqref{eq:End-AMB-1} - \eqref{eq:End-AMB-7} model the endogenous battery swapping of SDs while serving the CLs and the AMBs, respectively. Here, $\zeta_{ih/lh}^t$ and $\zeta_{ih/lh}^{' t}$ are continuous variables representing the actual and potential remaining energy, respectively, in the battery of an SD of type $t$ upon returning to LH $h$ from CL $i$/AL $l$. Furthermore, $A_u^t$ and $A_\ell^t$ are the upper and lower bounds of $(\zeta_{ih}^{' t} - B_{min}^t)$, respectively, and $U^t$ is an upper bound on $(B_{max}^t - \zeta_{ih}^t)$.
We refer the readers to \citet{bhuiyan2024aerial} for a detailed explanation of the endogenous battery swapping constraints of drones.

\begin{align}
    & \zeta_{jh}^{' t}\leq B_{max}^{t}- \rho_{jh}^{t}+M(1-\theta_{hj}^{t}), \quad\forall h\in\mathcal H,\, j\in\mathcal{I},\, t\in\mathcal{T} \label{eq:End-CL-1}\\
    & \zeta_{jh}^{' t}\leq \zeta_{ih}^{t}-\rho_{jh}^{t}+M(1-\theta_{ij}^{t}), \quad\forall h\in\mathcal H,\, (i,j)\in\mathcal A_h,\, t\in\mathcal{T} \label{eq:End-CL-2}\\
    & \zeta_{jh}^{' t}-B_{min}^{t}\geq A_{\ell}^{t} \cdot s_{ih}^{t}+M(\theta_{ij}^{t}-1), \quad\forall h\in\mathcal H,\, (i,j)\in\mathcal A_h,\, t\in\mathcal{T} \label{eq:End-CL-3}\\
    & \zeta_{jh}^{' t}-B_{min}^{t}\leq A_{u}^{t}(1-s_{ih}^{t})+M(1-\theta_{ij}^{t}), \quad\forall h\in\mathcal H,\, (i,j)\in\mathcal A_h,\, t\in\mathcal{T} \label{eq:End-CL-4}\\
    & \zeta_{jh}^{t}\leq B_{max}^{t}-\rho_{jh}^{t}+U^{t}(1-s_{ih}^{t})+M(1-\theta_{ij}^{t}), \quad\forall h\in\mathcal H,\, (i,j)\in\mathcal A_h,\, t\in\mathcal{T} \label{eq:End-CL-5}\\
    & \zeta_{jh}^{t}\leq \zeta_{ih}^{t}-\rho_{jh}^{t}+U^{t}s_{ih}^{t}+M(1-\theta_{ij}^{t}), \quad\forall h\in\mathcal H,\, (i,j)\in\mathcal A_h,\, t\in\mathcal{T} \label{eq:End-CL-6}\\
    & \zeta_{jh}^{t}\leq B_{max}^{t}-\rho_{jh}^{t}+U^{t}s_{hh}^{t}+M(1-\theta_{hj}^{t}), \quad\forall h\in\mathcal H,\, j\in\mathcal{I},\, t\in\mathcal{T} \label{eq:End-CL-7} \\
    & \zeta_{oh}^{' t}\leq B_{max}^{t}-\rho_{oh}^{t}+M(1-\theta_{ho}^{t}), \quad\forall h\in\mathcal H,\, o\in\mathcal{L},\, t\in\mathcal{T} \label{eq:End-AMB-1}\\
    & \zeta_{oh}^{' t}\leq \zeta_{lh}^{t}-\rho_{oh}^{t}+M(1-\theta_{lo}^{t}), \quad\forall h\in\mathcal H,\, (l,o)\in\mathcal B_h,\, t\in\mathcal{T} \label{eq:End-AMB-2}\\
    & \zeta_{oh}^{' t}-B_{min}^{t}\geq A_{\ell}^{t} \cdot s_{lh}^{t}+M(\theta_{lo}^{t}-1), \quad\forall h\in\mathcal H,\, (l,o)\in\mathcal B_h,\, t\in\mathcal{T} \label{eq:End-AMB-3}\\
    & \zeta_{oh}^{' t}-B_{min}^{t}\leq A_{u}^{t}(1-s_{lh}^{t})+M(1-\theta_{lo}^{t}), \quad\forall h\in\mathcal H,\, (l,o)\in\mathcal B_h,\, t\in\mathcal{T} \label{eq:End-AMB-4}\\
    & \zeta_{oh}^{t}\leq B_{max}^{t}-\rho_{oh}^{t}+U^{t}(1-s_{lh}^{t})+M(1-\theta_{lo}^{t}), \quad\forall h\in\mathcal H,\, (l,o)\in\mathcal B_h,\, t\in\mathcal{T} \label{eq:End-AMB-5}\\
    & \zeta_{oh}^{t}\leq \zeta_{lh}^{t}-\rho_{oh}^{t}+U^{t}s_{lh}^{t}+M(1-\theta_{lo}^{t}), \quad\forall h\in\mathcal H,\, (l,o)\in\mathcal B_h,\, t\in\mathcal{T} \label{eq:End-AMB-6}\\
    & \zeta_{oh}^{t}\leq B_{max}^{t}-\rho_{oh}^{t}+U^{t}s_{hh}^{t}+M(1-\theta_{ho}^{t}), \quad\forall h\in\mathcal H,\, o\in\mathcal{L},\, t\in\mathcal{T} \label{eq:End-AMB-7} 
\end{align}

\setcounter{equation}{0}
\setcounter{figure}{0}
\setcounter{table}{0}
\section{Proofs}
\label{Appendix_Proof}

\renewcommand{\theprop}{\ref{prop:NP-hard}}

\begin{prop}
    The \textit{HDDNM} problem is NP-hard.
\end{prop}

\begin{proof}
    Consider a directed network graph $G = (\mathcal V, \mathcal E)$, where $\mathcal V$ and $\mathcal E$ are the sets of nodes and edges, respectively. We construct a special case of our \textit{HDDNM} problem on $G$ as follows. We only consider the CD in layer 1 and the CLs in layer 3 of the network, ignoring the hierarchical design. Therefore, $\mathcal V = \{0 \} \cup \mathcal I $, and $\mathcal E = \{(0, i): i\in\mathcal I\} \cup \{(i, j) \in \mathcal I \times \mathcal I\} \cup \{(i, S): i\in\mathcal I\}$. Furthermore, the fleet comprises only a single LD, performing a single circular-delivery trip. The energy consumption of the LD on each edge $e\in\mathcal E$, denoted as $q_e$, is proportional to the distance, i.e., $q_e = k \cdot d_e$. Here, $d_e$ is the length of edge $e$, and $k$ is the proportional parameter. The battery capacity of the single LD is greater than the total energy consumption on all edges, i.e., $B_{max}^{LD} - B_{min}^{LD} \gg \sum_{e\in\mathcal E} (k \cdot d_e)$. Furthermore, the PWCC of the LD is greater than the total weight of all the packages, i.e., $P_{max}^{LD} \gg \sum_{i\in\mathcal I} P_i$. We also assume that all the packages are released at the beginning of the planning horizon (i.e., $R_i = 0, \forall i\in\mathcal I$). Additionally, no delivery due time is considered, i.e., each package can be delivered any time during the planning horizon. Therefore, to deliver all the packages, the goal is to visit all the CLs exactly once by a single LD, while minimizing the total distance traveled. If the \textit{HDDNM} problem is solved to optimality, its solution yields an optimal solution to the Euclidean Traveling Salesman Problem (TSP), where a vehicle visits each node $v\in \mathcal V$ in a complete graph $G$ defined with Euclidean distance while minimizing the total distance traveled. It is well-established that the TSP is NP-hard, which implies that the \textit{HDDNM} problem is also NP-hard.
\end{proof}

\renewcommand{\theprop}{\ref{prop:Return_Time}}

\begin{prop} 
    For any drone performing multiple trips, a feasible schedule in which the drone returns to either the CD (in echelon 1) or the origin LH (in echelon 2) immediately after serving all the locations in the sequence dominates an identical schedule with an idle time during the trip.
\end{prop}

\begin{proof}
    Let $P = (v_1, v_2, ..., v_m)$ be a feasible circular delivery schedule (i.e., the sequence of visited locations) for a given drone-trip $dg$. The total completion time of serving all the locations in this sequence can be computed as $T_P = T_L^d + |P| \cdot T_U^d + \sum_{i=1}^{m-1} T_{v_i, v_{i+1}}^d $. Here, the total completion time includes: ($i$) the loading time of all the packages at the beginning of the trip, ($ii$) the total time to unload each package to the corresponding location, and ($iii$) the total travel time along the sequence. Moreover, consider $T_1$ and $T_I$ as the departure time of drone-trip $dg$ and the total idle time during the trip, respectively. 

    Let $F$ be the set of feasible solutions of the \textit{BaM}, and suppose there exist two feasible solutions $x_1, x_2 \in F$, where ($i$) in $x_1$, drone-trip $dg$ follows delivery schedule $P$ and returns at $T_R = T_1 + T_P$ (i.e., immediately after serving all the locations); and ($ii$) in $x_2$, drone-trip $dg$ follows the same delivery schedule $P$ and returns at $T_R = T_1 + T_P + T_I$ (i.e., waiting during the trip). The departure time of the drone for the next trip $g+1$, denoted as $T_2$, is later than the drone's return time from the previous trip $g$, i.e., $T_2 \geq T_R$ (represented by constraints \eqref{eq:E1-Const12} and \eqref{E2-Const36}). Therefore, further delay in $T_R$ results in postponing the subsequent trips to later times, which can make some subsequent trips infeasible due to violating the time window constraints. In contrast, as a drone can postpone the departure time for the next trip, returning earlier does not affect the feasibility of the subsequent trip with respect to the package release times. Due to these two reasons, it is evident that $C(x_1) \leq C(x_2)$, i.e., solution $x_2$ cannot provide a better solution than $x_1$. Therefore, we conclude that returning immediately after serving all the locations in the sequence dominates waiting during the trip.
\end{proof}

\renewcommand{\theprop}{\ref{prop:Symmetry-Breaking}}

\begin{prop}
    The \textit{BaM} model augmented with symmetry-breaking constraints \eqref{SB-1} - \eqref{SB-4} remains as a valid model, and its feasible set is a proper subset of that of \textit{BaM}.
\end{prop}

\begin{proof}
    For the sake of simplicity, we only establish the proof for the symmetry-breaking constraints \eqref{SB-1} of the LDs in echelon 1. The proof for the other constraints follows a similar procedure. Let $F$ be the set of feasible solutions of the original problem without symmetry-breaking constraints (i.e., \textit{BaM} defined in Eqs. \eqref{eq:obj} - \eqref{E2-Const65}), and $F'$ as the set of feasible solutions after adding symmetry-breaking constraints \eqref{SB-1}. As multiple LDs with the same characteristics exist, there can be multiple optimal solutions that differ only in the LDs used. Let $d_1, d_2 \in \mathcal D$ be the indices of two LDs such that $d_1 < d_2$, and suppose there exist two optimal solutions $x_1, x_2 \in F$ where ($i$) in $x_1$, LD $d_1$ follows a route and delivery pattern while LD $d_2$ is not used, and ($ii$) in $x_2$, LD $d_2$ follows the same route and delivery pattern while $d_1$ is not used. As the two LDs share the same characteristics and have the same routing and delivery patterns, we conclude that $C(x_1) = C(x_2)$, i.e., both solutions result in the same objective value, creating redundant symmetry in the solution space. By enforcing the symmetry-breaking constraints \eqref{SB-1}, any solution $x_2$ where a higher-index LD is used but a lower-index LD is not, becomes infeasible. The remaining feasible solutions are only those where the LDs are used in an increasing order of indices, making $F' \subseteq F$. However, as the symmetry-breaking constraints \eqref{SB-1} do not eliminate distinct optimal solutions that involve different routing, delivery pattern, and the required number of LDs, but only remove the symmetric equivalents (i.e., differ only in the index of the LDs), the optimal solution and optimal objective value of the original problem remain unchanged. Therefore, symmetry-breaking constraints \eqref{SB-1} are valid inequalities that result in a reduction in the search space without affecting the optimal solution, improving the computational efficiency of the \textit{BaM}.  
\end{proof}

\renewcommand{\theprop}{\ref{prop:Last_Hub}}

\begin{prop}
    Let $\tilde{K}_h^{dg}:= \{p\in\mathcal P: \sum_{k\in\mathcal H \cup \{0\}} f_{khp}^{dg}=1\}$ and $\tilde{D}_h^{dg}:= \{p\in \tilde{K}_h^{dg}: w_{ph}^{dg}=1\}$ denote the set of packages carried and delivered to LH $h$, respectively, by LD-trip $dg$. Then, in every integer-feasible solution of the \textit{BaM}, the following two conditions hold: $(i)$ $\sum_{k\in\mathcal H} x_{hk}^{dg}=1 \Rightarrow \tilde{D}_h^{dg} \subset \tilde{K}_h^{dg}$; and $(ii)$ $\tilde{D}_h^{dg} = \tilde{K}_h^{dg} \Rightarrow \sum_{k\in\mathcal H} x_{hk}^{dg}=0,\, x_{hS}^{dg}=1,\, \sum_{k\in\mathcal H} \sum_{p\in\mathcal P} f_{hkp}^{dg}=0$.
\end{prop}

\begin{proof}
    We begin the proof for the aforementioned condition ($i$). Suppose $h$ is not the last visited LH in the delivery sequence of LD-trip $dg$. By definition, this means there exists exactly one outbound arc from LH $h$ to another arbitrary LH, $k\in\mathcal H$ for this LD-trip, i.e., $\sum_{k\in\mathcal H}x_{hk}^{dg}=1$. From the flow conservation constraint~\eqref{NF-2}, the following two properties hold.
    \begin{enumerate}
       \item If package $p$ is delivered to LH $h$, the flow of package $p$ stops at LH $h$: $w_{ph}^{dg}=1 \Rightarrow \sum_{k\in\mathcal H} f_{hkp}^{dg}=0$.
        \item If package $p$ is carried to LH $h$ but not delivered to this LH, this package $p$ flows to another subsequent LH $k$: $ \sum_{k\in\mathcal H \cup \{0\}} f_{khp}^{dg}=1,\,w_{ph}^{dg}=0 \Rightarrow \sum_{k\in\mathcal H} f_{hkp}^{dg}=1$.
    \end{enumerate}
    As the drone departs LH $h$ to another LH $k$ (i.e., $\sum_{k\in\mathcal H}x_{hk}^{dg}=1$), constraints~\eqref{NF-4} enforce that at least one package must flow from LH $h$ to another subsequent LH $k$. Therefore, the following condition holds:
    \begin{equation*}
        \sum_{k\in\mathcal H}x_{hk}^{dg}=1 \Rightarrow \exists \, p\in\mathcal P: \sum_{k\in\mathcal H \cup \{0\}} f_{khp}^{dg}=1,\, \sum_{k\in\mathcal H} f_{hkp}^{dg}=1, \, w_{ph}^{dg}=0
    \end{equation*}
    This, implies that $p \in \tilde{K}_h^{dg}$, and $p \notin \tilde{D}_h^{dg}$. Therefore, we conclude: $\sum_{k\in\mathcal H}x_{hk}^{dg}=1 \Rightarrow \tilde{D}_h^{dg} \subset \tilde{K}_h^{dg}$.
    We now continue the proof for the aforementioned condition ($ii$). Suppose that all the packages carried to LH $h$ are delivered to this LH: $\tilde{D}_h^{dg} = \tilde{K}_h^{dg}$. This means that $w_{ph}^{dg}=1,\, \sum_{k\in\mathcal H}f_{hkp}^{dg}=0, \forall p\in \tilde{K}_h^{dg}$, and therefore, $\sum_{k\in\mathcal H} \sum_{p\in \tilde{K}_h^{dg}}f_{hkp}^{dg}=0$. As no package flows from LH $h$ to another subsequent LH $k\in\mathcal H$, the drone does not depart LH $h$ to any other LHs, as enforced by constraints~\eqref{NF-4}. Therefore, $\sum_{k\in\mathcal H} x_{hk}^{dg}=0$, which implies $\sum_{k\in\mathcal H}\sum_{p\in\mathcal P}f_{hkp}^{dg}=0$, represented by constraints~\eqref{NF-3}. Moreover, according to constraints~\eqref{eq:E1-Const2}, the only feasible outbound arc from LH $h$ is the return to the CD arc, which implies $x_{hS}^{dg}=1$, i.e., LH $h$ is the last visited LH in the delivery sequence of LD-trip $dg$. Therefore, we conclude that: $\tilde{D}_h^{dg} = \tilde{K}_h^{dg} \Rightarrow \sum_{k\in\mathcal H} x_{hk}^{dg}=0,\, x_{hS}^{dg}=1,\, \sum_{k\in\mathcal H} \sum_{p\in\mathcal P} f_{hkp}^{dg}=0$.
\end{proof}

\renewcommand{\theprop}{\ref{prop:Conservative}}

\begin{prop}
    Denoting $PU$ be the package weight unit in discretized demand, any solution obtained by Algorithm~\ref{alg:Heuristic} is a feasible solution to the \textit{HDDNM} problem restricted to the fixed delivery destinations (i.e., CLs) if the following conditions hold: $(i)$ The release times of all the packages are earlier than the departure time of the LDs in the simplified problem variant, i.e., $T_E \geq \max_{i\in\mathcal I} R_i$. $(ii)$ The discretized package weight quantities of the CLs in the simplified problem variant overestimate the actual package weights, i.e., $\hat{P}_i = PU \cdot \lceil \frac{P_i}{PU} \rceil$.
\end{prop}

\begin{proof}
    We show that, under conditions ($i$) and ($ii$), any solution obtained by Algorithm~\ref{alg:Heuristic} can be mapped to a feasible solution of the \textit{HDDNM} problem restricted to the fixed delivery destinations.

    (a) Time window feasibility: According to condition ($i$), no LD departs before all the packages are released at the CD. Therefore, any LD containing any subset of packages does not violate the package release times. This is equivalent to satisfying constraints~\eqref{eq:E1-Const11}, i.e., $dt_0^{d} \geq R_i \cdot \sum_{h\in\mathcal H} w_{ih}^{d}, \forall i\in\mathcal I,\, d\in\mathcal D$. Moreover, MD operations at each LH are scheduled only after all LD deliveries to that LH are completed, as ensured by $AT_h^p \leq \pi_h,\forall h\in\mathcal H$. Therefore, no MD departs an LH before the assigned packages are available at the corresponding LH. This is equivalent to satisfying constraints~\eqref{E2-Const34}, i.e., $pt_h^v \geq \pi_i - M_1 (1 - y_{ih}^v), \forall h\in\mathcal H,\, i\in\mathcal I,\, v\in\mathcal V_h$. Additionally, Algorithm~\ref{alg:Greedy} ensures that all the packages are delivered before their respective delivery due times. Therefore, constraints~\eqref{E2-Const37} are satisfied, i.e., $\delta_i \leq L_i,\forall i\in\mathcal I$.

    (b) PWCC and battery capacity feasibility: Algorithm~\ref{alg:Greedy} ensures that for each MD route $T$, $\sum_{i\in T} P_i \leq P_{max}^{MD}$, and $G_T \leq B_{max}^{MD} - B_{min}^{MD}$. Therefore, the feasibility of each MD route with respect to the PWCC and the battery capacity is trivial. Additionally, $\hat{P}_i = PU \cdot \lceil \frac{P_i}{PU} \rceil$ implies that $\hat{P}_i \geq P_i,\forall i\in\mathcal I$. Therefore, condition ($ii$) ensures that the required amount of commodity in each LH is overestimated, i.e., $s_h = \sum_{i\in\mathcal I, \beta_{ih}=1} \hat{P}_i \geq \sum_{i\in\mathcal I, \beta_{ih}=1} P_i$. As a result, when mapping the discretized commodity amounts to their respective actual packages, the total weight each LD delivers to each LH is greater than the actual package weights: $W_h^p \geq \sum_{i\in\mathcal I} P_i \cdot w_{ih}^d$. This implies that constraints~\eqref{eq:E1-Const7} are satisfied, which ensures the satisfaction of the PWCC of the LDs. Furthermore, as the energy consumption function of the LDs is non-decreasing with respect to package weight, replacing discretized commodity amounts by smaller actual package weights ($P_i \leq \hat{P}_i$) cannot violate battery constraints. Therefore, any route that is feasible with respect to the battery capacity of the LDs in the simplified problem variant remains feasible when using actual package weights. 
    
    We conclude that the solution to the simplified problem variant under conditions ($i$) and ($ii$) satisfies all constraints of the original problem for fixed CLs.
\end{proof}

\renewcommand{\theprop}{\ref{prop:Exact_Heuristic}}

\begin{prop}
    The solution obtained by Algorithm~\ref{alg:Full_Heuristic} is the exact optimal solution to the simplified variant of the \textit{HDDNM} problem if the following conditions hold: $(i)$ $\mathcal{H}_i \cap \mathcal{H}_j = \emptyset, \forall i,j \in \mathcal I$; $(ii)$ $L_i \gg T_E + \sum_{(h,k) \in \mathcal E} T_{hk}^{LD} + |\mathcal H|T_U^{LD} + T_L^{LD} + \sum_{(i,j) \in \mathcal A}T_{ij}^{MD} + P_N^{MD} T_U^{MD} + T_L^{MD}$; $(iii)$ $\max_{\mathcal{S} \subseteq \mathcal{I}, |\mathcal{S}| = P_N^{MD}} \left( \sum_{i \in \mathcal{S}} P_i \right) \leq P_{max}^{MD}$; and $(iv)$ $P_i/PU \in \mathbb{Z}^+, \forall i \in \mathcal I$ (i.e., $PU$ is the common divisor of all package weights).
\end{prop}

\begin{proof}
    To prove the optimality of Algorithm~\ref{alg:Full_Heuristic} under conditions ($i$) to ($iv$), we show that (1) the decoupling of echelons in Step 1 of Algorithm~\ref{alg:Heuristic} does not remove the optimal solution from the search space, and (2) the decisions made in each sequential step of Algorithm~\ref{alg:Full_Heuristic} are optimal solutions.

    (a) Optimality of CL to LH assignment: According to condition ($i$), each CL is feasibly reachable from exactly one LH. This ensures that the decoupling in Step 1 of Algorithm~\ref{alg:Heuristic} (i.e., set partitioning model~\eqref{Heuristic_Set_Partition}) determines the same assignment of CLs to LHs compared to the assignment obtained from the exact solution of the simplified variant of the \textit{HDDNM} problem. 

    (b) Optimality of routing and scheduling MDs in echelon 2: In echelon 2, the goal is to minimize the fleet size of MDs. Condition ($ii$) represents that the delivery due times of CLs are sufficiently large such that any permutation of CLs can be served by a single MD without violating delivery due times. Furthermore, condition ($iii$) represents that the total weight of any set of $P_N^{MD}$ packages does not exceed the PWCC of the MDs. This ensures that PWCC is not a binding restriction for the MDs. Therefore, the greedy assignment of CLs to the MDs in greedy Algorithm~\ref{alg:Greedy} in Appendix~\ref{Appendix_Algorithms} is mathematically guaranteed to determine the minimum required number of MDs. Consequently, these two conditions result in the optimal routing and scheduling of MDs in echelon 2.

    (c) Optimality of routing and scheduling LDs in echelon 1: According to condition ($iv$), the package weight unit $PU$ is chosen such that the discretized package weights $\hat{P}_i$ are exactly equal to the actual package weights $P_i$. This ensures that the discretization of package weights is no longer a conservative overestimation for echelon 1. As the discretized echelon 1 routing and scheduling problem is solved using an exact set covering model~\eqref{Set_Covering} with full enumeration of pattern-paths, the solution obtained by solving model~\eqref{MILP-LD} provides the optimal routing and scheduling of LDs in echelon 1.

    As each decoupled component in \textit{M-H} (CL-to-LH assignment, MD routing and scheduling, and LD routing and scheduling) achieves its optimal solution under conditions $(i)-(iv)$, and each CL has only one feasible LH, we conclude that under conditions $(i)-(iv)$, the combined solution $(\mathcal{R}^{LD}, \mathcal{R}^{MD})$ obtained by Algorithm~\ref{alg:Full_Heuristic} is the optimal solution for the simplified variant of the \textit{HDDNM} problem (discussed in Section~\ref{subsec:Heuristic}).
\end{proof}

\setcounter{equation}{0}
\setcounter{figure}{0}
\setcounter{table}{0}
\setcounter{algorithm}{0}
\section{Detailed Explanation of the Restricted Sets}
\label{Appendix_Restricted_Sets}

For the CLs, $\mathcal I_h^{a} = \{ i\in\mathcal I: P_i \leq P_{max}^{a}, \; 2E_{ih}^{0, a} + E_{ih}^{1, a} P_i \leq B_{max}^{a} - B_{min}^{a}, \; T_L^{a}+T_{hi}^{a}+T_U^{a}+T_L^{LD}+T_{0h}^{LD}+T_U^{LD} \leq L_i - R_i\}$. Here, $P_i \leq P_{max}^{a}$ restricts the set $\mathcal I_h^{a}$ by ensuring that drone type $a$ is capable of carrying the package weight for CL $i$. The condition $2E_{ih}^{0, a} + E_{ih}^{1, a} P_i \leq B_{max}^{a} - B_{min}^{a}$ ensures that a drone of type $a$ can safely (i.e., without exceeding the battery capacity) deliver package $i$ from LH $h$ to CL $i$ and return. Furthermore, $T_L^{a}+T_{hi}^{a}+T_U^{a}+T_L^{LD}+T_{0h}^{LD}+T_U^{LD} \leq L_i - R_i$ ensures that the minimum travel time to pick up package $i$ from the CD and deliver it to the corresponding CL via LH $h$ does not exceed the time window (i.e., $L_i - R_i$). Unlike CLs that have fixed locations and pre-defined packages, the allocation of the AMBs to the ALs is decided endogenously within the optimization routine. Therefore, we adopt an optimistic feasibility check for the ALs, avoiding the exclusion of any AL that can be served under at least one admissible AMB allocation. Accordingly, we define $\mathcal L_h^{a} = \{ l\in\mathcal L: \bar{P}_m \leq P_{max}^{a}, \; 2E_{lh}^{0, a} + E_{lh}^{1, a} \bar{P}_m \leq B_{max}^{a} - B_{min}^{a}, \; T_L^{a}+T_{hl}^{a}+T_U^{a}+T_L^{LD}+T_{0h}^{LD}+T_U^{LD} \leq \max_{m\in\mathcal M} (L_m - R_m)\}$, where $\bar{P}_m = \min_{m\in\mathcal M}\{P_m\}$.

\setcounter{equation}{0}
\setcounter{figure}{0}
\setcounter{table}{0}
\setcounter{algorithm}{0}
\section{Variable Fixing for Infeasible Hub Origins}
\label{Appendix_VF}

Constraints \eqref{VF-HO-CL-1} - \eqref{VF-HO-CL-4}, and \eqref{VF-HO-AL-1} - \eqref{VF-HO-AL-4} represent the variable fixings for the CLs and the ALs, respectively.

\begin{align}
    & r_{ih}^t, s_{ih}^t, \rho_{ih}^t=0, \quad \forall h\in\mathcal H,\, t\in\mathcal T_h,\, i \notin \mathcal I_h^t \label{VF-HO-CL-1}\\
    & \theta_{ij}^t= 0, \quad \forall h\in\mathcal H,\, t\in\mathcal T_h,\, (i,j) \notin \mathcal A_h^t \label{VF-HO-CL-2}\\
    & y_{ih}^{vg'}=0, \quad \forall h\in\mathcal H,\, v\in\mathcal V_h, \, g'\in\mathcal G^{MD},\, i \notin \mathcal I_h^{MD} \label{VF-HO-CL-3}\\
    & z_{ij}^{vg'}, pw_{ij}^{vg'}, q_{ij}^{vg'} =0, \quad \forall h\in\mathcal H,\, v\in\mathcal V_h, \, g'\in\mathcal G^{MD},\, (i,j) \notin \mathcal A_h^{MD} \label{VF-HO-CL-4}\\
    & r_{lh}^t, s_{lh}^t, \rho_{lh}^t=0, \quad \forall h\in\mathcal H,\, t\in\mathcal T_h,\, l \notin \mathcal L_h^t \label{VF-HO-AL-1}\\
    & \theta_{lo}^t= 0, \quad \forall h\in\mathcal H,\, t\in\mathcal T_h,\, (l, o) \notin \mathcal B_h^t \label{VF-HO-AL-2}\\
    & y_{lh}^{vg'}=0, \quad \forall h\in\mathcal H,\, v\in\mathcal V_h, \, g'\in\mathcal G^{MD},\, l \notin \mathcal L_h^{MD} \label{VF-HO-AL-3}\\
    & z_{lo}^{vg'}, pw_{lo}^{vg'}, q_{lo}^{vg'} =0, \quad \forall h\in\mathcal H,\, v\in\mathcal V_h, \, g'\in\mathcal G^{MD},\, (l, o) \notin \mathcal B_h^{MD} \label{VF-HO-AL-4} 
\end{align}

\setcounter{equation}{0}
\setcounter{figure}{0}
\setcounter{table}{0}
\setcounter{algorithm}{0}
\section{Algorithms}
\label{Appendix_Algorithms}

\begin{algorithm}[H]
\caption{Infeasibility of a given path of a CL pair by a given drone type from a given LH.}
\label{alg:Feasibility_Check}
\begin{algorithmic}[1]
\Require Partial $path = (v_1, v_2)$, LH $h$, drone type $a\in \{\mathcal T_h\cup MD\}$
\Ensure Infeasibility status
\Procedure{InfeasibilityCheck}{$path, a, h$}
\State $Pick_{v_i} = R_{v_i} + TL^{LD} + T_{0h}^{LD} + T_U^{LD}, \quad \forall i=1, 2$ \Comment{Earliest possible arrival to LH $h$}
\If{$a=MD$}
    \State{\textbf{Step 1: Time Check}}
    \State $f_{v_0} \gets \max \left\{Pick_{v_1}, Pick_{v_2}\right\} + T_L^{MD}$
    \Comment{Departure time from LH $h$}
    \State $f_{v_1} \gets f_{v_0} + T_{h,v_1}^{MD} + T_U^{MD}$ \Comment{Earliest possible delivery time for $v_1$}
    \State $f_{v_2} \gets f_{v_1} + T_{v_1,v_2}^{MD} + T_U^{MD}$ \Comment{Earliest possible delivery time for $v_2$}
    \State{\textbf{Step 2: Energy Check}}
    \State $G \gets B_{max}^{MD}$ \Comment{Initial fully charged battery}
    \State $G \gets G - E_{h v_1}^{0, MD} - E_{h v_1}^{1,MD}(P_{v_1} + P_{v_2})$ \Comment{Battery energy after first delivery}
    \State $G \gets G - E_{h v_2}^{0, MD} - E_{h v_2}^{1,MD}(P_{v_2})$ \Comment{Battery energy after second delivery}
    \State $G \gets G - E_{v_2 h}^{0, MD}$ \Comment{Battery energy after returning to LH $h$}
    \If{$f_{v_1} \leq L_{v_1}$ and $f_{v_2} \leq L_{v_2}$ and $G \geq B_{min}^{MD}$}
        \State \Return Feasible
    \Else
        \State \Return Infeasible
    \EndIf
\Else
    \State $n_{v_1} \gets T_L^{t} + Pick_{v_1}$ \Comment{Earliest possible pickup time for $v_1$}
    \State $G_{v_1} \gets B_{max}^{t} - 2E_{h v_1}^{0, t} - E_{h v_1}^{1, t} P_{v_1}$
    \Comment{Battery energy after first delivery}
    \State $s_{v_1} \gets 0$
    \Comment{No battery swap for first delivery}
    \State $f_{v_1} \gets n_{v_1} + T_{h,v_1}^t + T_U^t$ \Comment{Earliest possible delivery time for $v_1$}
    \State $G_{v_2} \gets G_{v_1} - 2E_{h v_2}^{0, t} - E_{h v_2}^{1, t} P_{v_2}$
    \If{$G_{v_2} < B_{min}^{t}$} \Comment{Endogenous battery swap for $v_2$}
        \State $s_{v_2} \gets 1$
        \State $G_{v_{2}} \gets B_{max}^{t} - 2E_{h v_2}^{0, t} - E_{h v_2}^{1, t} P_{v_2}$
    \Else
        \State $s_{v_2} \gets 0$
    \EndIf
    \State $n_{v_2} \gets T_L^t + s_{v_2} T_{bat}^t + \max \left\{ Pick_{v_2} , f_{v_1} + T_{v_1,h}^t \right\}$
    \State $f_{v_2} \gets n_{v_2} + T_{h,v_{2}}^t + T_U^t$ \Comment{Earliest possible delivery time for $v_2$}
    \If{$f_{v_1} \leq L_{v_1}$ and $f_{v_2} \leq L_{v_2}$}
        \State \Return Feasible
    \Else
        \State \Return Infeasible
    \EndIf
\EndIf
\EndProcedure
\end{algorithmic}
\end{algorithm}

\begin{algorithm}[H]
\caption{Heuristic separation of capacity cuts.}
\label{alg:Capacity_Cut}
\begin{algorithmic}[1]
\Require Fractional solution $sol$, threshold $\epsilon$
\Ensure Added cut
\Procedure{CapacityCut}{$sol$}
\State Determine the values of $w^*, u^*$ from the solution $sol$
\For{$d\in\mathcal D$}
    \For{$g\in\mathcal G^{LD}$}
        \If{$u^{*dg} > \epsilon$}
            \State $C^{dg} \gets \{p\in\mathcal P: \sum_{h\in\mathcal H}w_{ph}^{*dg} > \epsilon\}$
            \If{$\sum_{p\in C^{dg}} P_p > PW_{max}^{LD}$}
                \If{$\sum_{h\in\mathcal H} \sum_{p\in C} w_{ph}^{*dg} > u^{*dg}\Big(|C^{dg}|-1\Big)$}
                    \State Add cut~\eqref{Cap_Cut} for $C^{dg}$ 
                \EndIf
            \EndIf
        \EndIf
    \EndFor
\EndFor
\EndProcedure
\end{algorithmic}
\end{algorithm}

\begin{algorithm}[H]
\caption{Heuristic separation of last hub cuts.}
\label{alg:LastHubCut}
\begin{algorithmic}[1]
\Require Fractional solution $sol$, threshold $\epsilon$
\Ensure Added cut
\Procedure{LastHubCut}{$sol$}
    \State Determine the values of $w^*, x^*, f^*, u^*, r^*$ from the solution $sol$
    \For{$d \in \mathcal D$}
        \For{$g \in \mathcal G^{LD}$}
            \If{$u^{* dg} > \epsilon$}
                \For{$h \in \mathcal H$}
                    \If{$r_h^{* dg} > \epsilon$}
                        \State $K_h^{dg} \gets \{p \in \mathcal P : \sum_{k \in \mathcal H \cup \{0\}} f_{khp}^{*dg} > \epsilon\}$
                        \If{$K_h^{dg} \neq \emptyset$}
                            \State $D_h^{dg} \gets \{p \in K_h^{dg} : w_{ph}^{*dg} > \epsilon\}$
                            \If{$\Big(\sum_{k\in\mathcal H}x_{hk}^{* dg} > \epsilon\Big)$ $\land$ $\Big((\nexists \, p\in K_h^{dg}:  w_{ph}^{* dg} > \epsilon)$ $\lor$ ($D_h^{dg}=K_h^{dg})\Big)$}
                                \State \textbf{Violation} $\gets$ \textbf{TRUE}
                            \Else
                                \If{$\Big(x_{hS}^{* dg} > \epsilon \Big)$ $\land$ $\Big( (\nexists \, p\in K_h^{dg}:  w_{ph}^{* dg} > \epsilon)$ $\lor$ $(D_h^{dg} \subset K_h^{dg}) \Big)$}
                                    \State \textbf{Violation} $\gets$ \textbf{TRUE}
                                \Else
                                    \State \textbf{Violation} $\gets$ \textbf{FALSE}
                                \EndIf
                            \EndIf
                            \If{\textbf{Violation}}
                                \If{$\sum_{k : (k,h) \in \mathcal A} \sum_{p \in K_h^{dg}} f_{khp}^{*dg} - \sum_{p \in D_h^{dg}} w_{ph}^{* dg} < \sum_{k \in \mathcal H} x_{hk}^{* dg}$}
                                    \State Add cut~\eqref{eq:Last_Hub_Cut} for $K_h^{dg}$ and $D_h^{dg}$
                                \EndIf
                            \EndIf
                        \EndIf
                    \EndIf
                \EndFor
            \EndIf
        \EndFor
    \EndFor
\EndProcedure
\end{algorithmic}
\end{algorithm}

\begin{algorithm}[H]
\caption{Fast greedy heuristic for single hub fleet minimization.}
\label{alg:Greedy}
\begin{algorithmic}[1]
\Require LH $h$, $\beta_{ih}$
\Ensure $ActiveTrips_h$
\Procedure{TripAssignment}{$h, \beta_{ih}$}
    \State $CL_h \gets \{i: \beta_{ih}=1\}$
    \State $SCL_h \gets$ Sort all CLs in $CL_h$ in an ascending order of $L_i$
    \State $Trips_h = \{T_1^h=\emptyset, T_2^h=\emptyset, ..., T_{|CL_h|}^h=\emptyset\}$ \Comment{All trips have no assigned CL}
    \State $pw_T = 0, \forall T\in Trips_h$ 
    \For {$i\in SCL_h$}
        \For {$T \in Trips_h$}
            \If{$(pw_T+ P_i \leq P_{max}^{MD}) \land (|T| < P_N^{MD})$}
                \If{$\Call{Energy}{T, i, h}$ and $\Call{TimeWindow}{T, i, h}$}
                    \State $T \gets \{T, i\}$ and $pw_T \gets pw_T + P_i$
                    \State Break the inner for loop
                \EndIf
            \EndIf
        \EndFor
    \EndFor
    \State \Return{$ActiveTrips_h \gets \{T_j^h: j=1, \cdots, |CL_h|; T_j^h \neq \emptyset\}$}
\EndProcedure
\Procedure{Energy}{$T, i, h$}
    \State $Seq = (v_0, v_1, \cdots, v_m, v_{m+1}):\{h, T, i, h\}$
    \State $P_{m+1} \gets 0$ \Comment{$m+1$ is the dummy sink}
    \State $G \gets \sum_{j=0}^{m} E_{v_j, v_{j+1}}^{0, MD} + E_{v_j, v_{j+1}}^{1, MD} \cdot \sum_{k=j+1}^{m+1} P_{v_k}$
    \If{$G \leq B_{max}^{MD} - B_{min}^{MD}$}
        \State \Return \textbf{TRUE}
    \Else
        \State \Return \textbf{FALSE}
    \EndIf
\EndProcedure
\Procedure{TimeWindow}{$T, i, h$}
    \State $Seq = (v_0, v_1, \cdots, v_m, v_{m+1}):\{h, T, i, h\}$
    \State $Dep \gets T_L^{LD} + T_{0h}^{LD} + T_U^{LD} + T_L^{MD}$ 
    \State \textbf{Status:} $\gets$ \textbf{Feasible}
    \For{$j\in \{1, \cdots, m \}$}
        \State $f_{v_j} \gets Dep + \sum_{k=0}^{j-1} T_{v_k, v_{k+1}} + T_U^{MD} \cdot \sum_{i'=1}^j i'$
        \If{$f_{v_j} > L_{v_j}$}
            \State \textbf{Status} $\gets$ \textbf{Infeasible}
        \EndIf
    \EndFor
    \If{\textbf{Status} $=$ \textbf{Feasible}}
        \State \Return \textbf{TRUE}
    \Else
        \State \Return \textbf{FALSE}
    \EndIf
\EndProcedure
\end{algorithmic}
\end{algorithm}

\begin{algorithm}[H]
\caption{Latest feasible departure time for a single MD-trip.}
\label{alg:Late_Departure}
\begin{algorithmic}[1]
\Require Trip $T_h = (i_1, i_2, \ldots, i_k)$ from LH $h$, due times $L_i$, travel times $T_{ij}^{MD}$
\Ensure $\text{Dep}(T_h)$: latest feasible departure time for the trip

\Procedure{LatestDeparture}{$T_h$}

    \State $t \gets T_L^{MD}$ 
    \State $latest\_dep \gets +\infty$
    \State $prev \gets h$ 

    \For{$i$ in $T_h$}
        \State $t \gets t + T_{prev,i}^{MD} + T_U^{MD}$ 
        \State $latest\_dep \gets \min(latest\_dep,\; L_i - t)$
        \State $prev \gets i$
    \EndFor

    \State \Return $latest\_dep$

\EndProcedure
\end{algorithmic}
\end{algorithm}

\setcounter{equation}{0}
\setcounter{figure}{0}
\setcounter{table}{0}
\setcounter{algorithm}{0}
\section{Exact Separation for Capacity Cuts}
\label{Appendix_Separation}

Let $w_{ph}^{*dg}$ and $u^{*dg}$ denote the optimal values obtained for each LD-trip $dg$ in a given fractional solution. Using these values, we determine whether there exists a subset of packages that: ($i$) exceeds the PWCC of the LDs, and ($ii$) violates the capacity cut~\eqref{Cap_Cut} for the LD-trip $dg$. This separation procedure requires solving model~\eqref{Exact_Cap_Cut_Separation} to find the maximally violated subset, if one exists.\vspace{-10pt}

\begin{subequations}
\begin{align}
    \max \quad & V^{dg}(\mathbf{z}) = \sum_{p \in \mathcal P} \left( \sum_{h \in \mathcal H} w_{ph}^{*dg} \right) z_p - u^{*dg} \cdot \left( \sum_{p \in \mathcal P} z_p - 1 \right) \label{Ex_CC_Obj}\\
    \text{s.t.} \quad & \sum_{p \in \mathcal P} P_p \cdot z_p > PW_{max}^{LD}  \label{Ex_CC_Cnst}\\
    & z_p \in \{0, 1\}
\end{align}\label{Exact_Cap_Cut_Separation}
\end{subequations}

In model~\eqref{Exact_Cap_Cut_Separation}, $z_p$ is a binary variable representing whether package $p$ is included in the potential subset. The objective function~\eqref{Ex_CC_Obj} maximizes the violation of the capacity cut~\eqref{Cap_Cut}, while constraints~\eqref{Ex_CC_Cnst} ensure the violation of the PWCC of the LDs. A violated capacity cut exists, if and only if, the optimal objective value of the model~\eqref{Exact_Cap_Cut_Separation} is non-negative, i.e., $V^{*dg}(\mathbf{z}) > 0$. Then the maximally violated subset of the packages is determined as $C^{dg} = \{p: z_p^* = 1 \}$. We then update $\tilde{\mathcal C}$ by adding all detected maximally violated subsets, and add the corresponding capacity cut~\eqref{Cap_Cut} to the \textit{BaM}.

\end{document}